\documentclass[12pt,letterpaper]{amsart}

\usepackage{fancyhdr}

\fancypagestyle{firststyle}{
    \fancyhf{}

    \fancyfoot[L]{%
        \footnotesize\\
        Date: Jan. 2026.
        The author is partially supported by NSFC (12171299).
    }
}

\usepackage{xcolor}
\definecolor{blue}{RGB}{0,0,255}
\definecolor{red}{RGB}{255,0,0}
\definecolor{green}{RGB}{0,255,0}
\usepackage{amsthm}
\usepackage{thmtools}
\declaretheoremstyle[
  headfont=\color{blue}\bfseries,
  bodyfont=\normalfont,
  headpunct={},
  postheadspace=1em
]{springer}

\usepackage{geometry}
\usepackage{graphicx}
\usepackage{amsmath,amsfonts,amssymb}
\usepackage{hyperref}
\hypersetup{
    colorlinks=true,
    linkcolor=red,
    citecolor=blue,
    urlcolor=black
}

\allowdisplaybreaks[4]

\newtheorem{prop}{Proposition}[section]
\newtheorem{theo}{Theorem}[section]
\newtheorem{lemm}{Lemma}[section]
\newtheorem{coro}{Corollary}[section]
\newtheorem{rema}{Remark}[section]
\newtheorem{exam}{Example}[section]
\newtheorem{defi}{Definition}[section]

\newcommand{\beq}{\begin{equation}}
\newcommand{\eeq}{\end{equation}}
\newcommand{\beqs}{\begin{eqnarray*}}
\newcommand{\eeqs}{\end{eqnarray*}}
\newcommand{\beqn}{\begin{eqnarray}}
\newcommand{\eeqn}{\end{eqnarray}}
\newcommand{\beqa}{\begin{array}}
\newcommand{\eeqa}{\end{array}}

\newtheorem*{remark*}{Remark}

\numberwithin{equation}{section}

\def\begeq{\begin{equation}}
\def\endeq{\end{equation}}

\newcommand{\R}{\mathbb{R}}

\def\S{\mathbb S}

\def\R{\mathbb R}
\def\S{\mathbb S}

\numberwithin{equation}{section}

\begin{document}

\title{\bfseries Uniqueness of $L_p$ Minkowski problem in the supercritical range}
\author{Shi-Zhong Du}

\date{}

\maketitle

\thispagestyle{firststyle}

\begin{abstract}
\noindent \setlength{\leftskip}{-1cm} 
\setlength{\rightskip}{-1cm} 
The uniqueness of the $L_p$-Minkowski problem has been a long standing problem in convex geometry, which draws back earlier in 1974 the paper \cite{F74} (Mathematika, {\bf21}, 1974) by Firey, and later developed by Lutwak, Yang, Zhang \cite{LYZ04} (Trans. Am. Math. Soc., {\bf356}, 2004) et al. In the groundbreaking paper by Brendle-Choi-Daskalopoulos \cite{BCD17} (Acta Math, {\bf219}, 2017), 
a full uniqueness result was shown for the subcritical exponents $p\in(-n-1,1]$. In the supercritical range, 
the uniqueness problem becomes much more complicated, even on the planar case $n=1$. One of the famous results was 
shown by Andrews in \cite{An03} (J. Amer. Math. Soc., {\bf16}, 2003), where he established that the uniqueness 
holds in the range $p\in(-7,-2)$ and fails to hold for the other supercritical exponents $p\in(-\infty,-7)$. 
In this paper, we study the same uniqueness problem in the full supercritical range $p\in(-2n-5,-n-1)$ for all 
higher dimensional cases $n\geq2$. We will prove that for $p\in(-2n-5,-n-1)$, the unique strongly symmetric 
solution is given by the unit sphere ${\S}^n$. The uniqueness range $(-2n-5,-n-1)$ is optimal due to our recent 
preprint \cite{DLL24} (arXiv: 2104.07426), where non-spherical strongly symmetric solutions have been constructed 
for all $p\in(-\infty,-2n-5)$. When considering general solutions which may not be symmetric, the uniqueness set 
$\Gamma$ of $p$ for which the uniqueness holds, is shown to be both relatively open and closed in the full interval $(-2n-5,-n-1)$.\\

\end{abstract}


\section{Introduction}

\noindent Given a convex body $K\subset{\mathbb{R}}^{n+1}$ containing the origin,
we denote it by $K\in{\mathcal{K}}_0$.
The support function $h: {\mathbb{S}}^{n}\to{\mathbb{R}}$ of $K$ is defined by $h(x)=\max\big\{z\cdot x\ |\ z\in\Omega\big\}$ for all $x\in{\mathbb{S}}^{n}$. Letting $[A_{ij}]=[h_{ij}+h\delta_{ij}]$, we denote $[A^{ij}]$ as the inverse matrix of $[A_{ij}]$ and $[U^{ij}]$ as the cofactor matrix of $[A_{ij}]$, where the Hessian matrix of $\nabla^2_{ij}h$ is acting on an orthonormal frame $\{e_i\}_{i=1}^n$ of ${\mathbb{S}}^{n}$ and $\delta_{ij}$ denotes the Kronecker symbol.

As usual, we may also use the notation $r(y)=\sup\big\{\lambda>0\ |\ \lambda y\in{K}\big\}, \ \forall y\in{\mathbb{R}}^{n+1}$ to be the radial distance function of $\partial {K}$. The radial position function is defined by
    $$
     Z(y)= r(y)y=h(x)x+\nabla h(x), \ \ \forall x\in{\mathbb{S}}^{n},
    $$
where $x$ is the unit outer normal of $\partial K$ at the point $Z(y)\in\partial K$. We will thus call $G: \ y\to x$ the Gauss mapping and $G^{-1}:\ x\to y$ its inverse Gauss mapping. Hence ${K}$ can be recovered from $h$ by $\partial{K}=\{h(x)x+\nabla h(x)\ |\ x\in{\mathbb{S}}^{n}\}$. The surface measure of $\partial{K}$ is given by $dS=\det(\nabla^2h+hI)d\sigma$, where $d\sigma$ is the surface measure of the sphere ${\mathbb{S}}^{n}$. The classical Minkowski problem looks for a convex body such that its surface measure
matches a given Radon measure $d\mu$ on ${\mathbb{S}}^{n}$.
For a convex body ${K}\in{\mathcal{K}}_0$, Lutwak \cite{LYZ18} introduced the $L_p$ surface measure $dS_p= h^{1-p}dS$ on $\partial{K}$ and proposed  the following $L_p$ Minkowski problem.
Given a finite Radon measure $\mu$ on ${\mathbb{S}}^n$,
find a convex body ${K}\in{\mathcal{K}}_0$ such that $dS_{p}({K},\cdot)=d\mu$. In the smooth category, when the measure $d\mu$ is given by $d\mu=fd\sigma$ for a function $f$,
the $L_p$ Minkowski problem can be formulated by
  \begin{equation}\label{e1.1}
    \det(\nabla^2h+hI)=fh^{p-1}, \ \ \forall x\in{\mathbb{S}}^{n},
  \end{equation}
a fully nonlinear equation of Monge-Amp\`{e}re type. The $L_p$-Minkowski problem has been extensively studied in recent years \cite{Al38,An00,An03,BS23,BBCY19,BLYZ13,BCD17,CCL21,CLZ19,CW06,HLX15,HLW16,IM23,JLZ16,JLW15,Li19,LYZ04,LW13,Mi23,Sa21}.
According to the different behaviors exhibited, the nonlinear exponents $p\in{\mathbb{R}}$ are divided into several ranges which include
   $$
    \begin{cases}
      \mbox{ Fast Growth Case} & \leftrightarrow p\in(n+1,\infty),\\
      \mbox{ Moderate Growth Case} & \leftrightarrow p=n+1,\\
      \mbox{ Slow Growth Case} & \leftrightarrow p\in(1,n+1),\\
      \mbox{ Linear Case} & \leftrightarrow p=1,\\
      \mbox{ Sublinear Case} & \leftrightarrow p\in(0,1),\\
      \mbox{ Logarithmic Case} & \leftrightarrow p=0,\\
      \mbox{ Subcritical Case} & \leftrightarrow p\in(-n-1,0),\\
      \mbox{ Critical Case} & \leftrightarrow p=-n-1,\\
      \mbox{ Supercritical Case} & \leftrightarrow p\in(-\infty,-n-1).
    \end{cases}
   $$
There are rich phenomena regarding the existence and uniqueness of the solutions. The readers may refer to the cornerstone work of Chou-Wang \cite{CW06} for the existence and uniqueness theorems with respect to different exponents $p$. The uniqueness problem is the main topic of this paper.

For general $f$, a uniqueness result of \eqref{e1.1} when $p>n+1$ has been shown in \cite{CW06} for all dimensions, using the maximum principle. Certainly, uniqueness cannot be expected for $p=n+1$ due to the homogeneity of the equation. When $p<n+1$, the uniqueness problem is much more subtle since lack of a maximum principle. Only some partial results were known in the past. Chow has shown in \cite{Ch85} for all $n\geq1$ and $p=1-n$, the uniqueness of \eqref{e1.1} holds true for constant function $f=1$. Later, Dohmen-Giga \cite{DG94} and Gage \cite{G93} have extended the result to $n=1, p=0$ for some symmetric function $f$. Lutwak \cite{Lu93} demonstrated uniqueness for $n\geq1, p>1$ and special symmetric $f$, leveraging the invertibility of the linearized equation. Subsequently, Andrews showed uniqueness for $n=2, p=0$ and arbitrary positive function $f$. Contrary to the results in \cite{DG94,G93}, Yagisita \cite{Ya06} demonstrated a surprising non-uniqueness result for $n=1, p=0$ and non-symmetric function $f$. When $p\in(-n-1,-n-1+\sigma), 0<\sigma\ll1$, a counterexample to uniqueness was also obtained in \cite{CW06}. More recently, Jian-Lu-Wang \cite{JLW15} proved that for $p\in(-n-1,0)$, there exists at least a smooth positive function $f$ such that \eqref{e1.1} admits two different solutions. Meanwhile, Chen-Huang-Li-Liu \cite{CHLL20} established partial uniqueness for origin-symmetric convex bodies with $p\in(p_0,1), p_0\in(0,1)$, using the $L_p$-Brunn-Minkowski inequality.

When $f=1$, equation \eqref{e1.1} is reduced to the special case
    \begin{equation}\label{e1.2}
     \det(\nabla^2h+hI)=h^{p-1}, \ \ \forall x\in{\mathbb{S}}^{n}.
    \end{equation}
Andrews-Guan-Ni \cite{AGN16} showed uniqueness of \eqref{e1.2} under some symmetricity for $p\in[0,1]$. Soon later, the uniqueness in case of $1-n\leq p<1/(n+1)$ was solved in \cite{Ch17} as part of K. Choi's Ph.D. thesis. In the celebrating paper by Brendle-Choi-Daskalopoulos \cite{BCD17}, a full uniqueness of \eqref{e1.2} was established when $p\in(-n-1,1]$. More recently, Ivaki-Milman \cite{IM23} gave a new and simple proof to Brendle-Choi-Daskalopoulos's theorem by using the Alexandrov-Fenchel inequality.

In the deeply negative range $p\leq-n-1$, the situations are much more complicated. As is well known \cite{Ca72}, for $n\geq1$ and $p=-n-1$, all ellipsoids with the volume of the unit ball are solutions of \eqref{e1.2}. The uniqueness fails in this case. In the planar case where $n=1$ for \eqref{e1.2}, Andrews showed in \cite{An03} that the uniqueness holds when $p\in(-7,-2)$ and fails to hold for $p\in(-\infty,-7)$. On the higher dimensions $n\geq2$, non-uniqueness part of Andrews's theorem has been extended to all $p\in(-\infty,-2n-5)$ in \cite{DLL24}.

Due to the importance and difficulty of the topic of uniqueness, it attracts much more attention. This is the main purpose for us to discuss the problem. From here, we denote ${\mathcal{C}}_{n, p}$ to be the set of solutions of  \eqref{e1.2} and denote $\mathcal{C}'_{n, p}=\mathcal{C}_{n, p}\setminus\{{\mathbb{S}}^n\}$ to be the set of nontrivial solutions. For simplicity, in the following we also denote
\begin{eqnarray*}
 \sigma_K&=&\det(\nabla^2h_K+h_KI),\\
 d\sigma&=&d\mathcal H^n\big|_{{\mathbb{S}}^n},\\
 d\sigma_K&=&\sigma_Kd\sigma
 \end{eqnarray*}
for a given solution $K$ of \eqref{e1.2} or denote
  \begin{eqnarray*}
   dV_K&=&h_Kd\sigma_K\\
   &=&h_K\det(\nabla^2h_K+h_KI) d\sigma,
  \end{eqnarray*}
where $d{\mathcal{H}}^n$ is $n$-dimensional Hausdorff measure. At the beginning, we will prove the following criterion condition to the uniqueness of \eqref{e1.2} in the supercritical range $p\in(-n-1-\sigma,-n-1)$.

\begin{theo}\label{t1.1}
 For any $\sigma>0$, suppose that there exists some constant $\Theta^*_{\sigma}>0$ such that
\begin{equation}\label{e1.3}
 \Theta(p)= \inf_{h_K\in \mathcal{C}'_{n,p}}Q(K) \ge \Theta^*_{\sigma}>0
\end{equation}
holds for all $p\in(-n-1-\sigma,-n-1)$ and the combined quotient
  $$
   Q(K)=\lambda_3(K)q(K), \ \ q(K)=\int_{{\S }^n}|Z_K^\perp|^2dV_K{\Big/}\int_{{\S }^n}|\nabla h_K|^2dV_K.
  $$
Then the solution of \eqref{e1.2} is unique for
$p\in (-n-1-\sigma',-n-1)$,
where
  $$
   \sigma'\equiv\min \left\{\sigma,\frac{n\sqrt{1+{\Theta^*_\sigma}^2}-n}{2}\right\},
  $$
$Z_K^\perp$ is the orthonormal component of $Z_K$ defined in \eqref{e2.10}, and $\lambda_3(K) $ is the third eigenvalue of the linearized eigenvalue problem \eqref{e2.1}.
\end{theo}

\noindent The proof of Theorem \ref{t1.1} was reduced to showing a $p$-Blaschke-Santal\'{o} functional inequality with remainder. Then, assuming the Combined Inequality \eqref{e1.3} for some solution $K$ of \eqref{e1.2}, we will show a reverse $p$-Blaschke-Santal\'{o} functional inequality. After comparing two inequalities with opposite orientations, we derive the uniqueness of \eqref{e1.2} in Section 2, assuming the Combined Inequality \eqref{e1.3} holds.

 We say a convex body $K$ is strongly symmetric if its support function $h$ is group invariant with respect to the symmetry group ${\mathcal{S}}={\mathcal{S}}(T)\subset O(n+1)$ of some regular polytope $T$. For this case, we will denote it by $K\in{\mathcal{K}}(T)$ for simplicity. Moreover, we will also denote ${\mathcal{C}}_{n,p}(T)={\mathcal{C}}_{n,p}\cap{\mathcal{K}}(T)$ and ${\mathcal{C}}'_{n,p}(T)={\mathcal{C}}'_{n,p}\cap{\mathcal{K}}(T)$ to be strongly symmetric solution sets for convenience. For some detailed discussion about the regular polytopes and their symmetry group, the readers may refer to the preprints \cite{DLL24,DWZ24}. Our second main result shows the uniqueness of strongly symmetric solution in the full supercritical range of $p\in(-2n-5,-n-1)$.

\begin{theo}\label{t1.2}
Letting dimension $n\geq2$ and $p\in(-2n-5,-n-1)$, the unique strongly symmetric solution of \eqref{e1.2} is given by $K={\mathbb{S}}^n$.
\end{theo}

\noindent It is remarkable that in the preprint \cite{DLL24}, we have constructed non-spherical, strongly symmetric solutions $K\not={\S}^n$ for all $p\in(-\infty,-2n-5)$. Hence, the range of $p$ in the Theorem \ref{t1.2} cannot be improved to higher supercritical range $p\in(-\infty,-2n-5)$. In the proof of Theorem \ref{t1.2}, a first step is to show the uniqueness holds in a slightly supercritical range $p\in(-n-1-\sigma_n,-n-1)$ for some universal constant $\sigma_n>0$, using the criterion Theorem \ref{t1.1}. In order verifying the validity of the Combined Inequality \eqref{e1.3}, a crucial part is to show various orthonormal lemmas (Lemmas \ref{l3.1}, \ref{l3.6}-\ref{l3.8}) for strongly symmetric convex body in Section 3.1. Let us first sum them in the following theorem for the convenience of the readers.

\begin{theo}\label{t1.3}
  Letting $n\geq2$ and $T$ be a regular polytope on ${\mathbb{R}}^{n+1}$, for each strongly symmetric convex body $K\in{\mathcal{K}}(T)$, the functions $x_\alpha, \alpha=1,2,\ldots,n+1$ are orthonormal to each other in the sense of
    \begin{equation}\label{e1.4}
      \int_{{\mathbb{S}}^n}\frac{x_\alpha x_\beta}{h_K^2}dV_K=\frac{\delta_{\alpha\beta}}{n+1}\int_{{\mathbb{S}}^n}h_K^{-2}dV_K, \ \ \forall \alpha,\beta=1,2,\ldots,n+1
    \end{equation}
 for $dV_K= h_K\det(\nabla^2h_K+h_KI)d\sigma, d\sigma=d\mathcal H^n|_{{\mathbb{S}}^n}$.
\end{theo}

\noindent The Orthonormal Theorem \ref{t1.3} plays an essential role in the proof of Theorem \ref{t1.2} as the Kazdan-Warner condition in the classical Minkowski problem, where it is known that
     \begin{equation}
       \int_{{\mathbb{S}}^n}\frac{x_\alpha}{h_K}dV_K=0, \ \ \forall\alpha=1,2,\ldots,n+1
     \end{equation}
is always true for convex body $K$ containing the origin, without assuming the symmetricity. However, \eqref{e1.4} is in generally false for general convex bodies, even in the planar case $n=1$. Fortunately, when the symmetricity is imposed, we will show that \eqref{e1.4} does hold for $K\in{\mathcal{K}}(T)$.

Turning to considering the general solutions of \eqref{e1.2} which may not be symmetric, we will explore the topological structure of the uniqueness set defined by  $\Gamma\equiv\big\{p\in(-2n-5,-n-1)|\ {\mathcal{C}}_{n,p}=\{{\S}^n\}\big\}$. We have the following result.

\begin{theo}\label{t1.4}
  For dimensions $n\geq2$, the uniqueness set $\Gamma$ is relatively open and closed in $(-2n-5,-n-1)$. That's to say, either $\Gamma=(-2n-5,-n-1)$ or $\Gamma=\emptyset$.
\end{theo}

\noindent A key ingredient in the proof of Theorem \ref{t1.4} is to prove the following {\it a-priori} estimate for solutions of \eqref{e1.2} in the supercritical range $p<-n-1$.

\begin{theo}\label{t1.5}
 For dimension $n\geq2$ and compact subset ${\mathcal{P}}\Subset(-\infty,-n-1)$, there exists a positive constant $C_{n}({\mathcal{P}})$ depending only on $n$ and ${\mathcal{P}}$, such that
  \begin{equation}\label{e1.6}
    1/C_{n}({\mathcal{P}})\leq h(x)\leq C_{n}({\mathcal{P}}), \ \ \forall x\in{\mathbb{S}}^n
  \end{equation}
 holds for all solutions $h\in{\mathcal{C}}_{n,p}$ of \eqref{e1.2} with respect to $p\in{\mathcal{P}}$.
\end{theo}

\noindent It is remarkable that a similar {\it a-priori} estimate has been obtained by B\"{o}r\"{o}czky-Saroglou \cite{BS23} in the range $p\in[0,1)$. To show the {\it a-priori} estimate \eqref{e1.6} in the supercritical range, we need to establish some ellipsoid lemmas in Subsection 5.1. Motivated from the blowup argument carried out in \cite{BS23}, we complete the proof of Theorem \ref{t1.5} in Subsection 5.2. As an application of {\it a-priori} estimate Theorem \ref{t1.5}, we will obtain the openness of the uniqueness set $\Gamma$ in Section 5. At the second step, we reduce the closedness of the uniqueness set to a Kernel Property of the linearized equation of \eqref{e1.2}. One of the crucial parts is to prove a perturbation lemma for full $L_p$-Minkowski problem with general $f$. Then, with the help of an Infinitesimal Generator Lemma and Prolongation Formula shown in Section 4, we are able to verify the validity of the Kernel Property of \eqref{e1.2}, and thus complete the proof of Theorem \ref{t1.4}. So, the full uniqueness Theorem \ref{t1.2} for strongly symmetric solution is a direct consequence of parallel version of Theorem \ref{t1.4} (Theorem \ref{t6.1}) together with the uniqueness in a slightly supercritical range.

 As an application of the main Theorem \ref{t1.2}, the Gauss curvature flow
   \begin{equation}\label{e1.7}
    \begin{cases}
     \displaystyle\frac{\partial X(t,x)}{\partial t}=-G^\alpha(t,x)x, & \forall t>0, x\in{\mathbb{S}}^n,\\
     X(0,x)=X_0(x), & \forall x\in{\mathbb{S}}^n
    \end{cases}
   \end{equation}
is studied for Gauss curvature $G$ of the convex body ${K}_t$ determined by radial position function $X(t,x)=h(t,x)x+\nabla h(t,x)$, where $h(t,\cdot)$ is the support function of ${K}_t$. As usual, supposing a solution ${K}_t$ of \eqref{e1.7} blows up at finite time $T\in(0,\infty)$, we will call the blowup to be
  \begin{equation}
   \begin{cases}
     \mbox{type I}, & \mbox{ if } 0<{\mathcal{L}}\leq{\mathcal{U}}<\infty,\\
     \mbox{type II}, & \mbox{ if }  {\mathcal{U}}=\infty,\\
     \mbox{type III}, & \mbox{ if } {\mathcal{L}}=0,
   \end{cases}
  \end{equation}
where
    \begin{eqnarray*}
      {\mathcal{L}}&=&\liminf_{t\to T^-}(T-t)^{-\beta}\inf_{x\in{\mathbb{S}}^n}h(t,x),\\ {\mathcal{U}}&=&\limsup_{t\to T^-}(T-t)^{-\beta}\sup_{x\in{\mathbb{S}}^n}h(t,x).
    \end{eqnarray*}
Then, as a consequence of Theorem \ref{t1.2}, we reach the following convergence result of strongly symmetric solutions for supercritical exponent $\alpha\in\left(\frac{1}{2n+6},\frac{1}{n+2}\right)$.

\begin{coro}\label{c1.1}
   For supercritical exponent $\alpha\in\left(\frac{1}{2n+6},\frac{1}{n+2}\right)$, consider strongly symmetric solutions of the Gauss curvature flow \eqref{e1.7}. Then, for type I singular time $T$, ${K}_t$ always converges to a round point as $t\to T^-$.
\end{coro}

\medskip

\section{$p$-Blaschke-Santal\'{o} functional inequality with remainder}

\noindent To begin with, let us consider the $p$-Blaschke-Santal\'{o} functional defined by
   $$
    {\mathcal{F}}_p(K)=Vol(K)\left(\inf_{\xi\in K}\int_{{\S}^n}(h_K(x)-\xi\cdot x)^{p}d\sigma\right)^{-\frac{n+1}{p}}
   $$
for all $K\in{\mathcal{K}}_0$. When $p=-n-1$, the well-known Blaschke-Santal\'{o} inequality asserts that the functional ${\mathcal{F}}_{-n-1}(\cdot)$ is bounded from above by the best constant ${\mathcal{F}}_{-n-1}({\S}^n)$. Moreover, the unique maximizers are given exactly by ellipsoids. When $p<-n-1$, we know that the functional is not bounded from above.

First, we will establish an inequality of the $p$-Blaschke-Santal\'{o} functional with remainder in Subsection 2.1. Then, after showing that the validity of the Combined Inequality \eqref{e1.3} implies a reverse $p$-Blaschke-Santal\'{o} functional inequality, we complete the proof of Theorem \ref{t1.1} in Subsection 2.2.

\subsection{The $1$\textsuperscript{st} and $2$\textsuperscript{nd} variations of the $p$-Blaschke-Santal\'{o} functional}

Given a convex body $K$, we denote $[U^{ij}_K]$ to be the cofactor matrix of $[A_{ij}]=[\nabla^2_{ij}h_K+h_K\delta_{ij}]$.
Aleksandrov \cite{Al38} studied the eigenvalue problem
  \begin{equation}\label{e2.1}
 {\begin{split}
   -U^{ij}_K(\nabla^2_{ij}\phi+\phi\delta_{ij})
     & =\lambda h_K^{-1}\det(\nabla^2h_K+h_KI)\phi\\
     &=\lambda h_K^{p-2} \phi\ \ \ \text{if $h_K$ is a solution}
     \end{split}}
  \end{equation}
and obtained the first and second eigenvalues of \eqref{e2.1}.

\begin{prop}\label{p2.1} [Aleksandrov]
 Considering the eigenvalue problem \eqref{e2.1}, the eigenfunctions corresponding to the zero eigenvalue $\lambda=0$ are given by $\phi=\langle a,x\rangle$ for some vector $a\in{\mathbb{R}}^{n+1}$. Moreover, $\lambda=-n$ is the unique negative eigenvalue of \eqref{e2.1}, whose eigenfunctions can only be $\kappa h_K$ for each $\kappa\neq0$.
\end{prop}

As in Proposition \ref{p2.1}, the first and second eigenvalues of the linearized eigen-problem \eqref{e2.1} are given by $-n$ and $0$ respectively. While, the corresponding eigenfunctions are given by $\kappa h_K$ and $x_i, i=1,2,\ldots,n+1$. Under the inner product $\langle f,g\rangle=\int_{{\S}^n}h_K^{-1}fgd\sigma_K$, given any function $\phi$ on the unit sphere, we define $\phi^{(1)}, \phi^{(2)}, \phi^{\perp}=\phi^{(1,2)}$ to be the components of $\phi$ which are orthonormal with respect to the eigen-subspaces
 $$
 \begin{cases}
  E^{(1)}=span\{h_K\},\\
  E^{(2)}=span\{x_i,i=1,2,\ldots,n+1\},\\
  E^{(1,2)}=E^{(1)}\cup E^{(2)}
 \end{cases}
 $$
respectively. In this subsection, we will firstly show the equivalence of the solutions of \eqref{e1.2} with the critical points of the functional ${\mathcal{F}}_p$.

\begin{prop}\label{p2.2}
  Supposing $p<-n-1$, if $K$ is a solution of the full equation
   \begin{equation}\label{e2.2}
    \det(\nabla^2h_K+h_KI)=\frac{V(K)}{\int_{{\S}^n}h^p_Kd\sigma}h_K^{p-1}, \ \ \forall x\in{\S}^n,
   \end{equation}
  then it is a critical point of the functional ${\mathcal{F}}_p$.  Conversely, if $K$ is a critical point of the functional ${\mathcal{F}}_p$, then $K-\xi_K$ is a solution of \eqref{e2.2}, where $\xi_K\in K$ is the unique Blaschke-Santal\'{o} center of $K$ determined by
      \begin{equation}\label{e2.3}
         \int_{{\S}^n}x(h_K(x)-\xi\cdot x)^{p-1}d\sigma=0, \ \ \xi=\xi_K.
      \end{equation}
\end{prop}

\begin{proof}
  For $\xi\in K$, the function $H(\xi)=\int_{{\S}^n}(h_K(x)-\xi\cdot x)^pd\sigma$ is a convex function in $\xi$. Moreover, $H(\xi)$ tends to infinity as $\xi$ approaches the boundary $\partial K$ of $K$, provided $p<-n-1<-n$. Therefore, one knows that the function $H(\cdot)$ attains infimum at a unique minimizer $\xi=\xi_K\in K$ satisfying \eqref{e2.3} in the interior of $K$. Moreover, letting $K^{(2)}$ be the convex body determined by the support function $h_K^{(2)}$ orthonormal to $E^{(2)}$, we have $K^{(2)}=K-\xi_K$ and
     \begin{equation}
      {\mathcal{F}}_p(K)=Vol(K^{(2)})\left(\int_{{\S}^n}h_{K^{(2)}}^pd\sigma\right)^{-\frac{n+1}{p}}.
     \end{equation}
  Given a solution $K$ (with support function $h_K$) of \eqref{e2.2}, it is inferred from Kazdan-Warner identity that $K=K^{(2)}$. For any variation $K_\varepsilon$ (with support function $h_\varepsilon$) satisfying
      $$
       h_0=h_K, \ \ \frac{d}{d\varepsilon}\Big|_{\varepsilon=0}h^{(2)}_\varepsilon=\phi,
      $$
  we have $\phi\perp E^{(2)}$ and
   \begin{eqnarray}\nonumber\label{e2.5}
     \frac{d}{d\varepsilon}\bigg|_{\varepsilon=0}{\mathcal{F}}_p(K_\varepsilon)&=&\frac{d}{d\varepsilon}\bigg|_{\varepsilon=0}{\mathcal{F}}_p(K^{(2)}_\varepsilon)\\
      &=&\left(\int_{{\S}^n}h_{K^{(2)}}^pd\sigma\right)^{-\frac{n+1}{p}}\int_{{\S}^n}\phi\det(\nabla^2h_{K^{(2)}}+h_{K^{(2)}}I)d\sigma\\ \nonumber
      &&-V(K^{(2)})\left(\int_{{\S}^n}h_{K^{(2)}}^pd\sigma\right)^{-\frac{p+n+1}{p}}\int_{{\S}^n}h_{K^{(2)}}^{p-1}\phi d\sigma=0
   \end{eqnarray}
  since $K=K^{(2)}$. That's to say, $K$ is a critical point of the functional ${\mathcal{F}}_p$. Conversely, if $K$ is a critical point of the functional ${\mathcal{F}}_p$, a same calculation as in \eqref{e2.5} shows that $K^{(2)}=K-\xi_K$ is a solution of \eqref{e2.2}. The conclusion has been shown.
\end{proof}

Second proposition illustrates the relation between the second variation of the functional with the third eigenvalue of the linearized eigen-problem \eqref{e2.1}.

\begin{prop}\label{p2.3}
  Letting $K$ be a solution of \eqref{e2.2} and given any variation $K_\varepsilon$ satisfying
     $$
      h_0=h_K, \ \ \frac{d}{d\varepsilon}\Big|_{\varepsilon=0}h_\varepsilon^{(2)}=\phi, \ \ \frac{d^2}{d\varepsilon^2}\Big|_{\varepsilon=0}h_\varepsilon^{(2)}=\varphi,
     $$
  we have
    \begin{eqnarray}\nonumber\label{e2.6}
     \frac{d^2}{d\varepsilon^2}\bigg|_{\varepsilon=0}{\mathcal{F}}_p(K_\varepsilon)&=&V^{-\frac{n+1}{p}}(K)\int_{{\S}^n}U^{ij}_K(\phi^\perp_{ij}+\phi^\perp\delta_{ij})\phi^\perp\\
     &&-(p-1)V^{-\frac{n+1}{p}}(K)\int_{{\S}^n}h_K^{p-2}{\phi^{\perp}}^2
    \end{eqnarray}
  holds for orthonormal component $\phi^\perp=\phi^{(1,2)}=\phi^{(1)}$ of $\phi$ with respect to $E^{(1,2)}$.
\end{prop}

\begin{proof}
   Taking twice differentiation on the function ${\mathcal{F}}_p(K_\varepsilon)={\mathcal{F}}_p(K^{(2)}_\varepsilon)$ with respect to $\varepsilon$, we obtain that
   \begin{eqnarray}\nonumber\label{e2.7}
     \frac{d^2}{d\varepsilon^2}\bigg|_{\varepsilon=0}{\mathcal{F}}_p(K_\varepsilon)&=&V^{-\frac{n+1}{p}}(K)\int_{{\S}^n}U^{ij}_K(\phi_{ij}+\phi\delta_{ij})\phi\\
      &&+(p-n-1)V^{-\frac{p+n+1}{p}}(K)\left(\int_{{\S}^n}h_K^{p-1}\phi\right)^2\\ \nonumber
      &&-(p-1)V^{-\frac{n+1}{p}}(K)\int_{{\S}^n}h_K^{p-2}\phi^2,
   \end{eqnarray}
where we have used the first variation of ${\mathcal{F}}_p(K_\varepsilon)$ at $\varepsilon=0$ vanishes for solutions $K$ of \eqref{e2.2}. Noting that $E^{(1)}, E^{(2)}, \phi^\perp$ are orthonormal to each other under both the inner product $\langle f,g\rangle=\int_{{\S}^n}h_K^{-1}fgd\sigma_K$ and the inner product $\langle f,g\rangle_*=\int_{{\S}^n}U^{ij}_K(f_{ij}+f\delta_{ij})g$ thanks to the operator in the eigenvalue problem \eqref{e2.1} being self-adjoint, \eqref{e2.6} follows. Actually, noting that $\phi\perp E^{(2)}$ by definition, there holds $\phi=\ell_0h_K+\phi^{\perp}$ for some constant $\ell_0$. Therefore,
   \begin{eqnarray*}
    \int_{{\mathbb{S}}^n}U^{ij}_K(\phi_{ij}+\phi\delta_{ij})\phi&=&\ell_0^2\int_{{\mathbb{S}}^n}U^{ij}_K(h_{K,ij}+h_K\delta_{ij})h_K+\int_{{\mathbb{S}}^n}U^{ij}_K(\phi^{\perp}_{ij}+\phi^{\perp}\delta_{ij})\phi^{\perp}\\
     &=&n\ell_0^2\int_{{\mathbb{S}}^n}h_K\det(\nabla^2h_K+h_KI)+\int_{{\mathbb{S}}^n}U^{ij}_K(\phi^{\perp}_{ij}+\phi^{\perp}\delta_{ij})\phi^{\perp}\\
     &=&n\ell_0^2\int_{{\mathbb{S}}^n}h_K^{p}+\int_{{\mathbb{S}}^n}U^{ij}_K(\phi^{\perp}_{ij}+\phi^{\perp}\delta_{ij})\phi^{\perp}.
   \end{eqnarray*}
Similarly, there hold
   $$
    \int_{{\mathbb{S}}^n}h_K^{p-1}\phi=\ell_0\int_{{\mathbb{S}}^n}h_K^{p-1}
   $$
and
   $$
    \int_{{\mathbb{S}}^n}h_K^{p-2}\phi^2=\ell_0^2\int_{{\mathbb{S}}^n}h_K^p+\int_{{\mathbb{S}}^n}h_K^{p-2}{\phi^{\perp}}^2.
   $$
The proof of \eqref{e2.6} has been done.
\end{proof}

Combining Proposition \ref{p2.3} with Proposition \ref{p2.2} yields the following $p$-Blaschke-Santal\'{o} functional inequality with remainder.

\begin{coro}\label{c2.1}
  Letting $K$ be a solution of \eqref{e2.2} or equivalently a critical point of the $p$-Blaschke-Santal\'{o} functional, there holds the following inequality
    \begin{eqnarray}\nonumber\label{e2.8}
      V^{\frac{n+1}{p}}(K)[{\mathcal{F}}_p(K_\varepsilon)-{\mathcal{F}}_p(K)]&\leq&-\frac{\lambda_3(K)+p-1}{2}\varepsilon^2\int_{{\S}^n}h_K^{p-2}{\phi^\perp}^2\\
      &&+o(\varepsilon^2), \ \
       \forall\varepsilon
    \end{eqnarray}
  with remainder, where $o(\varepsilon^2)$ is an higher-order infinitesimal with respect to $\varepsilon^2$, and $\lambda_3(K)$ is the third eigenvalue of the linearized equation \eqref{e2.1} at the point $K$.
\end{coro}

It is remarkable that variant Sobolev inequalities with remainder can be found in \cite{BE91,CW98,FZ22-1,FN19,GW99}. The readers may also refer to the works \cite{Bo10,BH10,Fu15,FZ22-2,Iv15} for stability of the other inequalities in convex geometry.

\begin{proof}
  \eqref{e2.8} is a direct consequence of Propositions \ref{p2.2} and \ref{p2.3} by using the inequality
     $$
      \int_{{\S}^n}U^{ij}_K({\phi^\perp}_{ij}+{\phi^\perp}\delta_{ij}){\phi^\perp}\leq-\lambda_3(K)\int_{{\S}^n}h_K^{p-2}{\phi^\perp}^2
     $$
  for the linearized eigen-problem \eqref{e2.1}, and then applying the Taylor's expansion.
\end{proof}

\subsection{Combined Inequality \eqref{e1.3} implies reverse $p$-Blaschke-Santal\'{o} functional inequality}

In this subsection, we will show the following reverse $p$-Blaschke-Santal\'{o} functional inequality assuming the Combined Inequality \eqref{e1.3}.

\begin{prop}\label{p2.4}
 Assuming the Combined Inequality \eqref{e1.3} holds for some positive constant $\Theta^*_\sigma$, there will be a reverse $p$-Blaschke-Santal\'{o} functional inequality
    \begin{equation}\label{e2.9}
    \begin{cases}
     \displaystyle \sum_{l=1}^{n+1}V^{\frac{n+1}{p}}(K)[{\mathcal{F}}_p(K^{(l)}_\varepsilon)-{\mathcal{F}}_p(K)]\geq-\vartheta_\sigma(n,p,K)\varepsilon^2\int_{{\S}^n}|Z_K^\perp|^2dV_K-o(\varepsilon^2),\\[5pt]
     \displaystyle   \vartheta_\sigma(n,p,K):=\frac{(n+p+1)(p+1)\lambda_3(K)}{2n\Theta^*_\sigma}+\frac{p-1}{2},
    \end{cases}
    \end{equation}
 where $K^{(l)}_\varepsilon$ is a variant of $K$ with respect to
  $$
    \phi=\phi^{(l)}=h_KZ^{(l)}_K, \ \ Z_K=(Z_K^{(1)},\ldots, Z^{(n+1)}_K)
  $$
  for $l=1,2,\ldots,n+1$, and
   \begin{equation}\label{e2.10}
   Z_K^\perp =
    Z_K -\int_{{\S }^n}Z_K  dV_K{\bigg/}\int_{{\S }^n}dV_K-\sum_{\alpha=1}^{n+1}\frac{X_\alpha}{h_K}\int_{{\S }^n}\frac{X_\alpha}{h_K}Z_K dV_K{\bigg/}\int_{{\S }^n}\bigg|\frac{X_\alpha}{h_K}\bigg|^2dV_K
   \end{equation}
  holds for an orthonormal frame $\{X_\alpha\}_{\alpha=1}^{n+1}$ of $E^{(2)}$.
\end{prop}

\noindent The proof of Proposition \ref{p2.4} will be decomposed into several lemmas.

\begin{lemm}\label{l2.1}
 Setting $\phi=h_K\varphi$, we have
   \begin{eqnarray}\nonumber\label{e2.11}
     \int_{{\S}^n}U^{ij}_K({\phi^\perp}_{ij}+{\phi^\perp}\delta_{ij}){\phi^\perp}&=&-\int_{{\S }^n}U^{ij}_Kh_K^2\varphi_i\varphi_j+n\int_{{\S }^n}\varphi^2dV_K\\
     &&-n\left(\int_{{\S }^n}\varphi dV_K\right)^2{\bigg/}\int_{{\S }^n}dV_K.
   \end{eqnarray}
\end{lemm}

\begin{proof}
 Noting that $\phi^\perp, \phi^{(1)}, \phi^{(2)}$ are orthonormal with respect to the inner product $\langle f,g\rangle_*=\int_{{\S}^n}U^{ij}_K(f_{ij}+f\delta_{ij})g$, there holds
   \begin{eqnarray*}
     \int_{{\S}^n}U^{ij}_K({\phi^\perp}_{ij}+{\phi^\perp}\delta_{ij}){\phi^\perp}&=&\int_{{\S}^n}U^{ij}_K({\phi}_{ij}+{\phi}\delta_{ij}){\phi}-n\left(\int_{{\S}^n}\phi d\sigma_K\right)^2\Big{/}\int_{{\S}^n}dV_K\\
     &=&-{\mathcal{X}}^1(\varphi)+{\mathcal{X}}^2(\varphi)
   \end{eqnarray*}
after using the Gauss-Weingarten's relation $x_{ij}+x\delta_{ij}=0$, where
 $${\begin{split}
   {\mathcal{X}}^1(\varphi)
    &=\int_{{\S }^n}U^{ij}_K(h_K\varphi)_i(h_K\varphi)_j,\\
   {\mathcal{X}}^2(\varphi)
   &=\int_{{\S }^n}U^{ij}_K\delta_{ij}h_K^2\varphi^2
       -n\left(\int_{{\S }^n}h_K\sigma_K \varphi\right)^2{\bigg/}\int_{{\S }^n}h_K\sigma_K.
\end{split}} $$
Noting that by divergence-free property $U^{ij}_{K,i}=0, \forall j$, we have
   \begin{eqnarray*}
     {\mathcal{X}}^1(\varphi)&=&\int_{{\S }^n}U^{ij}_Kh_K^2\varphi_i\varphi_j+\int_{{\S }^n}U^{ij}_Kh_{K,i}h_{K,j}\varphi^2+2\int_{{\S }^n}U^{ij}_Kh_Kh_{K,i}\varphi\varphi_j\\
     &=&\int_{{\S }^n}U^{ij}_Kh_K^2\varphi_i\varphi_j-\int_{{\S }^n}U^{ij}_Kh_Kh_{K,ij}\varphi^2\\
     &=&\int_{{\S }^n}U^{ij}_Kh_K^2\varphi_i\varphi_j-n\int_{{\S }^n}\varphi^2dV_K+\int_{{\S }^n}U^{ij}_K\delta_{ij}h_K^2\varphi^2.
   \end{eqnarray*}
This completes the proof of \eqref{e2.11}.
\end{proof}

Now, let $\{e_i\}_{i=1}^n$ be a local orthonormal frame of ${\S }^n$ such that the matrix $A=[\nabla^2h_K+h_KI]$ equals to a diagonal matrix $diag(\lambda)$ for $\lambda=(\lambda_1,\ldots,\lambda_n)$ at a given point $x_0\in\S^n$.
Let $\{E_l\}^{n+1}_{l=1}$ be the canonical orthonormal basis of ${\mathbb{R}}^{n+1}$ and denote $Z_K^{(l)} = \langle Z_K,E_l\rangle$. By Gauss-Weingarten relation $d_{e_i}Z_K=A_{ij}e_j$, we have $\partial_iZ_K^{(l)}=\langle d_{e_i}Z_K,E_l\rangle=\lambda_i\langle e_i,E_l\rangle$. Taking $\varphi=\varphi_l=Z_K^{(l)}$ for $l=1,2,\ldots,n+1$,
we have
  \begin{equation}
  {\begin{split}
   U^{ij}_K\partial_i\varphi_l\partial_j\varphi_l
    & =U^{ij}_K\partial_iZ^{(l)}_K\partial_jZ^{(l)}_K \\
    & =U^{ii}\lambda_i^2\langle e_i,E_l\rangle^2.
    \end{split}}
  \end{equation}
Summing on $l$, we obtain from cofactor identity $U^{ik}_K A_{kj}=\det(A)\delta_{ij}$ that
 \begin{eqnarray}\nonumber\label{e2.13}
    \sum_{l=1}^{n+1}\sum_{i,j=1}^{n}U^{ij}_K\partial_i \varphi_l\partial_j\varphi_l&=&\sum_{i=1}^nU^{ii}_K\lambda_i^2\\
    &=&\det(A)(\Delta h_K+nh_K).
 \end{eqnarray}
Substituting into Lemma \ref{l2.1}, we obtain a second lemma.

\begin{lemm}\label{l2.2}
  Taking $\phi=\phi^{(l)}=h_KZ_K^{(l)}$ for $l=1,2,\ldots,n+1$, we have
   \begin{eqnarray}\nonumber\label{e2.14}
    {\mathcal{X}}^3(K)&:=&\sum_{\phi=\phi^{(l)}}\int_{{\S}^n}U^{ij}_K({\phi^\perp}_{ij}+{\phi^\perp}\delta_{ij}){\phi^\perp}\\
    &=&(n+p+1)\int_{{\S}^n}|\nabla h_K|^2dV_K-n\left(\int_{{\S }^n}Z_K dV_K\right)^2{\bigg/}\int_{{\S }^n}dV_K.
   \end{eqnarray}
\end{lemm}

\begin{proof}
   Substituting \eqref{e2.13} into \eqref{e2.11}, we derive the desired identity \eqref{e2.14} by using
     $$
      \int_{{\S}^n}h_K^2\det(A)(\Delta h_K+nh_K)=-(p+1)\int_{{\S}^n}|\nabla h_K|^2dV_K+n\int_{{\S}^n}h_K^2dV_K
     $$
  and $\int_{{\S}^n}|Z_K|^2dV_K=\int_{{\S}^n}h_K^2dV_K+\int_{{\S}^n}|\nabla^2h_K|^2dV_K$.
\end{proof}

The third lemma handle the term containing $\int_{{\S}^n}Z_KdV_K$.

\begin{lemm}\label{l2.3}
 For any convex body $K$ containing the origin, there holds
    \begin{eqnarray}\nonumber\label{e2.15}
     {\int_{{\S }^n}}Z_K dV_K&=&(n+2)\int_Ky dy\\
      &=&\frac{n+p+1}{n}{\int_{{\S }^n}}\nabla h_KdV_K.
    \end{eqnarray}
If $K$ is with origin-centroid, then the integral equals to zero.
\end{lemm}

\begin{proof} For any given vector $w\in{\mathbb{R}}^{n+1}$, by the divergence theorem, we have
  \begin{eqnarray*}
    {\int_{{\S }^n}}\langle Z_K,w\rangle dV_K&=&\int_{\partial K}\langle p,w\rangle\langle p,\nu(p)\rangle d{\mathcal{H}}^n(p)\\
    &=&\int_Kdiv_{{\mathbb{R}}^{n+1}}\big[\langle y,w\rangle y\big]dy\\
    &=&(n+2)\int_K\langle y,w\rangle dy.
  \end{eqnarray*}
 To show the second identity in \eqref{e2.15}, we use the divergence theorem again to deduce
  \begin{eqnarray*}
   n{\int_{{\S }^n}}Z_KdV_K&=&n{\int_{{\S }^n}}h_K^p(h_K(x)x+\nabla h_K)\\
    &=&-{\int_{{\S }^n}}h_K^{p+1}\Delta x+n{\int_{{\S }^n}}h_K^p\nabla h_K\\
    &=&(n+p+1){\int_{{\S }^n}}h_K^p\nabla h_K.
  \end{eqnarray*}
The conclusion of the Lemma has been shown.
\end{proof}

\begin{proof}[Proof of Proposition \ref{p2.4}]
  Now, \eqref{e2.9} of Proposition \ref{p2.4} is a corollary of Lemma \ref{l2.2}, Lemma \ref{l2.3}, Young's inequality and Taylor's expansion. Actually, after submitting Lemmas \ref{l2.2} and \ref{l2.3} into the formula \eqref{e2.6}, we conclude that
     \begin{eqnarray*}
      \sum_{l=1}^{n+1}\frac{d^2}{d\varepsilon^2}\Big|_{\varepsilon=0}{\mathcal{F}}_p(K^{(l)}_\varepsilon)&=&(n+p+1)V(K)^{-\frac{n+1}{p}}\int_{{\mathbb{S}}^n}|\nabla h_K|^2dV_K\\
      &&-\frac{(n+p+1)^2}{n}V(K)^{-\frac{n+1}{p}}\frac{\left(\int_{{\mathbb{S}}^n}\nabla h_KdV_K\right)^2}{\int_{{\mathbb{S}}^n} dV_K}\\
      &&-(p-1)V(K)^{-\frac{n+1}{p}}\int_{{\mathbb{S}}^n}|Z_K^{\perp}|^2dV_K\\
      &\geq&-\frac{(n+p+1)(p-1)}{n}V(K)^{-\frac{n+1}{p}}\int_{{\mathbb{S}}^n}|\nabla h_K|^2dV_K\\
     &&-(p-1)V(K)^{-\frac{n+1}{p}}\int_{{\mathbb{S}}^n}|Z_K^{\perp}|^2dV_K\\
      &\geq&-\left[\frac{(n+p+1)(p-1)\lambda_3(K)}{n\Theta^*_\sigma}+(p-1)\right]\int_{{\mathbb{S}}^n}|Z_K^\perp|^2dV_K
     \end{eqnarray*}
  holds by Young's inequality
    $$
     \left(\int_{{\mathbb{S}}^n}\nabla h_KdV_K\right)^2\Bigg/\int_{{\mathbb{S}}^n} dV_K\leq\int_{{\mathbb{S}}^n}|\nabla h_K|^2dV_K,
    $$
  and assuming the Combined Inequality \eqref{e1.3} is true. Thus, the conclusion of Proposition \ref{p2.4} follows from the Taylor's expansion.
\end{proof}

\begin{proof}[Proof of Theorem \ref{t1.1}]
  Noting that for $p\in(-n-1-\sigma',-n-1)$, if one takes $\phi=h_KZ_K^{(l)}$ for $l=1,2,\ldots,n+1$ in Corollary \ref{c2.1}, it follows from \eqref{e2.8} that
     \begin{equation}
      \sum_{l=1}^{n+1}V^{\frac{n+1}{p}}(K)[{\mathcal{F}}_p(K^{(l)}_\varepsilon)-{\mathcal{F}}_p(K)]<-\vartheta_\sigma(n,p,K)\varepsilon^2\int_{{\S}^n}|Z_K^\perp|^2dV_K+o(\varepsilon^2).
     \end{equation}
  After comparing with the reverse $p$-Blaschke-Santal\'{o} functional inequality \eqref{e2.9}, we conclude that $Z^\perp_K\equiv0$ holds for all $x\in{\S}^n$.

\begin{lemm}\label{l2.4}
  For dimension $n\geq2$, the convex body $K$ containing the origin inside satisfies $Z_K^\perp=0$ for all $x\in{\S }^n$ is equivalent to $K$ is an origin-centering ellipsoid $E(\mu)$ for some $\mu$ after rotation.
\end{lemm}

\begin{proof} Supposing that $Z^\perp_K=0$, there must be
   \begin{equation}\label{e2.17}
     Z_K=h_K(x) x+\nabla h_K=A+Bx/h_K, \ \ \forall x\in{\S }^n
   \end{equation}
for a vector $A\in{\mathbb{R}}^{n+1}$ and a matrix $B\in GL(n+1)$.
Taking inner product on \eqref{e2.17} with $x$, we have $h_K=\langle A,x\rangle+x^TBx/h_K$, where the fact that support function $h$ is positive has been used for convex body containing the origin inside. Solving the quadratic equation yields that
  \begin{eqnarray*}
   h_K&=&\frac{\langle A,x\rangle\pm\sqrt{|\langle A,x\rangle|^2+4x^TBx}}{2}\\
   &=&\langle A/2,x\rangle\pm\sqrt{x^TCx}
  \end{eqnarray*}
for $C= AA^T/4+B$. Noting that if one takes some $x\in{\mathbb{S}}^n$ such that $\langle A,x\rangle=0$, the minus symbol in the above expression can be exclude since $K$ contains the origin inside. Henceforth, after rotation, the convex body corresponding to the support function $h_K$ given above is exactly an ellipsoid $E(\mu,z_0)$ for some $\mu$ and $z_0=A/2$. Nevertheless, substituting
   $$
    \nabla_ih_K=\langle A/2,x_{,i}\rangle+\frac{x^TCx_{,i}}{\sqrt{x^TCx}}, \ \ x_{,i}=\nabla_ix=e_i
   $$
into $Z=h_K(x)x+\nabla_ih_Kx_{,i}=A+Bx/h_K$, it yields that
   \begin{eqnarray}\nonumber\label{e2.18}
    0&=&\langle A/2,e_i\rangle-\frac{x^TCe_i}{\sqrt{x^TCx}}+\frac{x^TBe_i}{h_K}\\
    &=&\langle A/2,e_i\rangle-\langle A/2,x\rangle\frac{x^TCe_i+\langle A/2,e_i\rangle}{h_K\sqrt{x^TCx}}, \ \ \forall i=1,2,\ldots,n.
   \end{eqnarray}
Now, we claim that $A$ must zero. Otherwise, evaluating \eqref{e2.18} at a point $x\perp A, x\in{\S}^n$, we deduce that $\langle A, e_i\rangle=0$ holds for all $i=1,2,\ldots,n$. So, the only possibility is that $A=0$ since
  $$
   A\perp x, \ \ A\perp e_i, \ \ \forall i=1,2,\cdots,n
  $$
for a base $\{x, e_1,\cdots, e_n\}$ of ${\mathbb{R}}^{n+1}$. This is contradicting with our assumption $A\not=0$. Our claim holds true. Conversely, all ellipsoids $K=E(\mu)$ satisfy the identity $Z^\perp_K=0$ automatically. In fact, the support function of $K$ is given by $h_K(x)=\sqrt{x^TCx}$ for diagonal matrix $C=diag(\mu_1,\mu_2,\cdots,\mu_{n+1})$ with diagonal entries $\mu_1,\mu_2,\cdots,\mu_{n+1}$. After submitting $h_K$ into $Z_K=h_K(x)x+\nabla h_K$, we reach the conclusion $Z_K^\perp=0$. The proof of the lemma was done.
\end{proof}

Combining Lemma \ref{l2.4} with Lemma \ref{l5.2} for $p<-n-1$, we reach the conclusion that $K={\S}^n$. The proof of Theorem \ref{t1.1} has been done.
\end{proof}

\medskip

\section{Strongly symmetric solutions: Theorem \ref{t1.2}}

\noindent To show that the strongly symmetric solution of \eqref{e1.2} is unique in the full supercritical range $p\in(-2n-5,-n-1)$, a first step is to prove that a uniqueness result in a slightly supercritical range $p\in(-n-1-\sigma,-n-1)$ holds for some $\sigma>0$. By Theorem \ref{t1.1}, to achieve this goal, one needs to estimate both the third eigenvalue $\lambda_3(K)$ and the quotient $q(K)=\frac{\int_{{\S}^n}|Z_K^\perp|^2dV_K}{\int_{{\S}^n}|\nabla h_K|^2dV_K}$ in the Combined Inequality \eqref{e1.3} from below.

It is remarkable that the quotient $q(K)$ may not be bounded from below by a positive constant when $K$ is close to a non-spherical ellipsoid. Fortunately, when considering strongly symmetric solutions $K$, there does hold a desired lower bound. The key point is to show various Orthonormal Lemmas for group ${\mathcal{S}}(T)$ invariant functions in the next subsection. Before our discussion, let us recall different regular polytopes $T=T^{(k)}\subset{\R}^{n+1}$ listed in
    \begin{equation}
     \begin{cases}
      \mbox{When } n=1, \mbox{there exist } k\mbox{-regular polytopes for each } k\in{\mathbb{N}}, k\geq3.\\
      \mbox{When } n=2, \mbox{there exist only } k\mbox{-regular polytopes for } k=4, 6, 8, 12, 20.\\
      \mbox{When } n=3, \mbox{there exist only } k\mbox{-regular polytopes for } k=5, 8, 16, 24, 120, 600.\\
      \mbox{When } n\geq4, \mbox{there exist only } k\mbox{-regular polytopes for } k=n+2, 2n+2, 2^{n+1}.
     \end{cases}
    \end{equation}
with $k$ vertices on ${\R}^{n+1}$. The readers may refer to \cite{DWZ24} for some detailed argument. Since the identity \eqref{e1.4} is invariant under rotations on ${\mathbb{R}}^{n+1}$, the proof of Theorem \ref{t1.3} will be carried out firstly for regular polytope $T$ in good position defined as below.

\begin{defi}\label{d3.1}
  Let $e_i, i=1,2,\ldots, n+1$ denote axial unit vectors with a $1$ in the $i$-th position and $0$s elsewhere. Letting $T=\vartheta^{(k)}$ be a regular polytope in ${\mathbb{R}}^{n+1}$ with $k$ vertices, we call it to be in a good position supposing that

(1) when $k=2n+2$, if $T=\cup_{i=1}^{n+1}\{\pm e_i\}\subset{\mathbb{R}}^{n+1}$.

(2) When $k=2^{n+1}$, if $T=\{(\pm 1,\pm 1,\ldots,\pm1)\subset{\mathbb{R}}^{n+1}$.

(3) When $k=n+2$, we assume that $v_e=(1,1,\ldots,1)\in{\mathbb{R}}^{n+1}$ is one vertex of $T$. Moreover, $n+1$ different sides $L_i, i=1,2,\ldots,n+1$ starting from a same end point $v_e$ intersect with $n+1$ axial semi-lines $X_i=\{\lambda e_i, \lambda>0\}$ at $X^e_i=\lambda_ee_i$ for $i=1,2,\ldots,n+1$ respectively, where $\lambda_e\in{\mathbb{R}}^+$ is a positive constant determined below.

(4) When $n=2, k=12, 20$, the good position of $T$ is defined in the proof of Lemma \ref{l3.6}. When $n=3, k=24,120,600$, the good position of $T$ will be defined in the proof of Lemmas \ref{l3.7}-\ref{l3.8}.
\end{defi}

\begin{rema}
  Let us first remark the existence of such regular simplex $T=\vartheta^{(n+2)}$ in good position. Starting with a small regular simplex $T_r$ with one vertex $v_r=rv_e$ and centroid $O_r$ for $r>0$, we assume that $v_r, O_r, O$ are collinear. Using a hyperplane $H$ perpendicular to the line $v_rO_rO$ to cut $T_r$, we derive a $n$-dimensional regular simplex $C$. Using another hyperplane $H'$ perpendicular to $v_rO_rO$ to cut the first hexagram ${\mathbb{Q}}_1$ of ${\mathbb{R}}^{n+1}$, we note that the section is also a $n$-dimensional regular simplex $C'$. Move $H'$ and $v_r$ such that $C$ and $C'$ coincide, by rotation if necessary. Extending the sides $L_i, i=1,2,\ldots,n+1$ from vertex $v_r$ of $T_r$, we can form a $n+1$-dimensional regular simplex $T'_r$ with origin centroid. After scaling $v_r$ to $v_e$, we derive a desired $n+1$ regular simplex in good position. {\hfill $\square$}
\end{rema}

\subsection{Orthonormal lemmas for strongly symmetric functions}

\noindent In this subsection, we will prove various orthonormal lemmas (or Theorem \ref{t1.3}) for strongly symmetric functions, which include \eqref{e1.4} as a special case. The first one concerns $n$-simplex, $n$-cross-polytope and $n$-hypercube for all dimensions $n\geq1$.

\begin{lemm}\label{l3.1}
For dimension $n\geq1$, let ${\mathcal{S}}={\mathcal{S}}(T)\subset O(n+1)$ be the symmetry group of a regular polytope $T=\vartheta^{(k)}$ in good position for
    \begin{equation}
      \begin{cases}
       k=n+2, & n-\mbox{simplex},\\
       k=2n+2, & n-\mbox{cross polytope},\\
       k=2^{n+1}, & n-\mbox{hypercube}.
      \end{cases}
    \end{equation}
Then, there holds the orthonormal relation
    \begin{eqnarray}\nonumber\label{e3.3} {\mathcal{W}}_{\alpha}(\zeta)&=&\int_{{\mathbb{S}}^n}x_\alpha\zeta(x)d\sigma=0,\ \ \forall\alpha=1,2,\ldots,n+1,\\
   {\mathcal{W}}_{\alpha\beta}(\zeta)&=&\int_{{\mathbb{S}}^n}x_\alpha x_\beta \zeta(x)d\sigma\\ \nonumber
   &=&\frac{\delta_{\alpha\beta}}{n+1}\int_{{\mathbb{S}}^n}\zeta(x)d\sigma, \ \  \forall\alpha,\beta=1,2,\ldots,n+1
    \end{eqnarray}
for all group ${\mathcal{S}}$ invariant functions $\zeta$.
\end{lemm}

\begin{proof}
  For $T=\vartheta^{(2^{n+1})}$ which is a hypercube in good position, \eqref{e3.3} follows because the mappings defined by
   $$
    \begin{cases}
    \phi_i(x_i)=-x_i, \ \ \phi_i(x_j)=x_j, & \forall i, \forall j\not=i,\\
    \phi_{i,j}(x_i)=x_j, \ \ \phi_{i,j}(x_j)=x_i, \ \ \phi_{i,j}(x_k)=x_k, & \forall i,j, \forall k\not=i,j
    \end{cases}
   $$
 belong to ${\mathcal{S}}={\mathcal{S}}(T)$. Therefore, for any $\alpha\not=\beta$,
    \begin{eqnarray*} {\mathcal{W}}_{\alpha}(\zeta)&=&\frac{1}{2}\int_{{\mathbb{S}}^n}[x_\alpha+(-x_\alpha)]\zeta(x)d\sigma=0,\\
    {\mathcal{W}}_{\alpha\beta}(\zeta)&=&\frac{1}{2}\int_{{\mathbb{S}}^n}[x_\alpha+(-x_\alpha)]x_\beta\zeta(x)d\sigma=0.
   \end{eqnarray*}
 Moreover, given $\alpha$, there holds
    \begin{eqnarray*}
     {\mathcal{W}}_{\alpha\alpha}(\zeta)&=&\frac{1}{n+1}\int_{{\mathbb{S}}^n}|x|^2\zeta(x)d\sigma\\
     &=&\frac{1}{n+1}\int_{{\mathbb{S}}^n}\zeta(x)d\sigma
    \end{eqnarray*}
 by applying the mappings $\phi_{\alpha,\beta}$ repeatedly for $\beta\not=\alpha$. The validity of \eqref{e3.3} in case of $T=\vartheta^{(2n+2)}$ follows directly from the conclusion of $T=\vartheta^{(2^{n+1})}$ by duality, and the observation that the dual regular polytopes adhere a same symmetry group. Now, it remains to consider the case $T=\vartheta^{(n+2)}$. We prove by Mathematical Induction. When $n=1$, noting that the anti-clockwise rotation with angle $2j\pi/3$ belongs to the symmetry group ${\mathcal{S}}(T)$ for each integer $j$, there holds
  \begin{eqnarray*}
     \int^{2\pi}_0\sin\theta\zeta(\theta)&=&\int^{2\pi}_0\sin\left(\theta\pm\frac{2\pi}{3}\right)\zeta(\theta)\\
      &=&\int^{2\pi}_0\left(-\frac{1}{2}\sin\theta\pm\frac{\sqrt{3}}{2}\cos\theta\right)\zeta(\theta).
   \end{eqnarray*}
 Hence, we conclude from the above system
 that
   $$
    {\mathcal{W}}_{1}(\zeta)={\mathcal{W}}_{2}(\zeta)=0.
   $$
 Similarly, it follows from
   \begin{eqnarray*}
     \int^{2\pi}_0\sin2\theta\zeta(\theta)&=&\int^{2\pi}_0\sin\left(2\theta\pm\frac{2\pi}{3}\right)\zeta(\theta)\\
     &=&\int^{2\pi}_0\left(-\frac{1}{2}\sin2\theta\pm\frac{\sqrt{3}}{2}\cos2\theta\right)\zeta(\theta)
   \end{eqnarray*}
 that ${\mathcal{W}}_{12}(\zeta)=0$. Finally, the identity
  \begin{equation}\label{e3.4}
    {\mathcal{W}}_{11}(\zeta)={\mathcal{W}}_{22}(\zeta)=\frac{1}{2}\int^{2\pi}_0\zeta(\theta)
  \end{equation}
 is a consequence of the following lemma, which is valid for all dimensions.

 \begin{lemm}\label{l3.2}
  For dimension $n\geq1$, let ${\mathcal{S}}={\mathcal{S}}(T)\subset O(n+1)$ be the symmetry group of a regular simplex $T=\vartheta^{(n+2)}$ in good position. Then, the circle mappings $\phi_0^j, j\in{\mathbb{Z}}$ for $\phi_0\in O(n+1)$ defined by
    $$
     \phi_0(x_i)=x_{i+1}, \ \ \phi_0(x_{n+1})=x_1, \ \ \forall i=1,2,\ldots,n
    $$
  belong to ${\mathcal{S}}$.
 \end{lemm}

\begin{proof}[Proof of Lemma \ref{l3.2}.]
 We need only to show $T$ is invariant under $\phi_0$. Actually, let $L_i, i=1,2,\ldots,n+1$ be the sides of $T$ connecting the vertex $v_e$. Noting that $\phi_0$ maps side $L_i$ to $L_{i+1}$ for $i=1,2,\ldots,n$ and maps $L_{n+1}$ to $L_1$, it follows from high symmetry and rigidity of $T$ that $\phi_0$ leaves $T$ invariant. Actually, the symmetry group of a regular simplex is isomorphism to the permutation group of the vertices. The lemma is proven.
\end{proof}

Henceforth, \eqref{e3.4} follows from Lemma \ref{l3.2} for $n=1, x_1=\cos\theta, x_2=\sin\theta$ and
   $$
    {\mathcal{W}}_{11}(\zeta)+{\mathcal{W}}_{22}(\zeta)=\int^{2\pi}_0\zeta(\theta).
   $$
 The conclusion of Lemma \ref{l3.1} for dimension $n=1$ has been shown. To show the higher dimensional case $n\geq2$, we need a first reduction lemma.

\begin{lemm}\label{l3.3}
 Letting ${\mathcal{S}}^{(n)}\subset O(n)$ be the symmetry group of a $(n-1)-$dimensional regular simplex $T^{(n-1)}=\vartheta^{(n+1)}\subset{\mathbb{R}}^n$, assume the relation \eqref{e3.3} holds for group ${\mathcal{S}}^{(n)}$-invariant function $\zeta$ and $x_\alpha, x_\beta, \alpha,\beta=1,2,\ldots,n$. For any orthonormal basis $W=\{w_i\}_{i=1}^n\subset{\mathbb{R}}^n$, we denote the corresponding linear functions by $X_i=\langle w_i,x\rangle$ for each $i$. Then, there holds
    \begin{equation}\label{e3.5}
      \begin{cases}
       \displaystyle \int_{{\mathbb{S}}^{n-1}}X_i\zeta(x)d\sigma=0, & \forall i=1,2,\ldots,n,\\
      \displaystyle  \int_{{\mathbb{S}}^{n-1}}X_i X_j \zeta(x)d\sigma=\frac{\delta_{ij}}{n}\int_{{\mathbb{S}}^{n-1}}\zeta(x)d\sigma, & \forall i,j=1,2,\ldots,n.
      \end{cases}
    \end{equation}
\end{lemm}

\begin{proof}[Proof of Lemma \ref{l3.3}.]
 Regarding $w_i, i=1,2,\ldots,n$ as column vectors of a $n\times n$ matrix $W$, we have $W=[w_{i\alpha}]\in O(n)$ and $X_i=W_{i\alpha}x_\alpha$ for each $i$. Therefore, by \eqref{e3.3},
  \begin{eqnarray*}
   \int_{{\mathbb{S}}^{n-1}}X_i\zeta(x)d\sigma&=&W_{i\alpha}\int_{{\mathbb{S}}^{n-1}}x_\alpha\zeta(x)d\sigma=0,\\
   \int_{{\mathbb{S}}^{n-1}}X_iX_j\zeta(x)d\sigma&=&W_{i\alpha}W_{j\beta}\int_{{\mathbb{S}}^{n-1}}x_\alpha x_\beta\zeta(x)d\sigma=\frac{\delta_{ij}}{n}\int_{{\mathbb{S}}^{n-1}}\zeta(x)d\sigma
  \end{eqnarray*}
holds since $W\in O(n)$ is an orthonormal matrix. The proof of the lemma is completed.
\end{proof}

Let us continue our proof of Lemma \ref{l3.1} for $n+1$ regular simplex $T=\vartheta^{(n+2)}$. After rescaling, one may assume that all vertices of $T$ lie on the sphere ${\mathbb{S}}^n$. Fixing one vertex $v_{0}=(v_{0,1},v_{0,2},\ldots,v_{0,n+1})$ of $T$, we use a hyperplane
   $$
    H_r=\left\{x\in{\mathbb{R}}^{n+1}\bigg|\ \sum_{i=1}^{n+1}v_{0,i}x_i=r\right\}, \ \ r\in[-1,1]
   $$
to cut both $T, {\mathbb{S}}^n$, and obtain two possible sections $T_r, S_r\subset{\mathbb{S}}^n$ (when $r$ is close to $-1$, $T_r$ may be an empty set). Noting that $H_r$ is perpendicular to the line $Ov_0$, we know that $T_r$ is a $n$-dimensional regular simplex if $T_r\not=\emptyset$. To proceed further, let us state the second reduction lemma.

\begin{lemm}\label{l3.4}
  For dimension $n\geq2$ and $n+1$ regular simplex $T=\vartheta^{(n+2)}\subset{\mathbb{S}}^n$, we have the following two statements:

\noindent (1) when $T_r$ is non-empty, the symmetry group ${\mathcal{S}}_r={\mathcal{S}}_{T_r}\subset O(n)$ as a subgroup of $O(n)$ coincides with ${\mathcal{S}}_0$ (the symmetry group of a $n$ dimensional regular simplex) and is independent of $r$.

\noindent (2) When restricting $\zeta$ on $S_r$, it is a group ${\mathcal{S}}_{0}$ invariant function, even for the case of $T_r=\emptyset$.
\end{lemm}

\begin{proof}[Proof of Lemma \ref{l3.4}.]
  Part (1) of the lemma is clear. Actually, since $T_{r_1}$ and $T_{r_2}$ are similar $n$-regular simplex in a same position of ${\mathbb{R}}^n$ for different $r_1\not=r_2$, we know that ${\mathcal{S}}_{r_1}={\mathcal{S}}_{r_2}$. To show Part (2), we note that each mapping $\phi\in{\mathcal{S}}_0$ can be extended to a mapping in $\widetilde{\phi}\in{\mathcal{S}}(T)$ by defining it action on orthonormal subspace of the sections $T_r$ to be invariant. Actually, ${\mathcal{S}}_0$ can be regarded as a stabilizer subgroup of ${\mathcal{S}}(T)$. Noting that $\zeta$ is invariant under $\widetilde{\phi}$ and $S_r$ is invariant under $\widetilde{\phi}$, we conclude that when restricting $\zeta$ on $S_r$, it is $\phi$-invariant for each $\phi\in{\mathcal{S}}_0$. The proof of Part (2) is done.
\end{proof}

 Now, we choose an orthonormal basis
  $$
   \begin{cases}
  \displaystyle V_1=\left(v_{0,1},-v_{0,1}^2/v_{0,2},0,0,\ldots,0\right),\\[5pt] \displaystyle V_2=\left(v_{0,1},v_{0,2},-(v_{0,1}^2+v_{0,2}^2)/v_{0,3},0,\ldots,0\right),\\[5pt]
   V_j=(V_{j,1}, V_{j,2},\ldots, V_{j,i}, \ldots, V_{j,n+1}), \ \ j=3,4,\ldots,n
   \end{cases}
  $$
of $H_r$ for
  \begin{equation}
    V_{j,i}=\begin{cases}
      v_{0,i}, & i=1,2,\ldots,j,\\
       0, & i=j+2,j+3,\ldots,n+1,\\
      -\sum_{k=1}^{j}v_{0,k}^2/v_{0,j+1}, & i=j+1,
    \end{cases}
  \end{equation}
it follows from Lemmas \ref{l3.2}-\ref{l3.4} that
  \begin{eqnarray}\nonumber\label{e3.7}
  0&=&\int^1_{-1}\frac{1}{\sqrt{1-r^2}}\int_{S_r}\xi(x)\zeta(x)d\sigma_rdr\\ \nonumber
   &=&\int_{{\mathbb{S}}^n}\xi(x)\zeta(x)d\sigma\\
   &=&\int_{{\mathbb{S}}^n}{\mathcal{T}}_1(x)\zeta(x)d\sigma\\ \nonumber
   &=&{\mathcal{T}}_1\int_{{\mathbb{S}}^n}x_1\zeta(x)d\sigma,\\ \nonumber
   {\mathcal{T}}_1(x)&=&v_{0,1}x_1-v_{0,1}^2x_2/v_{0,2},\\ \nonumber
   {\mathcal{T}}_1&=&(v_{0,1}v_{0,2}-v_{0,1}^2)/v_{0,2},
  \end{eqnarray}
where $\xi=\langle V_1,x\rangle$ and $d\sigma_r$ denotes the induced measure of ${\mathcal{S}}_r$. Noting that $T$ has $n+2$ vertices, if the vertex $v_0$ is chosen such that
   \begin{equation}\label{e3.8}
     v_{0,1}\not=0, \ \ v_{0,2}\not=0, \ \ v_{0,1}\not=v_{0,2},
   \end{equation}
one can deduce from \eqref{e3.7} that $\int_{{\mathbb{S}}^n}x_1\zeta(x)d\sigma=0$. Using again Lemma \ref{l3.2}, we achieve the first identity of \eqref{e3.3}. To show the second identity of \eqref{e3.3}, we note that
  \begin{eqnarray}\nonumber\label{e3.9}
    0&=&\int^1_{-1}\frac{1}{\sqrt{1-r^2}}\int_{S_r}\xi(x)\eta(x)\zeta(x)d\sigma_rdr\\ \nonumber
    &=&\int_{{\mathbb{S}}^n}\xi(x)\eta(x)\zeta(x) d\sigma\\
     &=&\int_{{\mathbb{S}}^n}{\mathcal{T}}_2(x)\zeta(x)d\sigma\\ \nonumber
     &=&{\mathcal{T}}_2\int_{{\mathbb{S}}^n}x_1x_2\zeta(x)d\sigma, \\ \nonumber
     {\mathcal{T}}_2(x)&=&\left(v_{0,1}x_1-\frac{v_{0,1}^2}{v_{0,2}}x_2\right)\left(v_{0,1}x_1+v_{0,2}x_2-\frac{v_{0,1}^2+v_{0,2}^2}{v_{0,3}}x_3\right),\\ \nonumber {\mathcal{T}}_2&=&v_{0,1}(v_{0,2}-v_{0,1})\frac{v_{0,3}(v_{0,2}+v_{0,1})-(v_{0,1}^2+v_{0,2}^2)}{v_{0,2}v_{0,3}}
  \end{eqnarray}
holds for $\eta=\langle V_2,x\rangle$, where the symmetricity of $\zeta$ in Lemma \ref{l3.2} has also been used repeatedly. Now, supposing the chosen $v_0$ satisfies further that
  \begin{equation}\label{e3.10}
    v_{0,3}\not=0, \ \ v_{0,3}(v_{0,2}+v_{0,1})-(v_{0,1}^2+v_{0,2}^2)\not=0,
  \end{equation}
we deduce from \eqref{e3.9} that $\int_{{\mathbb{S}}^n}x_1x_2\zeta(x)d\sigma\not=0$. Hence, the desired second identity in \eqref{e3.3} is obtained by using Lemma \ref{l3.2} again repeatedly. Thus, the conclusion of our orthonormal Lemma \ref{l3.1} follows from the existence of vertex $v_0$ satisfying both \eqref{e3.8} and \eqref{e3.10}. Let us sum this existence result in the following lemma.

\begin{lemm}\label{l3.5}
  For dimension $n\geq2$, let $T=\vartheta^{(n+2)}\subset{\mathbb{R}}^{n+1}$ be the regular simplex in good position. Given $j\in[1,n+1]\cap{\mathbb{N}}$, we denote $v_j$ to be the vertex of $T$ located at the line $v_eX^e_j$. Then, the good position parameter $\lambda_e$ is given by $\lambda_e=\sqrt{n+2}-1$, and the coordinates of $v_j$ is given by
   \begin{equation}\label{e3.11}
     v_j=v_e+(\lambda_ee_j-v_e)\sqrt{\frac{n+2}{n+3-2\sqrt{n+2}}}
   \end{equation}
  for each $j$, where $v_e=(1,1,\ldots,1)$ has been defined in Definition \ref{d3.1} and $e_j$ is the $j-$th axial unit vector. As a corollary, if one chooses $v_0$ to be the vertex $v_2$, we have \eqref{e3.8} and \eqref{e3.10} is satisfied. Henceforth, it follows from \eqref{e3.7} and \eqref{e3.9} the desired orthonormal identity \eqref{e3.3}.
\end{lemm}

\begin{proof}
 After drawing a picture, for regular simplex $T$ in good position, the parameter $\lambda_e$ defined in Definition \ref{d3.1} is determined by
  \begin{eqnarray*}
   &&\sqrt{n+(1-\lambda_e)^2}=\lambda_e\sqrt{2}\\
   &\Leftrightarrow&\lambda_e=\sqrt{n+2}-1.
  \end{eqnarray*}
Assuming the length of side of $T$ is given by ${\mathcal{D}}(n+1)$, we claim that ${\mathcal{D}}(n+1)=\sqrt{2(n+2)}$. Actually, taking another $n+1$ regular simplex $T'(n+1)$ with inner radius ${\mathcal{I}}'(n+1)=1$, outer radius ${\mathcal{R}}'(n+1)=n+1$ and length of side ${\mathcal{D}}'(n+1)$. Regarding ${\mathcal{D}}'(n+1)T'(n)/{\mathcal{D}}'(n)$ as base of $T'(n+1)$, we get the relation
    \begin{eqnarray}\nonumber\label{e3.12}
     &&\displaystyle {\mathcal{D}}'^2(n+1)=(n+2)^2+\left[\frac{{\mathcal{D}}'(n+1){\mathcal{R}}'(n)}{{\mathcal{D}}'(n)}\right]^2\\
     &\Leftrightarrow&\displaystyle {\mathcal{D}}'(n+1)=\frac{(n+2){\mathcal{D}}'(n)}{\sqrt{{\mathcal{D}}'^2(n)-n^2}}.
    \end{eqnarray}
 Noting that ${\mathcal{D}}'(2)=2\sqrt{3}$, it is deduced from the recursive formula \eqref{e3.12} that
    \begin{equation}\label{e3.13}
    \begin{cases}
     {\mathcal{D}}'(n+1)=\sqrt{2(n+2)(n+1)},\\
     {\mathcal{R}}'(n+1)=n+1,\\
     {\mathcal{I}}'(n+1)=1.
    \end{cases}
    \end{equation}
 Comparing to the regular simplex $T$, the length of side, outer radius and inner radius are given by
   \begin{equation}\label{e3.14}
   \begin{cases}
     {\mathcal{D}}(n+1)={\mathcal{D}}'(n+1)/\sqrt{n+1},\\
     {\mathcal{R}}(n+1)=\sqrt{n+1},\\
     {\mathcal{I}}(n+1)=\sqrt{\frac{1}{n+1}}.
   \end{cases}
   \end{equation}
 Combining \eqref{e3.13} with \eqref{e3.14}, it follows from similarity that the claim ${\mathcal{D}}(n+1)=\sqrt{2(n+2)}$ is true. Now, using the fact that $v_j, X^e_j, v_e$ are collinear, it yields that
   \begin{eqnarray}\nonumber
    v_j-v_e&=&\frac{|v_ev_j|}{|v_eX^e_j|}(X^e_j-v_e)\\
       &=&(\lambda_ee_j-v_e)\sqrt{\frac{n+2}{n+3-2\sqrt{n+2}}}
   \end{eqnarray}
 and thus the desired identity \eqref{e3.9}. To show the final statement, we note that
   \begin{equation}
    \begin{cases}
   \displaystyle v_0=v_2=\left(1-\gamma_n, \frac{\sqrt{n+2}-n\gamma_n-2}{\sqrt{n+2}-2}, 1-\gamma_n,\ldots,1-\gamma_n\right),\\[10pt]
    \displaystyle\gamma_n=\sqrt{\frac{n+2}{n+3-2\sqrt{n+2}}}.
    \end{cases}
   \end{equation}
 Therefore, the conditions \eqref{e3.8} and \eqref{e3.10} hold true since
   $$
   \begin{cases}
    v_{0,1}\not=v_{0,2},\\
    v_{0,1}, v_{0,2}, v_{0,3}\not=0,\\
    v_{0,3}(v_{0,2}+v_{0,1})-(v_{0,1}^2+v_{0,2}^2)\not=0.
   \end{cases}
   $$
 The proof of Lemma \ref{l3.5} is completed.
\end{proof}

 Now, the orthonormal Lemma \ref{l3.1} follows from Lemma \ref{l3.5} as desired.
\end{proof}

Next, we consider the regular polyhedrons $T=\vartheta^{(k)}$ for $n=2, k=12, 20$ and show the following orthonormal lemma.

\begin{lemm}\label{l3.6}
 Letting ${\mathcal{S}}\subset O(3)$ be the symmetry group of a regular polytope $T=\vartheta^{(k)}$ for $n=2, k=12, 20$, the relations    \begin{equation}\label{e3.17}
      \begin{cases}
        \displaystyle{\mathcal{W}}_{\alpha}(\zeta)=0, & \forall\alpha=1,2,3,\\
        \displaystyle{\mathcal{W}}_{\alpha\beta}(\zeta)=\frac{\delta_{\alpha\beta}}{3}\int_{{\mathbb{S}}^n}\zeta(x)d\sigma, & \forall\alpha,\beta=1,2,3
      \end{cases}
    \end{equation}
 hold for all group ${\mathcal{S}}$ invariant functions $\zeta$.
\end{lemm}

\begin{proof}
 We will give the proof of $k=20$ in spirit of the proof to Lemma \ref{l3.1}, the case $k=12$ follows from the case $k=20$ by duality.  As in the following figure, the regular dodecahedron ($k=20$) in this position is built on a cube in good position, by adding two points on each face of the cube. We will call this regular dodecahedron to be in good position, and call its dual icosahedron ($k=12$) to be in good position after scaling if necessary.
\begin{figure}[h]
    \centering
    \includegraphics[width=0.35\textwidth]{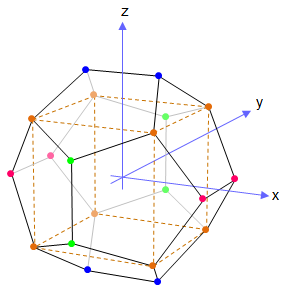}
    \caption{A regular dodecahedron in good position (source: Wikipedia)}
    \label{Dodecahedron}
\end{figure}
Therefore, the group ${\mathcal{S}}$ invariant functions (or convex bodies) are invariant under the mappings $\phi_0, \phi_i\in O(n+1), i=1,2,3$ satisfying  \begin{equation}\label{e3.18}
  \begin{cases}
   \phi_i(x_i)=-x_i, \ \ \phi_i(x_j)=x_j, \ \ \forall j\not=i,\\
   \phi_{0}(x_1)=x_2, \ \ \phi_{0}(x_2)=x_3, \ \ \phi_{0}(x_3)=x_1
  \end{cases}
  \end{equation}
for $i=1,2,3$ and circle mapping $\phi_0$. Therefore, \eqref{e3.17} follows from \eqref{e3.18} easily. Actually, acting $\phi_\alpha, \phi_{\alpha,\beta}$ at the integrals, we obtain that
   \begin{eqnarray}\nonumber\label{e3.19}
    {\mathcal{W}}_{\alpha}(\zeta)&=&\frac{1}{2}\int_{{\mathbb{S}}^2}[x_\alpha+(-x_\alpha)]\zeta(x)d\sigma, \ \ \forall\alpha,\\
    {\mathcal{W}}_{\alpha\alpha}(\zeta)&=&\int_{{\mathbb{S}}^2}x_\beta^2\zeta(x)d\sigma, \ \ \forall\alpha\not=\beta,\\ \nonumber
    {\mathcal{W}}_{\alpha\beta}(\zeta)&=&\frac{1}{2}\int_{{\mathbb{S}}^2}[x_\alpha+(-x_\alpha)] x_\beta\zeta(x)d\sigma=0,\ \ \forall\alpha\not=\beta.
   \end{eqnarray}
\end{proof}

\begin{lemm}\label{l3.7}
 Letting ${\mathcal{S}}\subset O(4)$ be the symmetry group of a regular polytope $T=\vartheta^{(k)}$ for $n=3, k=120, 600$, the relations    \begin{equation}\label{e3.20}
      \begin{cases}
        \displaystyle{\mathcal{W}}_{\alpha}(\zeta)=0, & \forall\alpha=1,2,3,4,\\
        \displaystyle{\mathcal{W}}_{\alpha\beta}(\zeta)=\frac{\delta_{\alpha\beta}}{4}\int_{{\mathbb{S}}^n}\zeta(x)d\sigma, & \forall\alpha,\beta=1,2,3,4
      \end{cases}
    \end{equation}
 hold for all group ${\mathcal{S}}$ invariant functions $\zeta$.
\end{lemm}

\begin{proof}
  We need only prove the lemma for $k=120$, the case $k=600$ follows by duality. As is well known, sixteen vertices of a standard  $T=\vartheta^{(120)}$ (hexacosichoron) are given by
   $$
    v_1,\ldots, v_{16}\in\{(\pm1/2,\pm1/2,\pm1/2,\pm1/2)\},
   $$
  eight vertices are given by permutations of
   $$
    v_{17},\ldots, v_{24}\in\{(0,0,0,\pm1), (0,0,\pm1,0), (0,\pm1,0,0), (\pm1,0,0,0)\}
   $$
  and the remaining ninety six vertices are given by even permutations mappings $EP$ and
   $$
    v_{25},\ldots, v_{120}\in EP\{(\pm1/2,\pm\varphi/2,\pm1/(2\varphi),0)\}, \ \ \varphi=\frac{\sqrt{5}+1}{2}.
   $$
  Henceforth, we will call this regular hexacosichoron and its dual regular hecatonicosachoron to be in good position. Note that the regular polytope $T$ is invariant under mappings $\phi_i, \phi_{i,j}\circ\phi_{k,l}$ for all different $i,j,k,l$, where $\phi_i, \phi_{i,j}$ are the mappings defined by
  \begin{equation}\label{e3.21}
  \begin{cases}
   \phi_i(x_i)=-x_i, \ \ \phi_i(x_j)=x_j, \ \ \forall j\not=i,\\
   \phi_{i,j}(x_i)=x_j, \ \ \phi_{i,j}(x_j)=x_i, \ \ \phi_{i,j}(x_k)=x_k, \ \ \forall k\not=i,j
  \end{cases}
  \end{equation}
   Hence, we still have
   \begin{equation}\label{e3.20}
   \begin{cases}
   \displaystyle {\mathcal{W}}_{\alpha}(\zeta)=\frac{1}{2}\int_{{\mathbb{S}}^3}[x_\alpha+(-x_\alpha)]\zeta(x)d\sigma=0, \ \ \forall\alpha,\\[5pt]
   \displaystyle   {\mathcal{W}}_{\alpha\alpha}(\zeta)={\mathcal{W}}_{\beta\beta}(\zeta), \ \ \forall\alpha\not=\beta,\\[5pt] \nonumber
    \displaystyle  {\mathcal{W}}_{\alpha\beta}(\zeta)=\frac{1}{2}\int_{{\mathbb{S}}^3}[x_\alpha+(-x_\alpha)] x_\beta\zeta(x)d\sigma=0,\ \ \forall\alpha\not=\beta
   \end{cases}
   \end{equation}
  by applying $\phi_\alpha, \phi_{\alpha,\beta}\circ\phi_{\gamma,\delta}$. The proof of the lemma is done.
\end{proof}

\begin{lemm}\label{l3.8}
 Letting ${\mathcal{S}}\subset O(4)$ be the symmetry group of a regular polytope $T=\vartheta^{(k)}$ for $n=3, k=24$, the relations
    \begin{equation}\label{e3.22}
      \begin{cases}
        \displaystyle{\mathcal{W}}_{\alpha}(\zeta)=0, & \forall\alpha=1,2,3,4,\\
        \displaystyle{\mathcal{W}}_{\alpha\beta}(\zeta)=\frac{\delta_{\alpha\beta}}{4}\int_{{\mathbb{S}}^n}\zeta(x)d\sigma, & \forall\alpha,\beta=1,2,3,4
      \end{cases}
    \end{equation}
 hold for all group ${\mathcal{S}}$ invariant functions $\zeta$.
\end{lemm}

\begin{proof}
 At first, we note that eight vertices of a standard  $T=\vartheta^{(24)}$ (icositetrachoron) are given by
   $$
    v_1,\ldots, v_{8}\in\{(\pm1,0,0,0), (0,\pm1,0,0), (0,0,\pm1,0), (0,0,0,\pm1)\},
   $$
 and the other sixteen vertices are given by
   $$
    v_{9},\ldots, v_{24}\in\{(\pm1/2,\pm1/2,\pm1/2,\pm1/2)\}.
   $$
We will call this regular icositetrachoron to be in good position.
 It is clear that $T=\vartheta^{(24)}$ is invariant under all mappings $\phi_i, \phi_{i,j}$ defined by \eqref{e3.21}. Thus, a same formula as \eqref{e3.20} holds for this regular polytope. This completes the proof.
\end{proof}

\subsection{Uniqueness of strongly symmetric solutions in a slightly supercritical range $p\in(-n-1-\sigma_n,-n-1)$}

\noindent In this subsection, let us prove the following weaker version of Theorem \ref{t1.2} for strongly symmetric solutions with respect to the symmetry group ${\mathcal{S}}(T)$ of a regular polytope $T\subset{\mathbb{R}}^{n+1}$.

\begin{prop}\label{p3.1}
Let dimension $n\geq2$ and consider \eqref{e1.2}. There exists a universal constant $\sigma_n>0$, such that for all $p\in(-n-1-\sigma_n,-n-1)$, the unique strongly symmetric solution of \eqref{e1.2} is given by $K={\mathbb{S}}^n$.
\end{prop}

Once the validity of Proposition 3.1 and topological Theorem 6.1 (a parallel version of topological Theorem 1.4) for the uniqueness set are established, the full uniqueness Theorem \ref{t1.2} in the range $p\in(-2n-5,-n-1)$ will be proven. Now, let us first prove the following crucial lemma as a corollary of the orthonormal lemmas (or Theorem \ref{t1.3}) shown in Subsection 3.1.

\begin{lemm}\label{l3.9}
Letting $K\in{\mathcal{K}}(T)$ be a strongly symmetric convex body, there holds
 \begin{equation}\label{e3.23}
   \int_{{\S }^n}|Z_K^\perp|^2dV_K\geq\int_{{\S }^n}|\nabla h_K|^2dV_K.
 \end{equation}
\end{lemm}

\begin{proof} Supposing that $K$ satisfies the strongly symmetric property, by orthonormal Lemmas \ref{l3.1}, \ref{l3.6}-\ref{l3.8} or equivalently Theorem \ref{t1.3}, $\{x_1,x_2,\ldots,x_{n+1}\}$ forms an orthonormal frame of second eigen-subspace $E^{(2)}$ of the eigen-problem \eqref{e2.1}. By the divergence theorem,
  \begin{eqnarray}\nonumber\label{e3.24}
   \int_{{\S }^n}\frac{x_\alpha}{h_K}Z^{(l)}dV_K&=&\int_{\partial K} \nu_\alpha p_ld{\mathcal{H}}^n(p)\\
   &=&\int_Kdiv_{{\mathbb{R}}^{n+1}}(E_\alpha y_l)dy\\ \nonumber
   &=&\int_K\delta_{\alpha l}dy=\frac{1}{n+1}\delta_{\alpha,l}\int_{{\S }^n}dV_K,
  \end{eqnarray}
where $Z^{(l)}$ is the $l-$th entry of the vector $Z_K$. Noting that the centroid of a strongly symmetric convex body $K$ is also symmetric under the action of symmetry group of $K$, it is inferred that $K$ must be origin-centering. So, it follows from Lemma \ref{l2.3} that $\int_{{\S}^n}Z_KdV_K=0$. Noting that the orthonormal frame $\{X_\alpha\}_{\alpha=1}^{n+1}$ of $E^{(2)}$ can be taking to be $\{x_\alpha\}_{\alpha=1}^{n+1}$ by orthonormal Theorem \ref{t1.3}, it follows from \eqref{e2.10} and \eqref{e3.24} that
   \begin{eqnarray}\nonumber\label{e3.25}
    \int_{{\S }^n}|Z^\perp|^2dV_K&=&\int_{{\S }^n}|Z|^2dV_K-\sum_{\alpha,l=1}^{n+1}\left(\int_{{\S }^n}\frac{x_\alpha}{h_K}Z^{(l)}dV_K\right)^2{\bigg/}\int_{{\S }^n}\left|\frac{x_\alpha}{h_K}\right|^2dV_K\\
    &=&\int_{{\S }^n}|Z|^2dV_K-\frac{1}{n+1}\left(\int_{{\S }^n}dV_K\right)^2{\bigg/}\int_{{\S }^n}\left|\frac{x_1}{h_K}\right|^2dV_K\\ \nonumber
    &=&\int_{{\S }^n}|Z|^2dV_K-\left(\int_{{\S }^n}dV_K\right)^2{\bigg/}\int_{{\S }^n}h_K^{-2}dV_K\\ \nonumber
    &\geq&\int_{{\S }^n}|\nabla h|^2dV_K
   \end{eqnarray}
by H\"{o}lder's inequality. This completes the proof of \eqref{e3.23}.
\end{proof}

\begin{proof}[Complete the proof of Proposition \ref{p3.1}] Using the Rayleigh's quotient representation
    $$
     \lambda_3(K)=\inf_{P_{n+2}}\sup_{\phi\in P_{n+2}}R(\phi), \ \ R(\phi)=\frac{\int_{{\S }^n}U^{ij}_K\phi_i\phi_j-n\int_{{\S }^n}U^{ii}\phi^2}{\int_{{\S }^n}h_K^{-1}\det(\nabla^2h_K+h_KI)\phi^2}
    $$
 of the third eigenvalue of \eqref{e2.1}, where $P_{n+2}$ is any $n+2$ dimensional subspace of $C^{2,\alpha}({\S }^n)$, one can see that $\lambda_3(K)$ is continuous in $K$. Actually, since the first $n+2$ eigenvalues of \eqref{e2.1} are given by  $\{-n, \ 0, \ 0, \ldots, 0, 0\}$, the quantity $\lambda_3(K)$ is exactly the $(n+3)-$th eigenvalue of \eqref{e2.1}, taking into account the multiplicity. So, the continuity of $\lambda_3(\cdot)$ follows from the Rayleigh's representation formula above. Therefore, from the {\it a-priori} estimate (Lemmas 3.2 and 3.3, \cite{DWZ24}) for strongly symmetric solutions, we know that $\lambda_3(K)$ is bounded from above and below by positive constants. Combined with Lemma \ref{l3.9}, the validity of the Combined Inequality \eqref{e1.3} will be verified since
  \begin{eqnarray*}
   \Theta(p)&\geq&\inf_{K\in{\mathcal{C}}'_{n,p}(T)}\lambda_3(K)\\
    &\geq&\sigma_n>0.
  \end{eqnarray*}
Thus, Proposition \ref{p3.1} follows from Theorem \ref{t1.1} by taking $\sigma_n$ to be smaller if necessary.
\end{proof}

\begin{proof}[Proof of Theorem \ref{t1.2}.]
  Now, Theorem \ref{t1.2} directly follows from Proposition \ref{p3.1} and Theorem \ref{t6.1}.
\end{proof}

\medskip

\section{Infinitesimal Generator Lemma for solution sequence}

\noindent In the proof of our main uniqueness results, one of the crucial ingredients is the Infinitesimal Generator Lemma, which establishes the relation between the kernel of the linearized equation and the infinitesimal generator of the origin equation. Before stating the Infinitesimal Generator Lemma, let us first lay down some definitions and notations. This part is motivated by the ideas in the Lie's Theory (refer to the book \cite{O00} by Peter Olver) for Partial Differential Equations.

\subsection{Mapping Representation for approximation sequence}

 Letting $(x,u)\in{\R}^n\times{\R}$ be a varying pair satisfying the relation $u=u(x)$, denote $u^{(1)}\in{\R}^{n}, u^{(2)}\in{\R}^{n^2}$ to be the first and second derivatives of $u$. We call $u$ a solution to the second order Partial Differential Equation
     \begin{equation}\label{e4.1}
       \Delta(x,u,u^{(1)},u^{(2)})=0
     \end{equation}
 if substituting $(x,u,u^{(1)},u^{(2)})$ into \eqref{e4.1} yields an identity. The next lemma is a consequence of the implicit function theorem of Boothby (\cite{B75}, Theorem II. 7.1). The readers may also refer to the book (\cite{O00}, Theorem 1.8) for additional reference.

 \begin{lemm}\label{l4.1}(Mapping Representation Lemma)
  Let $u=u(x)$ be a $C^1$-function satisfying the maximal rank property $Du(x_0)\not=0$ at some point $x_0\in{\R}^n$. Then for some small constant $\varepsilon$ and some neighborhood $\Omega_\varepsilon(x_0)$ of $x_0$, for any $C^1$-function $v=v(x)$ close to $u=u(x)$ in the sense that
     \begin{equation}
       ||v(x)-u(x)||_{C^1(\Omega_\varepsilon(x_0))}\leq\varepsilon,
     \end{equation}
  there exist local $C^1-$diffeomorphism mapping pair $(\psi, \eta)$ such that the function $v=v(y)$ is given by
     \begin{equation}
       \begin{cases}
         y=\psi(x,u),\\
         v=\eta(x,u).
       \end{cases}
     \end{equation}
  Moreover, the mapping pair is close to the identity mapping in the sense
     \begin{equation}
       ||(\psi(x,u),\eta(x,u))-(x,u)||_{C^1}=o_\varepsilon(1),
     \end{equation}
  where $o_\varepsilon(1)$ is an infinitesimal as long as $\varepsilon$ is small.
 \end{lemm}

\noindent Lemma \ref{l4.1} enables us representing an approximation sequence of a given function by a sequence of mappings for this function, which closes to identity mapping as needed. Under below, we will call $g=(\psi,\eta)$ the mapping representation of the approximation function $v$ for given function $u$.

 \begin{proof}
    Noting that $Du(x_0)\not=0$, when $\varepsilon$ is small, one observes that both $u$ and $v$ satisfy the maximal rank property
       \begin{equation}
         Du(x)\not=0, \ \ Dv(y)\not=0, \ \ \forall x,y\in\Omega_\varepsilon(x_0)
       \end{equation}
    in some small neighborhood $\Omega_\varepsilon(x_0)$ of $x_0$. Without loss of generality, one may assume that
       \begin{equation}
        D_1u(x)\not=0, \ \ D_1v(y)\not=0, \ \ \forall x,y\in\Omega_\varepsilon(x_0).
       \end{equation}
    Assuming that the functions $u,v$ are given by $u=f(x)$ and $v=g(y)$. Then, by implicit theorem, one can solve the equation $u'=f(x_1,x_2,\ldots,x_n)$ to generate a function $x_1=F(u',x_2,\ldots,x_n)$. Noting that the Jacobian matrix 
     $$
      \frac{\partial(u',x_2,\cdots, x_n,)}{\partial(x_1,x_2,\cdots, x_n)}=\left[
       \begin{array}{cccccc}
        f_1 & f_2 & f_3 & \ldots & f_{n-1} & f_n\\
         0  & 1  & 0 & \ldots & 0 & 0\\
         0 &  0 &  1 & \ldots & 0 & 0\\
         \vdots & \vdots & \vdots & \vdots & \vdots & \vdots \\
          0 & 0 & 0 & \ldots & 1 & 0 \\
           0 & 0 & 0 & \ldots & 0 & 1
       \end{array}
       \right]
     $$
  is non-degenerate locally since $f_1(x_0)=D_1u(x_0)\not=0$, if one regards $(u',x_2,\ldots, x_n)$ as new coordinates instead of $(x_1,x_2,\ldots,x_n)$, the function $u=f(x)$ will be presented by
       $$
        u=f(F(u',x_2,\cdots,x_n),x_2,\cdots,x_n)=u'.
       $$
    Similarly, if one solves the equation $v'=g(y_1,y_2,\ldots,y_n)$ to generate a function $y_1=G(v',y_2,\cdots,y_n)$. Regarding $(v',y_2,\ldots, y_n)$ as new coordinates instead of $(y_1,y_2,\cdots,y_n)$, the function $v=v(y)$ will be presented by
       $$
        v=g(G(v',y_2,\cdots,y_n),y_2,\cdots,y_n)=v'.
       $$
    Defining the transformation on functions by the formation
      \begin{equation}
       \begin{cases}
         y_2=x_2, \ \ y_3=x_3,\ \ \ldots, \ \ y_n=x_n,\\
         y_1=G(f(x),x_2,\ldots,x_n),\\
         v=u,
       \end{cases}
      \end{equation}
    we obtained the desired mapping pair $(\psi,\eta)$ which maps $u=f(x)$ to $v=g(y)$. Moreover, the mapping pair closes to the identity mapping as long as $\varepsilon$ is small. Actually, the Jacobian matrix of this mapping is given by
      \begin{equation}
        \frac{\partial(y,v)}{\partial(x,u)}=\left[
          \begin{array}{cccccc}
            G_1f_1 & 0 & 0 & \ldots & 0 & 0\\
            G_1f_2+G_2 & 1 & 0 & \ldots & 0 & 0\\
            G_1f_3+G_3 & 0 & 1 &  \ldots & 0 & 0\\
            \vdots & \vdots & \vdots& \vdots & \vdots & \vdots\\
             G_1f_n+G_n & 0 & 0 & \ldots & 1 & 0\\
              0 & 0 & 0 & \ldots & 0 & 1
          \end{array}
          \right]
      \end{equation}
    Noting that $x_1\equiv F(f(x),x_2,\ldots,x_n)$ and $y_1\equiv G(g(y),y_2,\ldots,y_n)$, it follows from chain rule that
       \begin{equation}
         \begin{cases}
            F_1f_1=1, \ \ F_1f_2+F_2=0, \ \ldots, \ \ F_1f_n+F_n=0,\\
            G_1g_1=1, \ \ G_1g_2+G_2=0, \ \ldots, \ \ G_1g_n+G_n=0.
         \end{cases}
       \end{equation}
    By our assumption that $f(x)$ is closing to $g(y)$, we conclude that Jacobian matrix $\frac{\partial(y,v)}{\partial(x,u)}$ approaches the identity matrix as desired. The proof of the lemma has been done.
 \end{proof}

\begin{rema}\label{r4.1} It is remarkable that for given approximation solutions $u_\varepsilon(y_\varepsilon)$ of \eqref{e4.11}, the mapping representation $(\psi_\varepsilon, \eta_\varepsilon)$ of $$
  \begin{cases}
         y_\varepsilon=\psi_\varepsilon(y,u),\\
         u_\varepsilon=\eta_\varepsilon(y,u)
  \end{cases}
 $$ 
may not be unique in general. Actually, for approximation sequence $u_\varepsilon(y_{\varepsilon,1},y_{\varepsilon,2})=e^\varepsilon \sin (y_{\varepsilon,1}+y_{\varepsilon,2}$ of $u(y_1,y_2)=\sin (y_1,y_2)$, the mapping pair could be
  $$
   \begin{cases}
     y_{\varepsilon,1}=y_1, \ y_{\varepsilon,2}=y_2\\
     u_\varepsilon=e^\varepsilon u,
   \end{cases}
  $$
or could be 
   $$
   \begin{cases}
     y_{\varepsilon,1}=y_1+\varepsilon, y_{\varepsilon,2}=y_2-\varepsilon\\
     u_\varepsilon=e^\varepsilon u.
   \end{cases}
  $$
\end{rema}

\noindent Now, given a function $u$, we assume that $u_\varepsilon$ is a sequence of functions approximating $u$ as $\varepsilon$ tends to $0$. Letting $g_\varepsilon=\left[
    \begin{array}{c}
       \psi_\varepsilon(x,u)\\
       \eta_\varepsilon(x,u)
    \end{array}
    \right]$ be the Mapping Representation of the approximating function $u_\varepsilon$, we define the Infinitesimal Generator Field to be the vector field
        \begin{equation}\label{e4.10}
          \upsilon:=\lim_{\varepsilon\to0}\frac{g_\varepsilon-g_0}{\varepsilon}=\xi^i(x,u)\frac{\partial}{\partial x_i}+\zeta(x,u)\frac{\partial}{\partial u},
        \end{equation}
    assuming the limit in \eqref{e4.10} exists.
    
\begin{rema}\label{r4.2}
  As shown in Remark \ref{r4.1}, for given approximation sequence, there may be different mapping representations with different assigned vector fields.
\end{rema}

\subsection{Prolongation formula and infinitesimal generator field}

In the Lie's theory, Prolongation Formula \cite{O00} plays a central role, which shows transformations on the domain and target variables $(x,u)$ induce transformations on the derivatives. Let us now lay down some notations and formulas before imposing the famous Prolongation Formula of the Lie's theory. It is remarkable that our settings are somewhat different from those in the book \cite{O00} by Olver, since we deal with discrete transformations rather than $C^1$-continuous group actions.

Assuming $h=h(x), x\in{\S}^n$ is a solution to \eqref{e1.2}, performing semi-spherical projection
   $$
     x=T^*(y)=\left(\frac{y}{\sqrt{1+|y|^2}},-\frac{1}{\sqrt{1+|y|^2}}\right)\in{\S}^n, \ \ y\in{\R}^n
   $$
and setting $h(x)=\frac{u(y)}{\sqrt{1+|y|^2}}$, the equation \eqref{e1.2} will be transformed into
   \begin{equation}\label{e4.11}
     \det D^2u=(1+|y|^2)^{-\frac{p+n+1}{2}}u^{p-1}, \ \ \forall y\in{\R}^n.
   \end{equation}
Now, given a solution $h_K$ of \eqref{e1.2}, we assume that there exists a sequence of approximation solutions $h_\varepsilon, \varepsilon\in{\R}$ of the approximation equation
    \begin{equation}\label{e4.12}
      \det(\nabla^2h_\varepsilon+h_\varepsilon I)=h_\varepsilon^{p-1}+o\big(||h_\varepsilon-h_K||_{L^2({\S}^n)}\big), \ \ \forall x\in{\S}^n
    \end{equation}
 tending to $h_K$ as $\varepsilon$ tends to zero. By elliptic estimates of the subtraction equation
  \begin{eqnarray}\label{e4.13}
    \int^1_0U^{ij}_tdt(\phi_{\varepsilon,ij}+\phi_\varepsilon\delta_{ij})=(p-1)\int^1_0h_t^{p-2}dt\phi_\varepsilon+o(1)
  \end{eqnarray}
 for $h_t=th_\varepsilon+(1-t)h_K$, infinitesimal $o(1)$ and $\phi_\varepsilon:=\frac{h_\varepsilon-h_K}{||h_\varepsilon-h_K||_{L^2({\S}^n)}}$,
we have $\phi_\varepsilon$ sub-converges to a limit function $\phi$, which is a solution of the linearized equation
   \begin{eqnarray}\nonumber\label{e4.14}
      U^{ij}_K(\phi_{ij}+\phi\delta_{ij})&=&(p-1)h_K^{p-2}\phi\\
         &=&(p-1)\frac{\det(\nabla^2h_K+h_KI)}{h_K}\phi.
   \end{eqnarray}
 By rearranging the parameter $\varepsilon$ if necessary, for example $\varepsilon=||h_\varepsilon-h_K||_{L^2({\S}^n)}$, we will call $\phi$ to be discrete differentiation of $h_\varepsilon$ and denote by
    $$
     \partial_\varepsilon|_{\varepsilon=0} h_\varepsilon=\phi.
    $$
Noting that $u_\varepsilon(y_\varepsilon)=\sqrt{1+|y_\varepsilon|^2}h_\varepsilon(x_\varepsilon)$ is a sequence of solutions of \eqref{e4.11} tending to solution $u_K(y)=\sqrt{1+|y|^2}h_K(x)$ of \eqref{e4.11}, we use
   $$
    \left[
       \begin{array}{c}
        y_\varepsilon\\
        u_\varepsilon
       \end{array}
    \right]=g_\varepsilon\circ\left[
       \begin{array}{c}
        y\\
        u
       \end{array}
    \right]=
      \left[
        \begin{array}{c}
           \psi_\varepsilon(y,u)\\
           \eta_\varepsilon(y,u)
    \end{array}
      \right]
   $$
to be the Mapping Representation of $u_\varepsilon(y_\varepsilon)$. Taking the discrete differentiation of $g_\varepsilon$ on $\varepsilon$, we assume that the infinitesimal generator is given by the vector field
    \begin{equation}\label{e4.15}
      \upsilon=\partial_\varepsilon\big|_{\varepsilon=0}g_\varepsilon\circ\left[
       \begin{array}{c}
        y\\
        u
       \end{array}
    \right]=\xi^i(y,u)\frac{\partial}{\partial y_i}+\zeta(y,u)\frac{\partial}{\partial u}.
    \end{equation}
Since the transforms on the domain and target variables $(x,u)$ bringing transforms on the differentiations, we need to impose the Prolongation Formula (\cite{O00}, page 110, Theorem 2.36)
  \begin{equation}\label{e4.16}
   pr^{(2)}\upsilon=\xi^i\frac{\partial}{\partial y_i}+\zeta\frac{\partial}{\partial u}+\zeta^i\frac{\partial}{\partial u_i}+\zeta^{ij}\frac{\partial}{\partial u_{ij}},
  \end{equation}
where $y, u, u_i, u_{ij}$ are regarded as independent variables and
    \begin{eqnarray}\nonumber\label{e4.17}
      \zeta^i&=&D_i(\zeta-\xi^ju_j)+\xi^ju_{ij}\\ \nonumber
        &=&(\zeta_i+\zeta_uu_i)-(\xi^j_i+\xi^j_uu_i)u_j,\\ \nonumber
      \zeta^{ij}&=&D_{ij}(\zeta-\xi^au_a)+\xi^au_{ija}\\
       &=& D_j[\zeta_i+\zeta_uu_i-(\xi^a_i+\xi^a_uu_i)u_a-\xi^au_{ia}]+\xi^au_{ija}\\ \nonumber
       &=&\zeta_{ij}+\zeta_{iu}u_j+(\zeta_{uj}+\zeta_{uu}u_j)u_i+\zeta_uu_{ij}\\ \nonumber       &&-[(\xi^a_{ij}+\xi^a_{iu}u_j)+(\xi^a_{uj}+\xi^a_{uu}u_j)u_i+\xi^a_uu_{ij}]u_a\\ \nonumber
       &&-(\xi^a_i+\xi^a_uu_i)u_{ja}-(\xi^a_j+\xi^a_uu_j)u_{ia}.
    \end{eqnarray}
 Here, the symbol $D_i\zeta$ denotes the total  derivative of $\zeta$ with respect to $y_i$. Says, $D_i\zeta(x,u(x))=\zeta_i+\zeta_uu_i$.

\begin{defi}
  Given a second order Partial Differential Equation
    \begin{equation}\label{e4.18}
     \Delta\left(y,u,u^{(1)},u^{(2)}\right)=0,
    \end{equation}
  we call a vector field $\upsilon=\xi^i(y,u)\frac{\partial}{\partial y_i}+\zeta(y,u)\frac{\partial}{\partial u}$ to be an Infinitesimal Generator (Field) of \eqref{e4.18}, if the identity
    \begin{equation}\label{e4.19}
      pr^{(2)}\upsilon\circ \Delta\left(y,u,u^{(1)},u^{(2)}\right)\equiv0
    \end{equation}
  is fulfilled.
\end{defi}

\begin{exam}
  To clarify the infinitesimal generator field for a sequence of solutions (or approximation solutions), let us consider the following simplified equation
     \begin{equation}\label{e4.20}
      \Delta\left(y,u,u',u''\right):=u''-u^p=0
     \end{equation}
  and assume that $u=u(y)$ is a give solution. Noting that $u_\varepsilon(y_\varepsilon)=e^{\frac{2}{p-1}\varepsilon}u(e^\varepsilon y_\varepsilon)$ are also solutions of \eqref{e4.20} for any $\varepsilon\in{\R}$, the assigned mappings are given by
     \begin{equation}
      g_\varepsilon\circ\left[
        \begin{array}{c}
          y\\
          u
        \end{array}
        \right]=\left[
        \begin{array}{c}
          \psi_\varepsilon(y,u)\\
          \eta_\varepsilon(y,u)
        \end{array}
        \right]=\left[
        \begin{array}{c}
          e^{-\varepsilon}u^{-1}\Big(e^{-\frac{2}{p-1}\varepsilon}u\Big)\\
          u
        \end{array}
        \right]
     \end{equation}
  for each $\varepsilon$. Then, the assigned vector field
     $$
      \upsilon=\frac{d}{d\varepsilon}\Big|_{\varepsilon=0}g_\varepsilon\circ\left[
        \begin{array}{c}
          y\\
          u
        \end{array}
        \right]=\left[
        \begin{array}{c}
          -y-\frac{2}{p-1}\frac{u}{u'(y)}\\
          0
        \end{array}
        \right]
     $$
  is an infinitesimal generator of the equation \eqref{e4.20}. Actually, by Prolongation Formula \eqref{e4.16}, we have
    \begin{equation}
     pr^{(2)}\upsilon=\xi\frac{\partial}{\partial y}+\zeta'\frac{\partial}{\partial u'}+\zeta''\frac{\partial}{\partial u''}
    \end{equation}
  holds for the coefficients
     \begin{eqnarray*}
      \xi&=&-y-\frac{2}{p-1}\frac{u}{u'},\\
      \zeta'&=&\bigg(\frac{p+1}{p-1}-\frac{2}{p-1}\frac{uu''}{u'^2}\bigg)u',\\
      \zeta''&=&-\frac{2}{p-1}\frac{uu'''+u'u''-(p+1)u'u''}{u'}.
     \end{eqnarray*}
  Substituting into \eqref{e4.20}, we obtain that
    $$
     pr^{(2)}\upsilon\circ \Delta(y,u,u',u'')=-\frac{2}{p-1}\frac{uu'''-pu'u''}{u'}=0
    $$
  and thus conclude that $\upsilon$ is an infinitesimal generator field of \eqref{e4.20}.
\end{exam}

\subsection{Infinitesimal Generator Lemma}

 Now, let us state the key Infinitesimal Generator Lemma which plays a central role in our proof of uniqueness under below.

  \begin{lemm}\label{l4.2}(Infinitesimal Generator Lemma) Letting $h_K$ be a solution of \eqref{e1.2} and assuming that $\phi\in E^{(*)}(K)$ is a nontrivial solution to linearized equation \eqref{e4.14}, then the function
     $$
      h_\varepsilon(x)=\int^\varepsilon_0\phi(x) d\varepsilon+h_K(x), \ \ \forall x\in{\S}^n
     $$
  is an approximation solution of \eqref{e4.12}. Moreover, the assignedvector field
    \begin{equation}\label{e4.23}
     \upsilon=-\frac{\sqrt{1+|y|^2}}{D_1u_K(y)}\phi\left(\frac{y}{\sqrt{1+|y|^2}},-\frac{1}{\sqrt{1+|y|^2}}\right)\frac{\partial}{\partial y_1}
    \end{equation}
  from the special mapping representation in Lemma \ref{l4.1} is an infinitesimal generator field of \eqref{e4.11}. Conversely, given an infinitesimal generator field of \eqref{e4.11} corresponding to the group action $g_\varepsilon=\left[
    \begin{array}{c}
      \psi_\varepsilon(y,u)\\
      \eta_\varepsilon(y,u)
    \end{array}
    \right]$, the functions determined by parametric system
     \begin{equation}
       u_\varepsilon(y_\varepsilon):\begin{cases}
       y_\varepsilon=\psi_\varepsilon(y,u(y))\\
      u_\varepsilon=\eta_\varepsilon(y,u(y))
       \end{cases}
     \end{equation}
    with respect to parameter $\varepsilon\in{\R}$ are all solutions of \eqref{e4.11} for each $\varepsilon$.
  \end{lemm}
  
\begin{rema}\label{r4.3}
 In the Lie's theory, the formation of the infinitesimal generator field usually is in genreal the expression of $\upsilon=\xi^i\frac{\partial}{\partial y_i}+\eta\frac{\partial}{\partial u}$. In the formula \eqref{e4.23} of Lemma \ref{l4.2}, the fact that the infinitesimal generator $\upsilon$ has a simple form is due to the fact that we have chosen a special mapping representation from Lemma \ref{l4.1}. Actually, as is well known (the readers may refer to the book \cite{O00}, Page 30, Proposition 1.29) any locally non-degenerate vector could be in a simplest form $\upsilon=\frac{\partial}{\partial y_1}$ after choosing some new coordinates system.
\end{rema}

\begin{proof}
  At first, in order to show $h_\varepsilon$ to be an approximation solution of \eqref{e4.12}, one needs only to check that $\varepsilon\sim||h_\varepsilon-h_K||_{L^2({\S}^n)}$ and
      \begin{eqnarray*}
        &&\frac{d}{d\varepsilon}\bigg|_{\varepsilon=0}\left[\det(\nabla^2h_\varepsilon+h_\varepsilon I)-h_\varepsilon^{p-1}\right]\\
        &&\hskip40pt =U^{ij}_K(\phi_{ij}+\phi\delta_{ij})-(p-1)h_K^{p-2}\phi=0.
      \end{eqnarray*}
 Now, let us to show that the assigned vector field $\upsilon$ based on Mapping Representation pair in Lemma \ref{l4.1} will be given by \eqref{e4.23}. Actually, using the semi-spherical projection
    \begin{equation}\label{e4.25}
      u_\varepsilon(y_\varepsilon)=\sqrt{1+|y_\varepsilon|^2}h_\varepsilon\left(\frac{y_\varepsilon}{\sqrt{1+|y_\varepsilon|^2}},-\frac{1}{\sqrt{1+|y_\varepsilon|^2}}\right):=g_\varepsilon(y_\varepsilon)
    \end{equation}
  in Subsection 4.2, one can project the approximation solution $h_\varepsilon$ of \eqref{e4.12} onto the approximation solution $u_\varepsilon$ of
    \begin{equation}\label{e4.26}
      \det D^2u_\varepsilon=(1+|y_\varepsilon|^2)^{-\frac{p+n+1}{2}}u_\varepsilon^{p-1}+o(\varepsilon), \ \ \forall y_\varepsilon\in{\mathbb{R}}^n
    \end{equation}
  Comparing with the semi-spherical projection
   \begin{equation}\label{e4.27}
     u=u_K(y)=\sqrt{1+|y|^2}h_K\left(\frac{y}{\sqrt{1+|y|^2}},-\frac{1}{\sqrt{1+|y|^2}}\right):=f(y)
   \end{equation}
  for solution $h_K$ of \eqref{e1.2}, the mapping representation $g_\varepsilon=(\psi_\varepsilon(y,u),\eta_\varepsilon(y,u))$ defined in Lemma \ref{l4.1} is given by
    \begin{equation}\label{e4.28}
      \begin{cases}
        y_{\varepsilon,2}=y_2, \ y_{\varepsilon,3}=y_3, \ \cdots, y_{\varepsilon,n}=y_n,\\
        y_{\varepsilon,1}=G_\varepsilon(f(y), y_2, \cdots, y_n),\\
        u_\varepsilon=u,
      \end{cases}
    \end{equation}
  where the function $y_{\varepsilon,1}=G_\varepsilon(v, y_{\varepsilon,2},\cdots, y_{\varepsilon,n})$ is determined as in the proof of Lemma \ref{l4.1} by $v=g_\varepsilon(y_{\varepsilon})$ and implicit theorem. Now, let us carry out explicit calculations as following. Noting that $G_\varepsilon(g_\varepsilon(y), z_y,\cdots, y_n)\equiv y_1$, we have  
     $$
       \frac{d}{d\varepsilon}\Big|_{\varepsilon=0} G_\varepsilon+D_1 G\cdot \frac{d}{d\varepsilon}\Big|_{\varepsilon=0} g_\varepsilon=0
     $$
   or equivalently
     \begin{eqnarray}\nonumber
      \frac{d}{d\varepsilon}\Big|_{\varepsilon=0} G_\varepsilon&=&-D_1G\cdot \frac{d}{d\varepsilon}\Big|_{\varepsilon=0} g_\varepsilon (z)=-\frac{1}{D_1g}\cdot \frac{d}{d\varepsilon}\Big|_{\varepsilon=0} g_\varepsilon\\
      &=&-\frac{\sqrt{1+|y|^2}}{D_1u_K}\phi\left(\frac{y}{\sqrt{1+|y|^2}},-\frac{1}{\sqrt{1+|y|^2}}\right),
     \end{eqnarray}
   where we have used $g(y)=u_K(y)$ and
     $$
      g_\varepsilon(y)=\sqrt{1+|y|^2}h_\varepsilon\left(\frac{y}{\sqrt{1+|y|^2}},-\frac{1}{\sqrt{1+|y|^2}}\right).
     $$
   So, \eqref{e4.23} follows.
  
  To show $\upsilon$ is an infinitesimal generator field of \eqref{e4.11}, after taking differentiation on
    \begin{eqnarray}\nonumber
      \Delta\left(y_\varepsilon,u_\varepsilon, u_\varepsilon^{(1)},u_\varepsilon^{(2)}\right)&=&\det D^2u_\varepsilon-(1+|y_\varepsilon|^2)^{-\frac{p+n+1}{2}}u_\varepsilon^{p-1}\\
      &\equiv& o(\varepsilon), \ \ \forall \varepsilon, y
    \end{eqnarray}
with respect to $\varepsilon$, we obtain \eqref{e4.19} by Chain rule. One may also refer to the book by Olver (\cite{O00}, Chapter 2) for detailed arguments. At there, the converse part of Lemma \ref{l4.2} can be found by one-to-one corresponding of one-parameter group actions with infinitesimal generator. In here, the constructed local transformation also satisfies the hypotheses of the infinitesimal criterion in (\cite{O00}, Chapter 2).
\end{proof}

 Infinitesimal Generator Lemma \ref{l4.2} establishes the relation between the Kernel of the linearized equation \eqref{e4.14} with the infinitesimal generators of the equation \eqref{e4.11}. The next proposition classifies all infinitesimal generator of \eqref{e4.11} by the comparing the like terms in \eqref{e4.19}, which was shown firstly by the author and his doctor student Huan-Jie Chen in \cite{CD25}. (The readers may also refer to the book \cite{O00} for some detailed discussion)

  \begin{prop}\label{p4.1}(\cite{CD25}, Theorem 1.1)
    Considering $n\geq2$ and \eqref{e4.11} for $p\not=n+1,1,-n-1$, the possible infinitesimal generators $\upsilon=\xi^i\frac{\partial}{\partial y_i}+\zeta\frac{\partial}{\partial u}$ of \eqref{e4.11} are given by
      \begin{equation}\label{e4.31}
       \begin{cases}
        \displaystyle \xi^i=\sum_{j=1}^n(C_jy_jy_i+E_{ij}y_j)+C_i,\\
        \displaystyle \zeta=\sum_{j=1}^nC_jy_ju,
       \end{cases}
      \end{equation}
    where $[E_{ij}]$ is an anti-symmetric matrix and $C_i, i=1,2,\ldots,n$ are constants.
  \end{prop}

\medskip

\section{{\it A-priori} estimates and openness of the uniqueness set}

\noindent Now, we turn to explore the topological structure of the uniqueness set defined by $\Gamma=\{p\in(-2n-5,-n-1)| {\mathcal{C}}_{n,p}=\{{\S}^n\}\}$. At the beginning, let us prove the following openness result.

\begin{prop}\label{p5.1}
  For dimension $n\geq2$, the uniqueness set $\Gamma$ is relative open in $(-2n-5,-n-1)$.
\end{prop}

\noindent Before our arguments, let us impose the following spectral lemma of the linearized eigen-problem \eqref{e2.1} at the unit sphere $K={{\S}^n}$.

\begin{lemm}\label{l5.1}
 For all dimensions $n\geq2$, the $j-$th eigenvalue of linearized eigen-problem \eqref{e2.1} at the unit sphere ${\S}^n$ is given by
   \begin{equation}\label{e5.1}
    \begin{cases}
    \lambda_1({\S}^n)=-n,\\
    \lambda_2({\S}^n)=0,\\
    \lambda_3({\S}^n)=n+2,\\
    \lambda_4({\S}^n)=2n+6, \ldots, \\
    \lambda_j({\S}^n)=(j-2)(n+j-1), \ \ \forall j\geq5.
    \end{cases}
   \end{equation}
\end{lemm}

\begin{proof}[Proof of Lemma \ref{l5.1}] When $K={\S}^n$, the eigenvalue problem \eqref{e2.1} changes to
  \begin{equation}
   -\Delta\phi-n\phi=\lambda\phi, \ \ \forall x\in{\mathbb{S}}^n.
  \end{equation}
As is well known, the eigenfunctions of Laplace-Beltrami operator $-\Delta$ on the sphere are given by homogeneous harmonic polynomials $v(y)=|y|^{j-1}\phi\Big(\frac{y}{|y|}\Big)$ of degree $j-1$ on ${\mathbb{R}}^{n+1}$, whose eigenvalues are given exactly by $\lambda_{0,j}\equiv (j-1)(n+j-2)$ for positive integer $j$. In a paper of Kazdan (\cite{K98}, page 12), the multiplicities of the eigenvalues were also calculated explicitly. Therefore, \eqref{e5.1} follows from $\lambda_j({\S}^n)=\lambda_{0,j}-n$.
\end{proof}

Now, we can prove Proposition \ref{p5.1} as a corollary of Theorem \ref{t1.5} and Lemma \ref{l5.1}.

\begin{proof}[Proof of Proposition \ref{p5.1}] For each $p_0\in\Gamma_\sigma$, we need to show that there exists a small constant $\sigma_0>0$ such that $(p_0-\sigma_0,p_0+\sigma_0)\in\Gamma_\sigma$. Suppose on the contrary, there exists a sequence of $p=p_j, j\in{\mathbb{N}}$ converges to $p_0$ such that \eqref{e1.2} admits a sequence of non-constant solutions $h_j=h_{p_j}$ for each $j\in{\mathbb{N}}$. By {\it a-priori} estimate Theorem \ref{t1.5}, for a subsequence, $h_{p_j}$ tends to a limiting function $h_{p_0}$ as $j$ tends to infinity. The limiting function $h_{p_0}$ must be identical to the constant solution since $p_0\in\Gamma$. Subtracting the equation
   \begin{equation}
     \det(\nabla^2h_p+h_pI)=h_p^{p-1}, \ \ \forall x\in{\mathbb{S}}^n
   \end{equation}
of $h_p, p=p_j$ from the equation of constant solution $h_{p_0}\equiv1$, we obtain that
  \begin{equation}\label{e5.4}
    \int^1_0U^{ij}_tdt(\nabla^2_{ij}\phi_{t}+\phi_t\delta_{ij})=(p-1)h_t^{p-1}\phi_t,\ \ \forall x\in{\mathbb{S}}^n
  \end{equation}
for $h_t\equiv th_p+(1-t)$ and $[U^{ij}_t]$ stands for the cofactor matrix of $A_t\equiv[\nabla^2_{ij}h_t+h_t\delta_{ij}]$, where $\phi_t\equiv\frac{h_p-1}{||h_p-1||_{L^{2}({\mathbb{S}}^n)}}$. Sending $p_j\to p_0$ and using the elliptic estimates for \eqref{e5.4}, we conclude that $\phi_t$ converges to a nontrivial function $\phi$ satisfying
   \begin{equation}\label{e5.5}
     -\Delta \phi-n\phi=(1-p)\phi, \ \ \forall x\in{\mathbb{S}}^n.
   \end{equation}
However, when $p\in(-2n-5,-n-1)$, $\lambda=n+1-p\in(2n+2,3n+6)$ is not an eigenvalue of linearized eigen-problem \eqref{e2.1} by Lemma \ref{l5.1}. A contradiction occurs and hence the conclusion holds true.
\end{proof}

Therefore, proving of Proposition \ref{p5.1} is reduced to proving of Theorem \ref{t1.5}. Under below, we will decompose the proof of {\it a-priori} Theorem \ref{t1.5} by several steps.

\subsection{Ellipsoid lemmas and {\it a-priori} estimates I}

\noindent From now on, we use the notation $E(\mu,z_0)$ to denote the ellipsoid defined by
   $$
     E(\mu,z_0)\equiv\left\{z\in{\mathbb{R}}^{n+1}\bigg|\ \sum_{i=1}^{n+1}\frac{(z_i-z_{0,i})^2}{\mu_i^2}=1\right\},
    $$
where the centroid is given by $z_0=(z_{0,1},\ldots,z_{0,n+1})$ and the semi-axis lengths are given by $\mu=(\mu_1,\ldots,\mu_{n+1}), \ \ 0<\mu_1\leq\mu_2\leq\ldots\leq\mu_{n+1}$. For $z_0=0$, we simply write $E(\mu)=E(\mu,0)$ for short. By John's lemma, there exists some universal constant $C_n>1$ such that for any convex body $K$ containing the origin, we can always assume that
   \begin{equation}\label{e5.6}
     C^{-1}_n E(\mu,z_0)\subset K\subset C_n E(\mu,z_0)
   \end{equation}
holds for some ellipsoid $E(\mu,z_0)$ by rotation if necessary. In this subsection, we will prove the following {\it a-priori} estimate for solutions of \eqref{e1.2}, which is a crucial part in establishing the {\it a-priori} estimates.

\begin{prop}\label{p5.2}
  Considering \eqref{e1.2} with $n\geq2$ and $p<-n-1$, there exists a universal constant $C_{n}>0$ depending only on $n$, such that
    \begin{equation}\label{e5.7}
     {Vol(K)}/\mu_1^{\frac{p+n+1}{2}}\geq 2^pC_{n}^{p-1}
    \end{equation}
  holds for all solutions $K\in{\mathcal{C}}_{n,p}$, assuming the John's normalized inclusion \eqref{e5.6}.
\end{prop}

 The proof of Proposition \ref{p5.2} will be divided by several fundamental ellipsoid lemmas which would be useful in this and the next subsections. We have a first lemma.

\begin{lemm}\label{l5.2}
  For dimension $n\geq2$, the support function $h(\mu,z_0)$ of the ellipsoid $E(\mu,z_0)$ for $\mu_i>0, i=1,\ldots,n+1$ and $z_0\in{\R}^{n+1}$ satisfies that
     \begin{eqnarray}\nonumber\label{e5.8}
      \Delta_1&:=&\det(\nabla^2h(\mu,z_0)+h(\mu,z_0)I)\\
       &&-\left(\Pi_{i=1}^{n+1}\mu_i^2\right)\left(h(\mu,z_0)-\langle z_0,x\rangle\right)^{-n-2}=0, \ \ \forall x\in{\S}^n.
     \end{eqnarray}
 Consequently, each ellipsoid $E(\mu,z_0)$ is not a solution of \eqref{e1.2}, except the origin-centering unit sphere ${\S}^{n}$.
\end{lemm}

\begin{proof} For arbitrary $\mu$, after scaling $E(\mu,0)$ by $E(\nu,0)=E(\mu,0)\Pi_{i=1}^{n+1}\mu_i^{-1/(n+1)}$, there holds $\Pi_{i=1}^{n+1}\nu_i=1$. As is well known, $E(\nu,0)$ satisfies \eqref{e1.2} for $p=-n-1$ (actually, this follows from the characterization of origin-centered ellipsoids as the unique smooth solutions of the centro-affine Minkowski problem, see for example \cite{Ca72}). So, we conclude that $h_\mu$ is a solution to the equation
  \begin{eqnarray}\nonumber\label{e5.9}
     \Delta_2&:=&\det(\nabla^2h_\mu+h_\mu I)\\
     &&-\left(\Pi_{i=1}^{n+1}\mu_i^2\right)h_\mu^{-n-2}=0, \ \ \forall x\in{\S}^n.
  \end{eqnarray}
For ellipsoid $E(\mu,z_0)$ with centroid $z_0$, the support function is given by
  \begin{eqnarray*}
    h_{E(\mu,z_0)}&=&h_{E(\mu,0)}+\langle z_0,x\rangle\\
    &=&\sqrt{\sum_{i=1}^{n+1}\mu_i^2x_i^2}+\langle z_0,x\rangle, \ \ \forall x\in{\S}^n.
  \end{eqnarray*}
Thus, we conclude that $E(\mu,z_0)$ satisfies \eqref{e5.8} since
   $$
    \nabla^2h_{E(\mu,z_0)}+h_{E(\mu,z_0)}I=\nabla^2h_\mu+h_\mu I
   $$
by Gauss-Weingarten's relation $\nabla^2x+xI=0$. So, all ellipsoids $E(\mu,z_0)$ are not solutions of \eqref{e1.2} unless it is the origin-centering unit sphere ${\S}^n$. We have shown the conclusion of the lemma.
\end{proof}

\begin{lemm}\label{l5.3}
 Let $Vol(B_1)$ denote the volume of the unit ball ${B_1}$, we have
   \begin{equation}\label{e5.10}
      \int_{{\mathbb{S}}^n}h_\mu^{-n-1}=(n+1)\frac{Vol(B_1)}{\mu_1\mu_2\ldots\mu_{n+1}}
   \end{equation}
 holds for support function $h_\mu$ of the ellipsoid $E(\mu)$. Moreover, for each $p<-n-1$, there exists a positive constant $C_{n}$ depending only on $n$ such that there holds
   \begin{equation}\label{e5.11} 2^{p}/C_{n}\leq\frac{\mu_1\mu_2\ldots\mu_{n+1}}{\mu_1^{p+n+1}}\int_{{\mathbb{S}}^n}h_\mu^p\leq (n+1)Vol(B_1).
   \end{equation}
\end{lemm}

\begin{proof} As in the proof of Lemma \ref{l5.2}, it follows from  the volume identity
    $$
     Vol(K)=\frac{1}{n+1}\int_{{\mathbb{S}}^n} h_K\det(\nabla^2h_K+h_KI)
    $$
together with \eqref{e5.8} that
  \begin{eqnarray*}
   \int_{{\mathbb{S}}^n}h_\mu^{-n-1}&=&(n+1)\frac{Vol(E(\mu))}{\mu_1^2\mu_2^2\ldots\mu_{n+1}^2}\\
   &=&(n+1)\frac{Vol(B_1)}{\mu_1\mu_2\ldots\mu_{n+1}}
  \end{eqnarray*}
and thus \eqref{e5.10} follows. R.H.S. of \eqref{e5.11} can be verified by applying \eqref{e5.10} to deduce
  \begin{eqnarray*}
   \int_{{\mathbb{S}}^n}h^p_\mu&\leq&\mu_1^{p+n+1}\int_{{\mathbb{S}}^n}h_\mu^{-n-1}\\
     &=&(n+1)\frac{Vol(B_1)\mu_1^{p+n+1}}{\mu_1\mu_2\ldots\mu_{n+1}}.
  \end{eqnarray*}
To show L.H.S. of \eqref{e5.11}, defining the integration
  $$
   {\mathcal{J}}^1(p,t,\mu,n)=\int_{\partial B^n_1}(t+\mu_1^2x_1^2+\ldots+\mu_n^2x_n^2)^{p/2}
  $$
for $n$ dimensional unit ball $B^n_1$, nonnegative $t$ and $0\leq\mu_1\leq\ldots\leq\mu_n$, we claim that
   \begin{equation}\label{e5.12}
     {\mathcal{J}}^1(p,t,\mu,n)\geq\frac{2^p(t+\mu_1^2)^{(p+n)/2}}{C_n\sqrt{(t+\mu_1^2)(t+\mu_2^2)\ldots(t+\mu_{n+1}^2)}}
   \end{equation}
holds for some universal constant $C_{n}>0$ depending only on $n$. Actually, when $n=2, \mu=(\mu_1,\mu_2)$, we have
   \begin{eqnarray}\nonumber\label{e5.13}
     {\mathcal{J}}^1(p,t,\mu,2)&=&4\int^{\pi/2}_0\left[t+\mu_1^2+(\mu_2^2-\mu_1^2)\sin^2\theta\right]^{p/2}d\theta\\
     &\geq& 4(t+\mu_1^2)^{p/2}\int^{\pi/2}_0\left(1+\frac{\mu_2^2-\mu_1^2}{t+\mu_1^2}\theta^2\right)^{p/2}d\theta\\ \nonumber
     &\geq&4(t+\mu_1^2)^{p/2}\sqrt{\frac{t+\mu_1^2}{\mu_2^2-\mu_1^2}}\int^{\frac{\pi}{2}\sqrt{\frac{\mu_2^2-\mu_1^2}{t+\mu_1^2}}}_0(1+\alpha^2)^{p/2}d\alpha
   \end{eqnarray}
if $\mu_1\not=\mu_2$. When $\mu_1=\mu_2$, \eqref{e5.12} is clearly true by the first inequality of \eqref{e5.13}. Supposing that $\mu_1<\mu_2$ and noting that for $\frac{\mu_2^2-\mu_1^2}{t+\mu_1^2}\leq1$, we have $t+\mu_2^2\leq2(t+\mu_1^2)$. Then, it follows from \eqref{e5.13} that
  \begin{eqnarray*}
    {\mathcal{J}}^1(p,t,\mu,2)&\geq&2\pi2^{p/2}(t+\mu_1^2)^{p/2}\sqrt{\frac{t+\mu_1^2}{\mu_2^2-\mu_1^2}}\sqrt{\frac{\mu_2^2-\mu_1^2}{t+\mu_1^2}}\\
  &\geq&2\sqrt{2}\pi2^{p/2}\frac{(t+\mu_1^2)^{(p+2)/2}}{\sqrt{t+\mu_1^2}\sqrt{t+\mu_2^2}}
  \end{eqnarray*}
and hence \eqref{e5.12} follows. Now, assuming $\frac{\mu_2^2-\mu_1^2}{t+\mu_1^2}>1$, then $1/2<\frac{\mu_2^2-\mu_1^2}{t+\mu_2^2}\leq1$. So, it follows from \eqref{e5.13} that
   $$
    {\mathcal{J}}^1(p,t,\mu,2)\geq C_n^{-1}\sqrt{-p}(t+\mu_1^2)^{p/2}\sqrt{\frac{t+\mu_1^2}{\mu_2^2-\mu_1^2}}
   $$
and thus \eqref{e5.12} follows. When dimension $n\geq3$, we use
  \begin{eqnarray*}
   {\mathcal{J}}^1(p,t,\mu,n)&=&(t+\mu_1^2)^{p/2}\int_{\partial B_1^n}\left(1+\frac{\mu_2^2-\mu_1^2}{t+\mu_1^2}x_2^2+\ldots+\frac{\mu_n^2-\mu_1^2}{t+\mu_1^2}x_n^2\right)^{p/2}\\
     &=&(t+\mu_1^2)^{p/2}\int^1_{-1}(1-x_1^2)^{(p+n-2)/2}\left(\int_{\partial B^{n-1}_1}(t'+\nu_2^2x_2^2+\ldots+\nu_n^2x_n^2)^{p/2}\right)dx_1\\
     &=&(t+\mu_1^2)^{p/2}\int^1_{-1}(1-x_1^2)^{(p+n-2)/2}{\mathcal{J}}^1(p,t',\nu,n-1)dx_1
  \end{eqnarray*}
to perform mathematical induction, where
   $$
    t'\equiv\frac{1}{1-x_1^2}, \ \ \nu_i\equiv\sqrt{\frac{\mu_i^2-\mu_1^2}{t+\mu_1^2}}, \ \ i=2,3,\ldots,n.
   $$
By induction hypothesis for $n-1$, we have
  \begin{eqnarray}\nonumber\label{e5.14}
   {\mathcal{J}}^1(p,t,\mu,n)&\geq&C_n^{-1}2^{p/2}\int^1_{-1}(1-x_1^2)^{\frac{n-2}{2}}y_2^{\frac{p+n-2}{2}}\frac{dx_1}{\sqrt{y_3\ldots y_n}}\\
   &\geq&\frac{2^{p/2}}{C_n\Pi_{i=2}^{n}\sqrt{t+\mu_i^2}}\int^1_{-1}(1-x_1^2)^{\frac{n-2}{2}}\big(t+(1-x_1^2)\mu_1^2+x_1^2\mu_2^2\big)^{\frac{p+n-1}{2}}dx_1
  \end{eqnarray}
for $y_i\equiv t+(1-x_1^2)\mu_1^2+x_1^2\mu_i^2, i=2,3,\ldots,n$. We claim that
  \begin{eqnarray}\nonumber\label{e5.15}
   {\mathcal{J}}^2(\mu_1,\mu_2)&:=&\int^1_{-1}(1-x_1^2)^{\frac{n-2}{2}}\big(t+(1-x_1^2)\mu_1^2+x_1^2\mu_2^2\big)^{\frac{p+n-1}{2}}dx_1\\
    &\geq&C_n^{-1}2^{\frac{p+n-1}{2}}(t+\mu_1^2)^{\frac{p+n-1}{2}}.
  \end{eqnarray}
Actually, reformulating ${\mathcal{J}}^2(\mu_1,\mu_2)$ by
 \begin{eqnarray*}
  {\mathcal{J}}^2(\mu_1,\mu_2)&=&2(t+\mu_1^2)^{\frac{p+n-1}{2}}\int^1_0(1-x_1^2)^{\frac{n-2}{2}}\left(1+\frac{\mu_2^2-\mu_1^2}{t+\mu_1^2}x_1^2\right)^{\frac{p+n-1}{2}}dx_1\\
  &=&2(t+\mu_1^2)^{\frac{p+n-1}{2}}\int^{\pi/2}_0\cos^{n-1}\theta\left(1+\frac{\mu_2^2-\mu_1^2}{t+\mu_1^2}\sin^2\theta\right)^{\frac{p+n-1}{2}}d\theta\\
  &\geq&2(t+\mu_1^2)^{\frac{p+n-1}{2}}\int^{\pi/2}_0(1-\theta^2)^{\frac{n-1}{2}}\left(1+\frac{\mu_2^2-\mu_1^2}{t+\mu_1^2}\theta^2\right)^{\frac{p+n-1}{2}}d\theta,
 \end{eqnarray*}
one can verify \eqref{e5.15} as in the proof of the first claim \eqref{e5.12} for $n=2$. Substituting \eqref{e5.15} into \eqref{e5.13}, we obtain the desired inequality \eqref{e5.12} by mathematical induction. Now, L.H.S. of \eqref{e5.11} is a special case of \eqref{e5.12} by taking $t=0$ and replacing $n$ by $n+1$. The proof of the lemma has been completed.
\end{proof}

\begin{proof}[Proof of Proposition \ref{p5.2}] By John's lemma, one may assume that
   \begin{equation}\label{e5.16}
     C_n^{-1}E(\mu,z_0)\subset K\subset C_nE(\mu,z_0)
   \end{equation}
holds for some ellipsoid $E(\mu,z_0)$ satisfying $0<\mu_1\leq\mu_2\leq\ldots\leq\mu_{n+1}$ and some $z_0$, by rotation if necessary. Noting that by Minkowski volume identity,
   \begin{eqnarray*}
     (n+1){Vol(K)}&=&\int_{{\mathbb{S}}^n}h_K\det(\nabla^2h_K+h_KI)=\int_{{\mathbb{S}}^n}h_K^p\\
       &\geq& C_{n}^{p}\int_{{\mathbb{S}}^n}h_{C_nE(\mu,z_0)}^p\\ \nonumber
      &\geq& C_n^p\int_{{\S}^n}h_{C_nE(\mu)}^p\\
      &\geq& C_n^{p-1}2^p\frac{\mu_1^{p+n+1}}{\mu_1\mu_2\ldots\mu_{n+1}}\\
      &\geq&2^pC_{n}^{p-1}\mu_1^{p+n+1}{Vol(K)}^{-1}
   \end{eqnarray*}
holds by Lemma \ref{l5.3} and the fact that the function $\int_{{\mathbb{S}}^n}h_{E(\mu,z_0)}^p$ of $z_0$ attains its unique minimum at the origin, the desired inequality \eqref{e5.7} holds true. The proof is completed.
\end{proof}

\subsection{{\it A-priori} estimates II and proof of Theorem \ref{t1.5}}

  In this subsection, we turn to show the other {\it a-priori} estimate lemmas and then prove the Main Theorem \ref{t1.5} in the supercritical range. The following proposition shows that once lower or upper bounds obtained, then the volume ${Vol(K)}$, $m\equiv\min h_K$, $M\equiv\max h_K$ are all bounded.

\begin{prop}\label{p5.3}
  Letting $n\geq2$, there exists positive constant $C_{n}$ depending only on $n$, such that
    \begin{equation}\label{e5.17}
    \begin{cases}
     M\leq C_{n}^{2-p}m^{p-n}, & Vol(K)\leq C_{n}^{1-p}m^p,\\
     m^{p+n}\leq C_n(-p)^nM^{2n+1}, & Vol(K)\leq C_{n}M^{n+1}
    \end{cases}
   \end{equation}
  holds for solution $K\in{\mathcal{C}}_{n,p}$ of \eqref{e1.2} for each $p<0$.
\end{prop}

\begin{proof} First, we claim that there exists a universal constant $C_n>0$, such that
   \begin{equation}\label{e5.18}
    {Vol(K)}\leq C_n^{1-p}m^{p}.
   \end{equation}
Actually, by John's lemma, after rotation if necessary, one may assume that
   $$
     C_n^{-1}E(\mu,z_0)\subset K\subset C_nE(\mu,z_0)
   $$
holds for ellipsoid $E(\mu,z_0)$ with $0<\mu_1\leq\mu_2\leq\ldots\leq\mu_{n+1}$ and centroid $z_0$. Therefore, we have the comparison formula
  \begin{eqnarray}\nonumber\label{e5.19}
     |\partial K|&\simeq& C_n\sigma_{n}(\mu)\\
      &\simeq& C_n\sigma_{n+1}(\mu)/\mu_1\\ \nonumber
      &\leq& C_n{Vol(K)}/m,
  \end{eqnarray}
where we have used $m\leq C_n\mu_1$ and $V(K)\geq C_n^{-1}\mu_1\cdots\mu_{n+1}$ by John's normalization. Applying the H\"{o}lder's inequality, we have
  \begin{eqnarray}\nonumber\label{e5.20}
   {Vol(K)}&=&\frac{1}{n+1}\int_{{\mathbb{S}}^{n}}h_Kd\sigma_K\\
    &\leq&\frac{1}{n+1}\left(\int_{{\mathbb{S}}^{n}}h_K^{1-p}d\sigma_K\right)^{\frac{1}{1-p}}\left(\int_{{\mathbb{S}}^{n}}d\sigma_K\right)^{\frac{-p}{1-p}}\\ \nonumber
    &\leq& C_n{Vol(K)}^{\frac{-p}{1-p}}m^{\frac{p}{1-p}}.
  \end{eqnarray}
So, \eqref{e5.18} follows when $p<0$. Secondly, we claim that
 \begin{equation}\label{e5.21}
  Vol(K)\geq C_{n}^{-1}(-p)^{-n}M^{-n}m^{p+n}
 \end{equation}
holds for all $p<0$. Without loss of generality, one may assume that $m=\min h=h(x_0), x_0\in{\mathbb{S}}^{n}$ in spherical coordinates of ${\mathbb{S}}^{n}$ at south polar. By convexity, there holds
  \begin{equation}\label{e5.22}
    h_K(x)\leq m+C_1Mdist(x,x_0), \ \ \forall x\in{\mathbb{S}}^{n}
  \end{equation}
for geodesic distance $dist(x,x_0)$ and some positive constant $C_1$ depending only on $n$. Hence,
  \begin{eqnarray*}
   {Vol(K)}&=&\frac{1}{n+1}\int_{{\mathbb{S}}^{n}} h_K\det(\nabla^2h_K+h_KI)\\
   &=&\frac{1}{n+1}\int_{{\mathbb{S}}^{n}}h_K^{p}\\
   &\geq& C_2\int^{\pi/2}_0\frac{\theta^{n-1} d\theta}{(m+C_1M\theta)^{-p}}\\
   &=& C_2M^{-n}m^{p+n}\int^{M\pi/(2m)}_0\frac{t^{n-1}}{(1+C_1t)^{-p}}dt,
  \end{eqnarray*}
where $\theta\equiv angle_{{\R}^{n+1}}(x,x_0)$ and $C_2>0$ is another constant depending only on $n$. Noting that
 \begin{eqnarray*}
  \int^{M\pi/(2m)}_0\frac{t^{n-1}}{(1+C_1t)^{-p}}dt&\geq& \int^{(-pC_1)^{-1}}_0 e^{-1}t^{n-1}dt\\
   &\geq& C_n^{-1}(-p)^{-n},
 \end{eqnarray*}
the proof of \eqref{e5.21} has been shown.

 Now, drawing a cone $V$ with the base of a $n$ dimensional disc $B^{n}_m$ and with the height $M$, we have $V$ is contained inside $K$ by convexity. Thus, there holds
  \begin{equation}\label{e5.23}
   C_n^{-1}m^{n}M\leq {Vol(K)}\leq C_nM^{n+1}
  \end{equation}
 for another universal constant $C_n>0$. Combing \eqref{e5.18} with L.H.S. of \eqref{e5.23}, we obtain that
    $$
     \begin{cases}
      M\leq C_n^{2-p}m^{p-n}, \\
      {Vol(K)}\leq C_n^{1-p}m^p
     \end{cases}
    $$
for some universal constant $C_n>0$. Substituting \eqref{e5.21} into R.H.S. of \eqref{e5.23}, it yields that
    $$
     \begin{cases}
       m^{p+n}\leq C_n(-p)^nM^{2n+1}, \\
       {Vol(K)}\leq C_nM^{n+1}.
     \end{cases}
    $$
The proof of the proposition is completed.
\end{proof}

 Next, let us prove the following upper bound of the solutions of \eqref{e1.2}. This result was inspired by the work of B\"{o}r\"{o}czky-Saroglou \cite{BS23}.

\begin{prop}\label{p5.4}
  Letting $n\geq2, p\in{\mathcal{P}}\Subset(-\infty,-n-1)$, there exists positive constant $C_{n,{\mathcal{P}}}$ depending only on $n$ and compact subset ${\mathcal{P}}$, such that
    \begin{equation}\label{e5.24}
      M=\max_{x\in{\mathbb{S}}^n}h_K(x)\leq C_{n,{\mathcal{P}}}
    \end{equation}
  holds for solutions $h_K$ of \eqref{e1.2}.
\end{prop}

\begin{proof}
 Using again the John's lemma, one may assume that
   $$
     C_n^{-1}E(\mu,z_0)\subset K\subset C_nE(\mu,z_0)
   $$
holds for some ellipsoid $E(\mu,z_0)\subset{\R}^{n+1}$ by rotation if necessary, where
   $$
    \begin{cases}
    \mu=(\mu_1,\mu_2,\ldots, \mu_{n+1}), \\ z_0=(z_{0,1},z_{0,2},\ldots,z_{0,n+1}), \\
    0<\mu_1\leq \mu_2\leq\ldots\leq \mu_{n+1}.
    \end{cases}
   $$
Letting $K^{(j)}$ be a sequence of solutions of \eqref{e1.2} for $p^{(j)}$ and with equivalent ellipsoid $E(\mu^{(j)},z^{(j)}_0)$, one may assume that
   \begin{equation}\label{e5.25}
   \begin{cases}
     \mu^{(j)}_{1}\sim\mu^{(j)}_{2}\sim\ldots\sim\mu^{(j)}_{k}, \\ \lim_{j\to\infty}\frac{\mu^{(j)}_1}{\mu^{(j)}_{k+1}}=\ldots=\lim_{j\to\infty}\frac{\mu^{(j)}_1}{\mu^{(j)}_{n+1}}=0
   \end{cases}
   \end{equation}
holds for some $k\in[1,n+1]$. We claim that $k\leq n$, unless $K^{(j)}, j\in{\mathbb{N}}$ are uniformly bounded. Actually, if $k=n+1$ and $K^{(j)}$ are not uniformly bounded, then we have
   \begin{equation}
     \mu^{(j)}_{1}\sim\mu^{(j)}_{2}\sim\ldots\sim\mu^{(j)}_{n+1}\sim M_j\gg1, \ \ \forall j.
   \end{equation}
By Proposition \ref{p5.3}, we know that $\lim_{j\to\infty}m_j=0$ for $m_j=\min h_{K^{(j)}}$. Then, after rescaling $L^{(j)}=M_j^{-1}K^{(j)}$ and using the Blaschke's selection axiom, there exists a limiting convex body $L^{(\infty)}$ satisfying $0\in\partial L^{(\infty)}$ (since $m_j/M_j\to0$ as $j\to\infty$), and
    \begin{eqnarray}\nonumber\label{e5.27}
      1/C_\infty&\leq&\min_{{\mathbb{S}}^n}W_{L^{(\infty)}}\\
       &\leq&\max_{{\mathbb{S}}^n}W_{L^{(\infty)}}\\ \nonumber
       &\leq& C_\infty
    \end{eqnarray}
for width function $W_{L^{(\infty)}}$ of $L^{(\infty)}$ and some positive constant $C_\infty$. Moreover, setting
 $$
  \omega\equiv\{x\in{\mathbb{S}}^n| h_{L^{(\infty)}}(x)>0\}, \ \ \Sigma\equiv G^{-1}(\omega)
 $$
and passing to the limit, we deduce from $\det(\nabla^2h_{L^{(j)}}+h_{L^{(j)}}I)=M_j^{p-n-1}h_{L^{(j)}}^{p-1}\to 0$ on $\omega$ that
   \begin{equation}\label{e5.28}
     \det(\nabla^2h_{L^{(\infty)}}+h_{L^{(\infty)}}I)=0, \ \ \forall x\in\omega.
   \end{equation}
Consequently, it follows from \eqref{e5.28} that
 $$
  (n+1)Vol(L^{(\infty)})=\int_{\omega} h_{L^{(\infty)}}\det(\nabla^2h_{L^{(\infty)}}+h_{L^{(\infty)}}I)=0,
 $$
which contradicts with \eqref{e5.27}. The proof of the claim was done.

 Setting $T^{(j)}\in GL(n+1)$ by
    $$
     T^{(j)}\equiv diag(1/\mu^{(j)}_1,\ldots,1/\mu^{(j)}_1, 1/\mu^{(j)}_{k+1},\ldots,1/\mu^{(j)}_{n+1})
    $$
and denoting $L^{(j)}\equiv T^{(j)}K^{(j)}$, there exists a limiting convex body $L^{(\infty)}$ of $L^{(j)}$ satisfying
  \begin{equation}\label{e5.29}
    C_n^{-1}B_1\left(z^{(\infty)}\right)\subset L^{(\infty)}\subset C_nB_1\left(z^{(\infty)}\right)
  \end{equation}
for another universal constant $C_n>0$. We need to impose the changing of variables formula
   \begin{equation}\label{e5.30}
    \int_{{\mathbb{S}}^n}\varphi(x)dS_{p^{(j)}}(L^{(j)})=|\det T^{(j)}|^2\int_{{\mathbb{S}}^n}\varphi(x)||T^{(j)}x||^{-(n+1+p)}dS_{p^{(j)}}(K^{(j)})
   \end{equation}
 (Lemmas 3.2-3.4, \cite{BS23}) established by B\"{o}r\"{o}czky-Saroglou. Note that by Proposition \ref{p5.2} and $m_j\to0$ from Proposition \ref{p5.3},
  \begin{equation}\label{e5.31}
    |\det T^{(j)}|^2(\mu_1^{(j)})^{n+1+p}\leq CVol^{-2}(K^{(j)})(\mu_1^{(j)})^{n+1+p}\leq C, \ \ \forall j
  \end{equation}
 is satisfied for $p<-n-1$. Moreover, when $p<-n-1$, there also holds
   \begin{equation}\label{e5.32}
    ||T^{(j)}x||^{-(n+1+p)}\leq C(\mu_1^{(j)})^{n+1+p}, \ \ \forall j.
   \end{equation}
We conclude from \eqref{e5.25}, \eqref{e5.30}-\eqref{e5.32} and weak continuity of the $L_p$ surface area measure under Hausdorff convergence that
 $$
   dS_{p^{(\infty)}}(L^{(\infty)})\equiv\lim_{j\to\infty}dS_{p^{(j)}}(L^{(j)})
 $$
is a Radon measure supported on ${\mathbb{S}}^n\cap H_k$ for $k$ dimensional subspace $H_k\equiv\{x\in{\mathbb{R}}^{n+1}| x_{k+1}=x_{k+2}=\ldots=x_{n+1}=0\}$ of ${\mathbb{R}}^{n+1}$. Taking any orthonormal matrix $O\in O(n+1)$ keeping the orthonormal subspace $H_k^\perp$ invariant, it follows from \eqref{e5.30} that
  \begin{eqnarray*}
   \int_{{\mathbb{S}}^n}\varphi(Ox)dS_{p^{(\infty)}}(L^{(\infty)})&=&\lim_{j\to\infty}
    \int_{{\mathbb{S}}^n}\varphi(Ox)dS_{p^{(j)}}(L^{(j)})\\
    &=&\lim_{j\to\infty}|\det (T^{(j)}\circ O)|^2\int_{{\mathbb{S}}^n}\varphi(x)||T^{(j)}\circ Ox||^{-(n+1+p)}dS_{p^{(j)}}(K^{(j)})\\
    &=&\lim_{j\to\infty}|\det T^{(j)}|^2\int_{{\mathbb{S}}^n}\varphi(x)||T^{(j)}x||^{-(n+1+p)}dS_{p^{(j)}}(K^{(j)})\\
    &=&\int_{{\mathbb{S}}^n}\varphi(x)dS_{p^{(\infty)}}(L^{(\infty)}).
  \end{eqnarray*}
This means that $dS_{p^{(\infty)}}(L^{(\infty)})$ is an orthonormal invariant Haar measure on $H_k\cap{\mathbb{S}}^n$, and hence equals to a multiple of $k-1$ dimensional Hausdorff measure $d{\mathcal{H}}^{k-1}|_{H_k\cap{\mathbb{S}}^n}$. Noting that $k\in[1,n]$ by our claim, this contradicts a non-existence result of Saroglou (\cite{Sa21}, Theorem 1.3).
\end{proof}

\begin{proof}[Proof of Theorem \ref{t1.5}]
  Now, Theorem \ref{t1.5} follows directly from Propositions \ref{p5.3} and \ref{p5.4}.
\end{proof}

\medskip

\section{Kernel Property and closedness of $\Gamma_\sigma$}

\noindent  From now on, we adopt a different idea to explore the topological structure of the uniqueness set by using the Kernel Property of the linearized equation \eqref{e4.14} of \eqref{e1.2}. Given a nontrivial solution $K\not={\S}^n$ of \eqref{e1.2}, the kernel of the linearized operator consists of solutions $\phi$ of the linearized equation \eqref{e4.14}. Suppose there exists a sequence $K_\varepsilon, \varepsilon\in{\R}$ tending to $K$, we will define the set of discrete differentiations (or call them tangential functions) of $K_\varepsilon$ by
      \begin{equation}\label{e6.1}
        \partial_\varepsilon|_{\varepsilon=0}h_\varepsilon:=\left\{\mbox{subconvergence limits of }\lim_{\varepsilon\to0}\phi_\varepsilon\right\}, \ \ \phi_\varepsilon:=\frac{h_\varepsilon-h_K}{||h_\varepsilon-h_K||_{L^2({\S}^n)}}.
      \end{equation}
 After subtracting the equation of $h_\varepsilon$ with the equation of $h_K$, one can see that $\phi_\varepsilon$ is a solution of the approximating linearized equation
    \begin{equation}\label{e6.2}
       \int^1_0U^{ij}_tdt(\phi_{\varepsilon,ij}+\phi_\varepsilon\delta_{ij})=(p-1)\int^1_0h_t^{p-2}dt\phi_\varepsilon
    \end{equation}
 for $h_t=th_\varepsilon+(1-t)h_K$. Using the elliptic estimates for equation \eqref{e6.2}, the sub-convergence in \eqref{e6.1} will be ensured with respect to any given sequence of $\varepsilon$. Passing to the limit $\varepsilon\to0$, one sees that each tangential function $\phi$ of $K_\varepsilon$ is a solution to the linearized equation \eqref{e4.14}.

 Denoting $E^{(*)}(K)$ the solution set of the linearized equation \eqref{e4.14}, the next crucial step in the proof of our uniqueness result is to classify this kernel completely. Since \eqref{e1.2} is invariant under rotations $A_\varepsilon\in SO(n+1)$, we know that $\lambda=1-p$ is eigenvalue of linearized eigen-problem of \eqref{e2.1}. Moreover, some of the corresponding eigenfunctions are given by the \textit{Rotational Tangential Set}
   $$
    T_O(K)=\left\{\phi_B(x)=\sum_{i,j=1}^{n+1}B_{ij}x_iD_jh_K\bigg| \ \forall B\in{\R}^{n\times n}, B+B^T=0\right\}
   $$
of $h_K$, where $h_K$ has been extended to homogeneous function
  $$
    H_K(X)=|X|h_K\left(\frac{X}{|X|}\right), X\in{\R}^{n+1}
  $$
of degree one and
   $$
    \frac{\partial H_K}{\partial X_i}=x_ih_K+\sum_{j=1}^{n+1}D_jh_K(\delta_{ij}-x_ix_j).
   $$
When $p=-n-1$, one knows that $K_\varepsilon=A_\varepsilon K, \varepsilon\in{\R}$ are also solutions of \eqref{e4.14} for affine transformations $A_\varepsilon\in SL(n+1)$ satisfying $\det(A_\varepsilon)=1, \forall\varepsilon$. Therefore, the functions in the \textit{Affine Tangential Set}
   $$
    T_A(K)=\left\{\phi_B(x)=(x^TBx)h_K+\sum_{j=1}^{n+1}(Bx)^t_jD_jh_K\bigg| \ \forall B\in{\R}^{n\times n}, tr(B)=0\right\}
   $$
are all solutions of \eqref{e4.14} when $p=-n-1$, where $(Bx)^t=Bx-(x^tB^Tx)x$ is the tangential component of $Bx$ with respect to ${\mathbb{S}}^n$. Now, define the following Kernel Property as
  \begin{equation}\label{e6.3}
    (KP^*) \hskip30pt E^{(*)}(K)=T_O(K), \ \ \forall K\in{\mathcal{C}}'_{n,p}={\mathcal{C}}_{n,p}\setminus\{{\S}^n\}.
  \end{equation}
It is remarkable that our Kernel Property \eqref{e6.3} is a two sides inclusion identity, where the strong inclusion $E^{(*)}(K)\subset T_O(K)$ will be shown in Subsection 6.3. Now, we will firstly show the following result in the next two subsections.

 \begin{prop}\label{p6.1}
  For dimension $n\geq2$ and given $\sigma\in(0,n+4]$, consider the linearized equation \eqref{e4.14} of \eqref{e1.2}. Supposing that the Kernel Property $(KP^*)$ \eqref{e6.3} holds for all $p\in(-n-1-\sigma,-n-1)\subset(-2n-5,-n-1)$, then $\Gamma_\sigma$ is relatively closed in $(-n-1-\sigma,-n-1)$.
 \end{prop}

\begin{rema}
   In the bibliographies on the Sobolev's inequality with remainder, a parallel version of the Kernel Property could be found in \cite{K89,M93,MS87} for Laplace equation, and could be found in \cite{FV14} for $s$-Laplace equation. The Kernel Property $(KP^*)$ \eqref{e6.3} is of independent interest for exploring the $L_p$-Minkowski problem \eqref{e1.2} in the future.
\end{rema}

\subsection{Perturbation lemmas and closedness of uniqueness set}

In this subsection, we always consider both the solution set and the Banach spaces modulo the equivalent relation $f(x)\sim f(Ax), A\in SO(n+1)$ of rotations $A$. Then, we will explore the closedness of $\Gamma_\sigma$ by characterizing the possible boundary $\overline{\Gamma}_\sigma\setminus\Gamma_\sigma$ of $\Gamma_\sigma$.

\begin{prop}\label{p6.2} For dimension $n\geq2$ and given $\sigma\in(0,n+4]$, consider \eqref{e1.2} in the range $p\in(-n-1-\sigma,-n-1)$. Assuming the uniqueness set $\Gamma$ is not closed and taking some $p_0\in\overline{\Gamma}_\sigma\setminus\Gamma_\sigma$, then the solution set ${\mathcal{C}}_{n,p_0}$ of \eqref{e1.2} satisfies the following properties:

(1) ${\mathcal{C}}_{n,p_0}$ is bounded in $C^{2,\alpha}({\mathbb{S}}^n)$.

(2) $K={\S}^n$ is an isolated point of ${\mathcal{C}}_{n,p_0}$.

(3) Each solution $K\not={{\S}^n}$ of ${\mathcal{C}}_{n,p_0}$ is not isolated after modulus equivalent relation.

(4) The counting number of ${\mathcal{C}}_{n,p_0}$ is infinite after modulus equivalent relation.
\end{prop}

We need two perturbation lemmas.

\begin{lemm}\label{l6.1}
 Letting $K_0$ be an isolated solution of
   \begin{equation}\label{e6.4}
    \det(\nabla^2h+hI)=fh^{p-1}, \ \ \forall x\in{\S}^n
   \end{equation}
  with respect to $p_0\in{\mathbb{R}}$ and positive function $f_0\in C^{\alpha}({\mathbb{S}}^n)$, there exist a small constant $\varepsilon_0>0$ and a relative large constant $\sigma_0>0$ depending only on $n, p_0, K_0$, such that for any pair $(p,f)$ satisfying
   \begin{equation}
     |p-p_0|<\varepsilon_0, \ \ ||f-f_0||_{C^{\alpha}({\mathbb{S}}^n)}<\varepsilon_0,
   \end{equation}
 the possible solution $h$ of \eqref{e6.4} with respect to $(p,f)$ satisfies
   \begin{equation}\label{e6.6}
     \mbox{either } ||h-h_{K_0}||_{C^{2,\alpha}({\mathbb{S}}^n)}<\varepsilon_0, \ \ \mbox{ or } ||h-h_{K_0}||_{C^{2,\alpha}({\mathbb{S}}^n)}>\sigma_0.
   \end{equation}
\end{lemm}

\begin{proof} By {\it a-priori} Theorem \ref{t1.5}, the conclusion of this lemma is clearly true since the limit of solutions is also a solution by the continuity of the Monge-Amp\`{e}re operator under $C^{2,\alpha}$-convergence. If \eqref{e6.6} does not hold true, a contradiction will occur with respect to the fact that $K_0$ is isolated.
\end{proof}

\begin{lemm}\label{l6.2}
 Letting $K_0$ be an isolated solution (after modulus rotations) of \eqref{e6.4} with respect to $p_0\in{\mathbb{R}}$ and positive function $f_0\in C^{\alpha}({\mathbb{S}}^n)$, there will be a small constant $\varepsilon_1>0$ such that \eqref{e6.4} admits a solution $h$ closing to $h_{K_0}$, with respect to any pair $(p,f)$ satisfying
   \begin{equation}
     |p-p_0|<\varepsilon_1, \ \ ||f-f_0||_{C^{\alpha}({\mathbb{S}}^n)}<\varepsilon_1.
   \end{equation}
\end{lemm}

The proof of this lemma was motivated by the work of Chen-Huang-Li-Liu \cite{CHLL20}.

\begin{proof} Taking $p_1$ close to $p_0$, we rewrite \eqref{e6.4} by
  \begin{equation}\label{e6.8}
    \det(\nabla^2h_{K_0}+h_{K_0}I)=f_1h_{K_0}^{p_1-1}, \ \ \forall x\in{\mathbb{S}}^n,
  \end{equation}
where $f_1\equiv f_0h_{K_0}^{p_0-p_1}$ is also close to $f_0$. Noting that the linearized equation of \eqref{e6.8} at $K_0$ is exactly the linearized eigen-problem \eqref{e2.1} for $\lambda=1-p_1$. By applying the Fredholm's alternative (Gilbarg-Trudinger, \cite{GTbook}, Section 5) for compact operator, the spectral set of \eqref{e2.1} is discrete. Since the spectrum is discrete without finite accumulation except infinity, one can choose $p_1$ closing sufficiently to $p_0$ such that the kernel of the linearized equation of \eqref{e6.8} is trivial. Considering a parameterized nonlinear operator
  $$
   {\mathcal{F}}_t(h)\equiv\det(\nabla^2h+hI)-f_th^{p_t-1}, \ \ t\in[0,1]
  $$
which is continuous mapping from positive $C^{2,\alpha}$ function to $C^\alpha$ function, for
   $$
    f_t\equiv tf+(1-t)f_1, \ \ p_t\equiv tp+(1-t)p_1,
   $$
we set a neighborhood of $K_0$ by
   $$
     {\mathcal{N}}(p_0,K_0)\equiv\{h\in C^{2,\alpha}({\mathbb{S}}^n)| ||h-h_{K_0}||<\varepsilon_{0}/2\},
    $$
where $\varepsilon_0$ is given in Lemma \ref{l6.1}. Noting that the zeros of ${\mathcal{F}}_t$ do not lie on the boundary of ${\mathcal{N}}(p_0,K_0)$ by Lemma \ref{l6.1}, it follows from the preservation property of topological degree \cite{Li89,Ni74} that
  \begin{equation}
     deg({\mathcal{F}}_t,{\mathcal{N}},0)=deg({\mathcal{F}}_0,{\mathcal{N}},0), \ \ \forall t\in[0,1].
  \end{equation}
By our choice of $p_1$ and noting that $K_0$ is the unique zero point of ${\mathcal{F}}_0$ in ${\mathcal{N}}(p_0,K_0)$ and the kernel of the linearized equation of \eqref{e6.8} is trivial, we have also $deg({\mathcal{F}}_0,{\mathcal{N}},0)\not=0$. Thus, the solvability of \eqref{e6.8} at the pair $(p,f)$ has been shown since no bifurcation at infinity occurs because of the a priori estimate Theorem \ref{t1.5}.
\end{proof}

\begin{proof}[Proof of Proposition \ref{p6.2}] Part (1) is the consequence of {\it a-priori} estimate Theorem \ref{t1.5} since all solutions remain in a fixed bounded subset of $C^{2,\alpha}$ space. Noting that $\lambda=1-p_0\in(n+2,2n+6)$ is not an eigenvalue of linearized eigen-problem \eqref{e2.1} at $K={{\S}^n}$, we know that the kernel of the linearized equation \eqref{e4.14} at the unit sphere $K={\S}^n$ is trivial. Supposing that there exists a sequence of nontrivial solutions $K_j\not={\S}^n, K_j\in{\mathcal{C}}_{n,p_0},j\in{\mathbb{N}}$ of \eqref{e1.2} tends to the unit sphere $K_\infty={\S}^n$, subtracting the equation of $K_j$ with the equation of $K_\infty={\S}^n$ and then passing to the limit, we conclude the linearized equation
   \begin{equation}
    -\Delta\phi-n\phi=(1-p_0)\phi, \ \ \forall x\in{\mathbb{S}}^n
   \end{equation}
have a nontrivial solution $\phi$. Contradiction holds for $p_0\in(-2n-5,-n-1)$. So, we conclude that Part (2) holds true. On the other hand, since the solutions of \eqref{e1.2} are not unique at $p=p_0$, Part (4) is a direct consequence of Part (3). We remain to show that Part (3) holds true. Otherwise, if \eqref{e1.2} admits a nontrivial isolated solution $K\not={\S}^n$ for $p_0$ after modulo rotations, by Lemma \ref{l6.2}, \eqref{e1.2} will admit a nontrivial solution $K'$ when $p\in\Gamma_\sigma$ and close to $p_0$. This contradicts with our definition of $\Gamma_\sigma$, and thus completes the proof of the proposition.
\end{proof}

\subsection{The method of integrating flow}

Considering \eqref{e1.2} for $p\in\overline{\Gamma}_\sigma\setminus\Gamma_\sigma$, by Proposition \ref{p6.2}, the solution set ${\mathcal{C}}'_{n,p}$ is an infinite set and contains no isolated solution modulo rotations. In this subsection, we will show an opposite phenomenon when assuming the Kernel Property $(KP^*)$ \eqref{e6.3}. Combining these two opposite phenomena will yield a conclusion that $\overline{\Gamma}_\sigma\setminus\Gamma_\sigma=\emptyset$. This means exactly that the uniqueness set $\Gamma_\sigma$ is relatively closed in $(-n-1-\sigma,-n-1)$.

 Here, we always consider the solution set and Banach spaces modulo rotations. By Proposition \ref{p6.2}, the nontrivial solution set ${\mathcal{C}}'_{n,p}={\mathcal{C}}_{n,p}\setminus\{{\S}^n\}$ is a Perfect Set after modulo rotations, since it contains no isolated point. Given two points $K, K'\in{\mathcal{C}}'_{n,p}$, define the modulus distance by
    $$
     d(K,K'):=\inf_{A\in SO(n+1)}||h_K-h_{AK'}||_{L^2({\S}^n)}.
    $$
 It is remarkable that $SO(n+1)$ is a compact group, so the infimum in the definition of $d(K,K')$ is attainable. Now, we will prove the following lemma.

\begin{lemm}\label{l6.3}
  Given a solution $K_0\in{\mathcal{C}}'_{n,p}$ of \eqref{e1.2}, there exists a positive constant $\delta_0>0$, such that for any solution $K\in{\mathcal{C}}'_{n,p}$ of \eqref{e1.2} located in the $\delta_0$-neighborhood
      $$
       {\mathcal{N}}_{\delta_0}:=\{K'|\ d(K_0,K')<\delta_0\}
      $$
  of $K_0$, there exists $A\in SO(n+1)$ such that $AK=K_0$.
\end{lemm}

\noindent  Next, let us adapt the idea of integrating flow from Differential Geometry and Lie Groups to give a proof of Lemma \ref{l6.3}. Given a tangential field $\xi\in TM$ on a manifold $M$, finding an orbit $\gamma(\varepsilon)\in M, \varepsilon\in{\R}$ (or integrating curves) from $\xi$ is reduced to solving the flow ODE
   \begin{equation}
   \begin{cases}
    \displaystyle \frac{d}{d\varepsilon}\gamma(\varepsilon)=\xi(\gamma(\varepsilon)), & \forall\varepsilon,\\
     \gamma(0)=\gamma_0.
   \end{cases}
   \end{equation}
See for example the book by Peter Olver \cite{O00}. Now, let $K_0\in{\mathcal{C}}'_{n,p}$ be a non-isolated solution, there must be a sequence of solutions $K_j\in{\mathcal{C}}'_{n,p}, K_j\not=K_0\in{\mathbb{N}}$ tending to $K_0$ as $j\to\infty$, under the modulus distance. From the definition of the modulus distance $d(\cdot,\cdot)$, after modulo rotations $A_j, j\in{\mathbb{N}}$, one may assume
   \begin{equation}\label{e6.12}
    \begin{cases}
     ||h_{K_{\varepsilon_j}}-h_{K_0}||_{L^2({\S}^n)}=d(K_{\varepsilon_j},K_0),\\
      K_{\varepsilon_j}:=A_jK_j, \forall j\in{\mathbb{N}}
    \end{cases}
   \end{equation}
after introducing the discrete parameters $\varepsilon_j=d(K_{\varepsilon_j},K_0)$. It is clear that $K_{\varepsilon_j}, j\in{\mathbb{N}}$ is also a sequence of approximation solutions. Applying the elliptic estimate for difference equation of $h_0:=h_{K_0}$ with $h_{\varepsilon_j}:=h_{K_{\varepsilon_j}}$, it yields that the sequence
  $$
   \phi_{\varepsilon_j}=\frac{h_{\varepsilon_j}-h_{0}}{||h_{\varepsilon_j}-h_{0}||_{L^2({\S}^n)}}=\frac{h_{\varepsilon_j}-h_{0}}{\varepsilon_j}
  $$
subconverges to a limiting function $\phi$ in $C^{2,\alpha}({\S}^n)$.

As defined above, this function is one of the discrete differentiations of $h_\varepsilon, \varepsilon=\varepsilon_j$ at $\varepsilon=0$ and thus satisfies $\phi\in\partial_\varepsilon|_{\varepsilon=0} h_{K_\varepsilon}\in E^{(*)}(K)$. By Kernel Property $(KP^*)$ \eqref{e6.3}, and our definition of the parameter $\varepsilon=\varepsilon_j, j\in{\mathbb{N}}$, one has
   \begin{eqnarray}\nonumber\label{e6.13}
     \phi&=&\partial_\varepsilon\big|_{\varepsilon=0}h_{\varepsilon}=\lim_{j\to\infty}\frac{h_{\varepsilon_j}-h_0}{\varepsilon_j}\\
      &=&\nabla_\xi h_K=\sum_{i,j}x_iB_{ij}D_jh_K, \ \ \xi=Bx
   \end{eqnarray}
holds for some antisymmetric matrix $B\in{\R}^{n\times n}$. That is to say, the tangent spaces $E^{(*)}(K)$ of the manifold ${\mathcal{C}}'_{n,p}$ at the points $K\in{\mathcal{C}}'_{n,p}$ are all of the form $\sum_{i,j}x_iB_{ij}D_jh_K$ for antisymmetric matrices $B$ satisfying $B+B^T=0$.

 Now, let us give the proof of Lemma \ref{l6.3} as follows.

\begin{proof}
   Given $K_0\in{\mathcal{C}}'_{n,p}$, if the conclusion of Lemma \ref{l6.3} is not true, there exists a sequence of approximation solutions $K_j\in{\mathcal{C}}'_{n,p}, K_j\not=K_0, \forall j\in{\mathbb{N}}$ tending to $K_0$. Introducing the discrete parameter $\varepsilon_j$ as above satisfying \eqref{e6.12}, there holds the discrete O.D.E. \eqref{e6.13} as well. Note that the solution of the O.D.E.
    \begin{equation}\label{e6.14}
      \begin{cases}
       \displaystyle \frac{d}{d\varepsilon}h_\varepsilon=\nabla_\xi h_\varepsilon, & \forall\varepsilon\in(0,1),\\
        h_0=h_K
      \end{cases}
    \end{equation}
is given by
   \begin{equation}
    K'_\varepsilon=A_\varepsilon K, \ A_\varepsilon=e^{\int_0^\varepsilon B_{\varepsilon'} d\varepsilon'}\in SO(n+1), \ \ \forall\varepsilon\in\tau(K,K').
   \end{equation}
Actually, this follows from the standard Lie group exponential map. After comparing \eqref{e6.14} with the discrete O.D.E. \eqref{e6.13}, we arrive at
    \begin{equation}\label{e6.16}
      \frac{||h_{K_{\varepsilon_j}}-h_{A_{\varepsilon_j}K_0}||_{L^2({\S}^n)}}{\varepsilon_j}<1/2, \ \ \mbox{ holds for large } j
    \end{equation}
by Taylor's expansion and the fact that the orbit map $\varepsilon\to h_{A_\varepsilon K_0}$ is $C^1$ as a map into $C^{2,\alpha}({\mathbb{S}}^n)$. Owing to our normalized assumption \eqref{e6.12} and normalized parameters $\varepsilon_j=d(K_{\varepsilon_j},K_0)$, it yields from \eqref{e6.16} that
   $$
    d(K_{\varepsilon_j},K_0)\leq d(K_{\varepsilon_j},K_0)/2, \ \ \mbox{ for all large } j
   $$
and thus $K_{\varepsilon_j}=K_0$ for large $j$. Contradicting with our assumption $K_{\varepsilon_j}\not=K_0, \forall j$, the proof of Lemma \ref{l6.3} has been proven.
\end{proof}

So, we reach the following proposition assuming the Kernel Property.

\begin{prop}\label{p6.3}
 For dimension $n\geq2$ and $\sigma\in(0,n+4]$, if the Kernel Property $(KP^*)$ \eqref{e6.3} holds for all $p\in(-n-1-\sigma,-n-1)$, then the solution set ${\mathcal{C}}'_{n,p}$ contains only isolated nontrivial solutions after modulo the rotations.
\end{prop}

\begin{proof}[Proof of Proposition \ref{p6.1}]
  Now, Proposition \ref{p6.1} follows from Proposition \ref{p6.2} and Proposition \ref{p6.3}.
\end{proof}

Next, in order showing the closedness of the uniqueness set $\Gamma$, we remain to verify the validity of the Kernel Property $(KP^*)$ \eqref{e6.3} for $p<-n-1$.

\subsection{Validity of the Kernel Property for $p<-n-1$}

We will prove the following result.

 \begin{prop}\label{p6.4}
  Letting $n\geq2$, the Kernel Property $(KP^*)$ \eqref{e6.3} holds for all
     $$
       p\not=n+1, 1, -n-1.
     $$
 \end{prop}

 \begin{proof}
   As mentioned above, each functions $\phi\in T_O(K)$ in the rotational tangential set belong all to the kernel $E^{(*)}(K)$ of the linearized equation \eqref{e4.14}. It remains to show the converse inclusion. By Lemma \ref{l4.2}, each solution $\phi\not\equiv0$ of the linearized equation \eqref{e4.14} corresponds to an assigned infinitesimal generator field $\upsilon$ of \eqref{e4.11}. Another hand, it is inferred from Proposition \ref{p4.1} that each infinitesimal generator field $\upsilon=\xi^i\frac{\partial}{\partial y_i}+\zeta\frac{\partial}{\partial u}$ of \eqref{e4.11} is given by \eqref{e4.31}. Using the relations
      \begin{equation}\label{e6.22}
       \begin{cases}
         \displaystyle x_\varepsilon=T^*(y_\varepsilon)=\left(\frac{y_\varepsilon}{\sqrt{1+|y_\varepsilon|^2}},-\frac{1}{\sqrt{1+|y_\varepsilon|^2}}\right)\in{\S}^n,\\
        \displaystyle  h_K(x)=\frac{u_K(y)}{\sqrt{1+|y|^2}},\ \ h_\varepsilon(x_\varepsilon)=\frac{u_\varepsilon(y_\varepsilon)}{\sqrt{1+|y_\varepsilon|^2}}
       \end{cases}
      \end{equation}
  for solution sequence $K_\varepsilon, \varepsilon\in{\R}$ (or support function $h_\varepsilon, \varepsilon\in{\R}$) of the approximation equation \eqref{e4.12} satisfying
     \begin{equation}
       h_0(x)=h_K(x), \ \ \frac{d}{d\varepsilon}\bigg|_{\varepsilon=0}h_\varepsilon(x)=\phi(x), \ \ x\in{\S}^n,
     \end{equation}
   we will classify all tangential solutions $\phi$ of the linearized equation \eqref{e4.14} as below. Actually, by definition of infinitesimal generator field,
      \begin{equation}\label{e6.24}
        \begin{cases}
         \displaystyle  \frac{d}{d\varepsilon}\bigg|_{\varepsilon=0}y_\varepsilon=\xi,\\[10pt]
         \displaystyle \frac{d}{d\varepsilon}\bigg|_{\varepsilon=0}u_\varepsilon=\zeta,
        \end{cases}
      \end{equation}
    where $\xi, \zeta$ are given \eqref{e4.31}. Therefore, we have the resolution formula
      \begin{eqnarray}\nonumber\label{e6.25}
       \phi&=&\frac{d}{d\varepsilon}\bigg|_{\varepsilon=0}h_\varepsilon(x)=\frac{d}{d\varepsilon}\bigg|_{\varepsilon=0}h_\varepsilon(x_\varepsilon^{-1})\\
       &=&\frac{d}{d\varepsilon}\bigg|_{\varepsilon=0}\frac{u_\varepsilon(y_\varepsilon^{-1})}{\sqrt{1+|y|^2}}\\ \nonumber
         &=&\frac{\zeta-\xi\cdot Du_K}{\sqrt{1+|y|^2}},
      \end{eqnarray}
    where $x_\varepsilon^{-1}, y_\varepsilon^{-1}$ are the inverses of the mappings $x\to x_\varepsilon$ and $y\to y_\varepsilon$ respectively. Now, we separate the classified infinitesimal generator $\upsilon$ into two cases:

  \noindent\it{Case 1:} $\xi^i=E_{ij}y_j, \zeta=0$: Using the relation \eqref{e4.31}, it follows from \eqref{e6.25} that
     \begin{equation}\label{e6.26}
      \frac{d}{d\varepsilon}\bigg|_{\varepsilon=0}h_\varepsilon(x_\varepsilon)=-\frac{\sum_{i,j=1}^ny_jE_{ij}D_iu_K}{\sqrt{1+|y|^2}}.
     \end{equation}
  Noting that the matrix $E$ is anti-symmetric, if one denotes another anti-symmetric matrix $B$ defined by $B=\left[
       \begin{array}{cc}
         -E & 0 \\
         0 & 0
       \end{array}
       \right]$, we deduce from
    \begin{eqnarray}\nonumber\label{e6.27}
     D_iu_K&=&\sum_{k=1}^n\sqrt{1+|y|^2}D_kh_K\left[\frac{\delta_{ik}}{\sqrt{1+|y|^2}}-\frac{y_iy_k}{(1+|y|^2)^{\frac{3}{2}}}\right]\\
           &&+\sqrt{1+|y|^2}D_{n+1}h_K\frac{y_i}{(1+|y|^2)^{\frac{3}{2}}}+\frac{y_i}{\sqrt{1+|y|^2}}h_K
    \end{eqnarray}
  that
    \begin{eqnarray}\nonumber\label{e6.28}
      \sum_{i,j=1}^ny_jE_{ij}D_iu_K&=&\sum_{i,j=1}^ny_jE_{ij}D_jh_K\\
          &=&\sqrt{1+|y|^2}\sum_{i,j=1}^{n+1}x_iB_{ij}D_jh_K.
    \end{eqnarray}
  Substituting into \eqref{e6.26}, we finally reach that
   \begin{eqnarray}\nonumber\label{e6.29}
      \phi&=&\frac{d}{d\varepsilon}\bigg|_{\varepsilon=0}h_\varepsilon(x_\varepsilon)\\
        &=&\sum_{i,j=1}^{n+1}x_iB_{ij}D_jh_K=\nabla_\xi h_K.
   \end{eqnarray}
  This is exactly the desired conclusion $E^{(*)}(K)\subset T_O(K)$.

 \noindent\it{Case 2:} $\xi^i=\sum_{j=1}^nC_jy_jy_i+C_i, \zeta=\sum_{j=1}^nC_jy_ju_K$: Similar calculation brings us that
   \begin{eqnarray*}
    \phi&=&\frac{d}{d\varepsilon}\bigg|_{\varepsilon=0}h_\varepsilon(x_\varepsilon)\\
       &=&\frac{\sum_{j=1}^nC_jy_ju_K-\sum_{i=1}^n(\sum_{j=1}^nC_jy_jy_i+C_i)D_iu_K}{\sqrt{1+|y|^2}}.
   \end{eqnarray*}
 Combining with \eqref{e6.27}, we derive that
   \begin{eqnarray*}
     \phi&=&-\sum_{j=1}^n\left(\frac{C_jD_jh_K}{\sqrt{1+|y|^2}}+D_{n+1}h_K\frac{C_jy_j}{\sqrt{1+|y|^2}}\right)\\
      &=&-\sum_{j=1}^n\left(-x_{n+1}C_jD_jh_K+x_jC_jD_{n+1}h_K\right)\\
      &=&-\sum_{i,j=1}^{n+1}x_iB_{ij}D_jh_K,
   \end{eqnarray*}
 where $B$ is an anti-symmetric matrix defined by
   $$
    B=\left[
      \begin{array}{ccccc}
        0 & 0 &  \ldots & 0 & C_1\\
        0 & 0 & \ldots & 0  & C_2\\
         \vdots& \vdots & \vdots& \vdots & \vdots\\
         0 & 0 & \ldots & 0 & C_n\\
        -C_1 & -C_2 & \ldots & -C_n & 0
       \end{array}
      \right].
   $$
The desired conclusion $E^{(*)}(K)\subset T_O(K)$ has been shown, since every anti-symmetric matrix generates a one-parameter subgroup of $SO(n+1)$, the above vector field is precisely an infinitesimal rotational deformation. \end{proof}

\begin{proof}[Complete the proof of Theorem \ref{t1.4}.]
  After combining Proposition \ref{p5.1}, Proposition \ref{p6.1} and Proposition \ref{p6.4}, the conclusion of Theorem \ref{t1.4} has been shown.
\end{proof}

When one focuses on group ${\mathcal{S}}(T)$ invariant solutions $K\in{\mathcal{C}}_{n,p}(T)$ for regular polytope $T\subset{\R}^{n+1}$, a mild modification as in the proof of Theorem \ref{t1.4} shows the following topological result for the group invariant uniqueness set.

\begin{theo}\label{t6.1}
  For dimensions $n\geq2$ and given regular polytope $T\subset{\R}^{n+1}$, the group ${\mathcal{S}}(T)$ invariant uniqueness set $\Gamma(T)=\{p\in(-2n-5,-n-1)| {\mathcal{C}}_{n,p}(T)={\S}^n\}$ is relatively open and closed in the full interval $(-2n-5,-n-1)$.
\end{theo}

\medskip

\section*{Acknowledgments}

The author would like to express his deepest gratitude to Professors Xi-Ping Zhu, Kai-Seng Chou, Xu-Jia Wang and Neil S. Trudinger for their constant encouragements and warm-hearted helps. Special thanks are also owed to Professor Xu-Jia Wang for bringing him to the topic and giving him many key observations. Without his help, the author could not have completed the current exploration. This paper is also dedicated to the memory of Professor Dong-Gao Deng.

\medskip

\noindent  The Department of Mathematics, Shantou University, Shantou 515063, P. R. China.

\noindent \textit{E-mail address:} \href{mailto:szdu@stu.edu.cn}{szdu@stu.edu.cn}


\medskip

\begin{thebibliography}{10}


\bibitem{Al42} A.D. Aleksandrov, {\it Existence and uniqueness of a convex surface with a given integral curvature}, C. R. (Doklady) Acad. Sci. URSS (N.S.), {\bf35} (1942), 131-134.

\bibitem{Al38} A.D. Aleksandrov, {\it On the theory of mixed volumes. IV: Mixed discriminants and mixed volumes}, Mat. Sb., {\bf3} (1938), 227-249.

\bibitem{An00}
B. Andrews, {\it Motion of hypersurfaces by gauss curvature}, Pacific Journal of Mathematics, {\bf195} (2000), 1-34.

\bibitem{An03}
B. Andrews, {\it Classification of limiting shapes for isotropic curve flows}, J. Amer. Math. Soc., {\bf16} (2003), 443-459.

\bibitem{AGN16} B. Andrews, P. Guan and L. Ni, {\it Flow by powers of the Gauss curvature}, adv. Math., {\bf299} (2016), 174-201.

\bibitem{Bo10} K.J. B\"{o}r\"{o}czky, {\it Stability of the Blaschke-Santal\'{o} and the affine isoperimetric inequality}, Adv. Math., {\bf225} (2010), 1914-1928.

\bibitem{B75} W.M. Boothby, {\it An introduction to Differentiable Manifolds and Riemannian Geometry}, Academic Press, New York, 1975.

\bibitem{BE91} G. Bianchi and H. Egnell, {\it A note on the Sobolev inequality}, J. Funct. Anal., {\bf100} (1991), 18-24.

\bibitem{BH10}, K.J. B\"{o}r\"{o}czky and D. Hug, {\it Stability of the reverse Blaschke-Santal\'{o} inequality for zonoids and applications}, Adv. in Appl. Math., {\bf44} (2010), 309-328.

\bibitem{BS23} K.J. B\"{o}r\"{o}czky and C. Saroglou, {\it Uniqueness when the $L_p$ curvature is close to be a constant for $p\in[0,1)$}, arXiv:2308.03367v2.

\bibitem{BLYZ12} L.J. B\"or\"oczky,  E. Lutwak, D. Yang and G. Zhang, {\it The log-Brunn-Minkowski inequality}, Adv. Math., {\bf231} (2012), 1974-1997.

\bibitem{BLYZ13} L.J. B\"or\"oczky,  E. Lutwak, D. Yang and G. Zhang, {\it The logarithmic Minkowski problem}, J. Amer. Math. Soc., {\bf26} (2013),  831-852.

\bibitem{BBCY19} G. Bianchi, K. B\"or\"oczky, A. Colesanti, D. Yang, {\it
The $L_p$-Minkowski problem for $-n < p < 1$},
Adv. Math., {\bf 341} (2019), 493-535.

\bibitem{BLYZ13} K. B\"or\"oczky, E. Lutwak, D. Yang, G. Zhang, {\it
The logarithmic Minkowski problem},
J. Amer. Math. Soc. {\bf 26} (2013), 831-852.

\bibitem{BCD17} S. Brendle, K. Choi, P. Daskalopoulos, {\it
Asymptotic behavior of flows by powers of the Gaussian curvature},
Acta Math, {\bf 219} (2017), 1-16.

\bibitem{BIS19} P. Bryan, M. Ivaki, J. Scheuer, {\it
A unified flow approach to smooth, even $L_p$-Minkowski problems},
Analysis \& PDE, {\bf 12} (2019), 259-280.

\bibitem{Ca72} E. Calabi, {\it Complete affine hyperspheres. I, in Symposia Mathematica, Vol. X (Convegno di Geometria Differenziale}, INdAM, Rom, (1972), 19-38.

\bibitem{Ca90}
L.A. Caffarelli,
{\it A localization property of viscosity solutions to the monge-ampere
  equation and their strict convexity}, Annals of mathematics, {\bf131} (1990), 129-134.

\bibitem{Ca902}
L.A. Caffarelli, {\it Interior $w^{2,p}$ estimates for solutions of the monge-ampere equation}, Annals of Mathematics, (1990), 135-150.

\bibitem{Ch85} B. Chow, {\it Deforming convex hypersurfaces by the $n$th root of the Gaussian curvature}, J. Differ. Geom., {\bf22} (1985), 117-138.

\bibitem{Ch17} K. Choi, {\it The Gauss Curvature Flow: Regularity and Asymptotic Behavior}, Ph.D. Thesis, Columbia University, New York, NY, 2017.

\bibitem{CCL21} H.D. Chen, S.B. Chen and Q.R. Li, {\it Variations of a class of Monge-Amp\`{e}re-type functionals and their applications}, Analysis \& PDE, {\bf14} (2021), 689-716.

\bibitem{CD25} H.J. Chen and S.Z. Du, {\it Complete classification of the symmetry group of $L_p$-Minkowski problem on the sphere}, J. Differential Equations, {\bf444} (2025), 25pp.

\bibitem{CHZ19}  C. Chen,  Y.  Huang, Y. Zhao, {\it
Smooth solutions to the $L_p$ dual Minkowski problem},
Math. Ann. {\bf 373} (2019), 953-976.

\bibitem{CHLL20} S.B. Chen, Y. Huang, Q.R. Li, J.K. Liu, {\it The $L_p$-Brunn-Minkowski inequality for $p<1$}, Adv. Math., {\bf368} (2020), 107166.

\bibitem{CLZ19} S. Chen, Q.R. Li, G. Zhu, {\it
The logarithmic Minkowski problem for non-symmetric measures},
Trans. Amer. Math. Soc. {\bf 371} (2019), 2623-2641.

\bibitem{Co73}
H.S.M. Coxeter, {\it Regular polytopes}, Courier Corporation, 1973.

\bibitem{Co97}
P.R. Cromwell, {\it Polyhedra}, Cambridge University Press, 1997.

\bibitem{CW06} K.S. Chou, X.J. Wang, {\it The $L_p$-Minkowski problem and the Minkowski problem in centroaffine geometry}, Adv. Math. {\bf 205} (2006), 33-83.

\bibitem{CW98} G. Cerami and J. Wei, {\it Multiplicity of multiple interior peak solutions for some singularly perturbed Neumann problem}, IMRN, 1998.

\bibitem{Du21} S.Z. Du, {\it On the planar $L_ p$-Minkowski problem},
J. Diff. Equations, {\bf287} (2021), 37-77.

\bibitem{DLL24} S.Z. Du, Y.N. Liu and J. Lu, {\it Nonuniqueness and nonexistence results for the $L_p$-dual Minkowski problem with supercritical exponents}, arXiv: 2104.07426.

\bibitem{DWZ24} S.Z. Du, X.J. Wang and B.C. Zhu, {\it
Limiting shape of the $L_p$-dual Minkowski problem}, preprint.

\bibitem{DG94} C. Dohmen and Y. Giga, {\it Selfsimilar shrinking curves for anisotropic curvature flow equations}, Proc. Jpn. Acad. Ser. A, {\bf70} (1994), 252-255.

\bibitem{F74} W. Firey, {\it Shapes of worn stones}, Mathematika, {\bf21} (1974), 1-11.

\bibitem{Fu15} N. Fusco, {\it The quantitative isoperimetric inequality and related topics}, Bull. Math. Sci., {\bf 5} (2015), 517-607.

\bibitem{FZ22-1} A. Figalli and Y. Zhang, {\it Sharp gradient stability for the Sobolev inequality}, Duke Math. J., {\bf171} (2022), 2407-2459.

\bibitem{FN19} A. Figalli and R. Neumayer, {\it Gradient stability for the Sobolev inequality: the case $p\geq2$}, J. Eur. Math. Soc., {\bf21} (2019), 319-354.

\bibitem{FV14} M.M. Fall and E. Valdinoci, {\it Uniqueness and nondegeneracy of positive solutions of $(-\Delta)^su+u=u^p$ in ${\R}^N$ when $s$ is close to $1$}, Commun. Math. Phys., {\bf329} (2014), 383-404.

\bibitem{FZ22-2} A. Figalli and Y. Zhang, {\it Strong stability for the Wulff inequality with a crystalline norm}, Comm. Pure Appl. Math., {\bf75} (2022), 422-446.

\bibitem{G93} M.E. Gage, {\it Evolving plane curvaes by curvature in relative geometries}, Duke Math. J., {\bf72} (1993), 441-466.

\bibitem{GG02}
B. Guan and P. Guan, {\it Convex hypersurfaces of prescribed curvatures}, Ann. of Math.(2), {\b156} (2002), 655-673.

\bibitem{GTbook}
D. Gilbarg and N.S. Trudinger, {\t Elliptic Partial Differential Equations of Second Order}, Springer-Verlag, Berlin, 2001, xiv+517pp.

\bibitem{GW99} C. Gui and J. Wei, {\it Multiple interior peak solutions for some singularly perturbation Neumann problems}, J. Diff. Equs., {\bf158} (1999), 1-27.

\bibitem{HLX15} Y. Huang, J.K. Liu and L. Xu, {\it On the uniqueness of $L_p$-Minkowski problems: the constant $p$-curvature case in ${\mathbb{R}}^3$}, Adv. Math., {\bf281} (2015), 906-927.

\bibitem{HLW16} Y. He, Q.R. Li, X.J. Wang, {\it
Multiple solutions of the $L_p$-Minkowski problem},
Calc. Var. PDEs, {\bf 55} (2016), 13 pp.

\bibitem{HLYZ16}
Y. Huang, E. Lutwak, D. Yang, Deane, G. Zhang, {\it Geometric measures in the dual Brunn-Minkowski theory and their associated Minkowski problems},
Acta Math., {\bf216} (2016), 325-388.

\bibitem{Iv15}, M. Ivaki, {\it Stability of the Blaschke-Santal\'{o} inequality in the plane}, Monatsh. Math., {\bf177} (2015), 451-459.

\bibitem{Iv16} M. Ivaki, {\it
Deforming a hypersurface by Gauss curvature and support function},
J. Funct. Anal., {\bf 271} (2016), 2133-2165.

\bibitem{IM23} M.N. Ivaki and E. Milman, {\it Uniqueness of solutions to a class of isotropic curvature problems}, Adv. Math., {\bf435} (2023), 1-11.

\bibitem{JLZ16} H. Jian, J. Lu, G.  Zhu, {\it
Mirror symmetric solutions to the centro-affine Minkowski problem},
Calc. Var. PDEs, {\bf 55} (2016), 22 pp.

\bibitem{JLW15} H. Jian, J. Lu, X.J. Wang, {\it
Nonuniqueness of solutions to the $L_p$-Minkowski problem},
Adv. Math. {\bf 281} (2015), 845-856.

\bibitem{JLW18} H. Jian, J. Lu, X.J. Wang, {\it
A priori estimates and existence of solutions to the prescribed centroaffine curvature problem}, J. Funct. Anal. {\bf 274} (2018), 826-862.

\bibitem{JWW11}  M. Jiang, L. Wang,  J. Wei, {\it $2\pi$-periodic self-similar solutions for the anisotropic affine curve shortening problem}, Calc. Var. PDEs, {\bf 41} (2011),  535-565.


\bibitem{K89} M.K. Kwong, {\it Uniqueness of positive solutions of $\Delta u-u+u^p=0$ in ${\R}^n$}, Arch. Rational Mech. Anal., {\bf105} (3) (1989), 243-266.

\bibitem{K98} J.L. Kazdan, {\it
Solving equations, an elegant legacy. The American mathematical monthly}, {\bf105} (1998), 1-21.

\bibitem{KM22} A.V. Kolesnikov and E. Milman, {\it Local $L_p$-Brunn-Minkowski inequalities for $p<1$}, Memoirs of the AMS, {\bf1360} (2022), 78pp.

\bibitem{Li89} Y.Y. Li:
Degree theory for second order nonlinear elliptic operators and its applications, Commun. Partial Differ. Equ., {\bf14} (1989), 1541-1578.

\bibitem{Li19}   Q.R. Li, {\it
Infinitely many solutions for centro-affine Minkowski problem},
Int. Math. Res. Not., {\bf18} (2019), 5577-5596.

\bibitem{LYZ04} E. Lutwak, D. Yang and G. Zhang, {\it On the $L_p$-Minkowski problem}, Trans. Am. Math. Soc., {\bf356} (2004), 4359-4370.

\bibitem{LYZ18} E. Lutwak, D. Yang and G. Zhang, {\it $L_p$ dual curvature measures}, Adv. Math., {\bf329} (2018), 85-132.

\bibitem{LSW20} Q.R. Li, W. Sheng, X.J. Wang, {\it Flow by Gauss curvature to the Aleksandrov and dual Minkowski problems},
J. Eur. Math. Soc., {\bf 22} (2020), 893-923.

\bibitem{Lu93} E. Lutwak, {\it The Bunn-Minkowski-Firey theory. I. Mixed volumes and the Minkowski problem}, J. Differ. Geom., {\bf38} (1993), 131-150.

\bibitem{LYZ02}
E. Lutwak, D. Yang, and G. Zhang, {\it Sharp affine $L_p$ sobolev inequalities}, Journal of Differential Geometry, {\bf62} (2002), 17-38.

\bibitem{LYZ18}
E. Lutwak, D. Yang, and G. Zhang, {\it $L_p$ dual curvature measures}, Adv. Math., {\bf329} (2018), 85-132.

\bibitem{Lu93} E. Lutwak, {\it
The Brunn-Minkowski-Firey theory I. Mixed volumes and the Minkowski problem}, J. Diff. Geom. {\bf 38} (1993), 131-150.

\bibitem{LO95} E. Lutwak, V. Oliker, {\it
On the regularity of solutions to a generalization of the Minkowski problem}, J. Diff. Geom., {\bf 41} (1995), 227-246.

\bibitem{LW13} J. Lu, X.J. Wang, {\it
Rotationally symmetric solutions to the $L_p$-Minkowski problem},
J. Differential Equations, {\bf 254} (2013), 983-1005.

\bibitem{LYZ02} E. Lutwak, D. Yang, G. Zhang, {\it
Sharp affine $L_p$ Sobolev inequalities},
J. Diff. Geom., {\bf 62} (2002), 17-38.

\bibitem{LYZ18} E. Lutwak, D. Yang, G. Zhang, {\it
$L_p$ dual curvature measures},
Adv. Math. {\bf 329} (2018), 85-132.

\bibitem{M93} K. Mcleod, {\it Uniqueness of positive radial solutions of $\Delta u+f(u)=0$ in ${\R}^n$, II}, Trans. Am. Maht. Soc., {\bf 339} (2) (1993), 495-505.

\bibitem{Ma39}
K. Mahler, {\it Ein minimalproblem fur konvexe polygone}, Mathematica, {\bf7} (1939), 118-127.

\bibitem{Ma30}
K. Mahler, {\it Ein ubertragungsprinzip fur konvexe korper}, Casopis Pest. Mat. Fys., {\bf68} (1930), 93-102.

\bibitem{Mi23} E. Milman, {\it
Centro-Affine Differential Geometry and the Log-Minkowski Problem},
J. Euro. Math. Soc., 2023.

\bibitem{MS87} K. McLeod and J. Serrin, {\it Uniqueness of positive radial solutions of $\Delta u+f(u)=0$ in ${\R}^n$}, Arch. Rational Mech. Anal., {\bf 99} (2) (1987), 115-145.

\bibitem{Ni74} L. Nirenberg, {\it Topics in nonlinera functional analysis}, Lecture Notes, 1973-1974. Courant Institute of Mathematical Sciences, New York University, New York, 1974, viii+259 pp.

\bibitem{O00} P. Olver, {\it Applications of Lie Groups to Differential Equations}, Graduate Texts in Mathematics, {\bf107}, New York, NY: Springer. xxviii, 513p. (2000).

\bibitem{Sa21} C. Saroglou, {\it A non-existence result for the $L_p$-Minkowski problem}, arXiv:2109.06545.

\bibitem{Sc14} R. Schneider, {\it Convex bodies: the Brunn-Minkowski theory}, second expanded edition,
Encyclopedia of Mathematics and its Applications, 151, Cambridge University Press, Cambridge, 2014.

\bibitem{Se02} L. Serge, {\it Algebra}, Grad. Texts in Math., {\bf211}, Springer-Verlag, New York, 2002, xvi+914 pp.

\bibitem{So20}
D.M. Sommerville, {\it Introduction to the Geometry of N Dimensions}, Courier Dover Publications, 2020.

\bibitem{Ya06} H. Yagisita, {\it Non-uniqueness of self-similar shrinking curves for an anisotropic curvature flow}, Calc. Var. Partia Differ. Equ., {\bf26} (2006), 49-55.\\

\end{thebibliography}
\end{document}